\documentclass[11pt,a4paper]{article}

\usepackage{amsmath,amsfonts,amssymb,amsthm}
\usepackage[all]{xy}
\evensidemargin=\oddsidemargin
\numberwithin{equation}{section}        

\newtheorem{thm}{Theorem}[section]
\newtheorem{lem}[thm]{Lemma}
\newtheorem{assum}[thm]{Assumption}
\newtheorem{cly}[thm]{Corollary}
\newtheorem{prop}[thm]{Proposition}

\theoremstyle{definition}
\newtheorem{defn}[thm]{Definition}

\newtheorem{rem}[thm]{Remark}

\DeclareMathOperator{\Diff}{Diff}          
\DeclareMathOperator{\Exp}{Exp}          
\DeclareMathOperator{\cof}{cof}          
\DeclareMathOperator{\Id}{Id}           
\DeclareMathOperator{\Op}{\mathfrak{Op}} 
\DeclareMathOperator{\sgn}{sgn}       
\DeclareMathOperator{\Tr}{Tr}           
\DeclareMathOperator{\argch}{argch}    
\DeclareMathOperator{\argsh}{argsh}    
\DeclareMathOperator{\sech}{sech}    

\renewcommand{\a}{\alpha}             
\newcommand{\B}{\mathcal{B}}            
\renewcommand{\b}{\beta}              
\newcommand{\bfr}{\mathfrak{b}}				
\newcommand{\C}{\mathbb{C}}             
\newcommand{\D}{\mathcal{D}}            
\newcommand{\E}{\mathcal{E}}						
\newcommand{\F}{\mathcal{F}}						
\newcommand{\del}{\partial}             
\newcommand{\eps}{\varepsilon}          
\newcommand{\Ga}{\Gamma}                
\newcommand{\ga}{\gamma}                
\renewcommand{\H}{\mathcal{H}}          
\newcommand{\HH}{\mathbb{H}}          
\newcommand{\half}{{\mathchoice{\thalf}{\thalf}{\shalf}{\shalf}}}
\newcommand{\ka}{\kappa}								
\newcommand{\la}{\lambda}               
\newcommand{\M}{\mathcal{M}}            
\newcommand{\N}{\mathbb{N}}             
\newcommand{\ol}{\overline}             
\renewcommand{\O}{\mathcal{O}}						
\newcommand{\wox}{\wh \otimes}					
\newcommand{\om}{\omega}                
\newcommand{\ox}{\otimes}               
\newcommand{\R}{\mathbb{R}}             
\newcommand{\RR}{\mathcal{R}}						
\newcommand{\rt}{\mathrm{t}}             
\newcommand{\rx}{\mathrm{x}}             
\newcommand{\ry}{\mathrm{y}}             
\newcommand{\rv}{\mathrm{v}}             
\newcommand{\ru}{\mathrm{u}}             
\newcommand{\shalf}{{\scriptstyle\frac{1}{2}}} 
\renewcommand{\S}{\mathcal{S}}         
\newcommand{\sg}{\sigma}                
\renewcommand{\th}{\theta}              

\newcommand{\thalf}{\tfrac{1}{2}}       
\newcommand{\Ups}{\Upsilon}							
\newcommand{\vth}{\vartheta}            
\newcommand{\wh}{\widehat}              
\newcommand{\wt}{\widetilde}            

\newcommand{\norm}[1]{\left\lVert#1\right\rVert} 
\newcommand{\set}[1]{\{\,#1\,\}}        
\newcommand{\twobyone}[2]{\begin{pmatrix}#1\\#2\end{pmatrix}} 
\def\<#1,#2>{\langle#1\mathbin|#2\rangle} 

\title{
\textbf{Pseudodifferential operators on manifolds with linearization}}
\author{\textbf{Cyril Levy} 
\\
\vspace{0cm}
\\
 \small Universit\'e de Provence and Centre de Physique Th\'{e}orique$^1$,\\
\small CNRS--Luminy, Case 907, 13288 Marseille Cedex 9, France.\\
\small {\em Email address:} {cyril.levy@ens-lyon.org} }
\date{June 2009}
\begin{document}

\maketitle
\vspace{0.6cm}  
\begin{abstract}
We present in this paper the construction of a pseudodifferential calculus on smooth non-compact manifolds associated to a globally defined and coordinate independant complete symbol calculus, that generalizes the standard pseudodifferential calculus on $\R^n$. We consider the case of manifolds $M$ with linearization in the sense of Bokobza-Haggiag \cite{Bokobza}, such that the associated (abstract) exponential map provides global diffeomorphisms of $M$ with $\R^n$ at any point. Cartan--Hadamard manifolds are special cases of such manifolds. The abstract exponential map encodes a notion of infinity on the manifold that allows, modulo some hypothesis of $S_\sigma$-bounded geometry, to define the Schwartz space of rapidly decaying functions, globally defined Fourier transformation and classes of symbols with uniform and decaying control over the $x$ variable. Given a linearization on the manifold with some properties of control at infinity, we construct symbol maps and $\la$-quantization, explicit Moyal star-product on the cotangent bundle, and classes of pseudodifferential operators. We show that these classes are stable under composition, and that the $\la$-quantization map gives an algebra isomorphism (which depends on the linearization) between symbols and pseudodifferential operators. We study, in our setting, $L^2$-continuity and give some examples. We show in particular that the hyperbolic 2-space $\HH$ has a $S_1$-bounded geometry, allowing the construction of a global symbol calculus of pseudodifferential operators on $\S(\HH)$. 
  
\vspace{1.5cm}  
\vspace{0.2cm}

 
 \vspace{1.5cm}

\noindent $^1$ UMR 6207

\noindent -- Unit\'e Mixte de Recherche du CNRS et des
Universit\'es Aix-Marseille I, Aix-Marseille 
II et de l'Universit\'e
du Sud Toulon-Var

\noindent -- Laboratoire affili\'e \`a la FRUMAM -- FR 2291
\end{abstract}

\tableofcontents


\section{Introduction}

Classically, a pseudodifferential operator on a (smooth, finite dimensional) 
manifold is defined through local charts and the notion of pseudodifferential operator
on open subsets of $\R^n$ \cite{Shubin,Treves}. In this setting, the full symbol of a
pseudodifferential operator is a coordinate dependent notion. However, 
the \emph{principal} symbol can be globally defined as a function on the cotangent bundle.
Naturally, the question of a full coordinate free definition of the symbol calculus of pseudodifferential 
operators on a manifold has been considered. One approach, based on the ideas of Bokobza-Haggiag \cite{Bokobza}, Widom \cite{Widom1,Widom2} 
and Drager \cite{Drager} 
allows such a calculus if one provides the manifold with a linear connection. Parallel transport along geodesics and the exponential map to connect any two points sufficiently close on the manifold are then used for the definitions and properties of local phase functions and oscillatory integrals.
Safarov \cite{Safarov} has formulated a version 
of a full coordinate free symbol calculus and $\la$-quantization ($0\leq \la \leq 1$) using invariant oscillatory integral over the cotangent bundle and 
determined by the linear connection. Pflaum \cite{PflaumQ,Pflaum} developped a complete symbol calculus on any Riemannian manifold using normal coordinates and microlocal lift on the test functions on manifolds with arbitrary Hermitian bundles. Sharafutdinov \cite{Shara1,Shara2} constructed a similar global pseudodifferential calculus, based on coordinate invariant geometric symbols. Further results in the same direction, connection to Weyl quantization and application to physics has been considered in Fulling and Kennedy \cite{Fulling}, Fulling \cite{Fulling0} and G\"{u}nt\"{u}rk \cite{Gunturk}. Connection between complete symbol calculus, deformation quantization and star-products on the cotangent bundle has also been made (see for instance Gutt \cite{Gutt}, Bordemann, Neumaier and Waldmann \cite{Bordeman} and Voronov \cite{Voronov1,Voronov2}). Getzler \cite{Getzler} used a global pseudodifferential calculus in the context of the Atiyah-Singer index theorem on supermanifolds.

All these pseudodifferential calculi are based on symbol (functions of $(x,\th)\in T^*M$) estimates over the covariable $\th$ while the dependence on the variable $x$ is only controlled locally uniformly on compact sets. This is well suited for the case of a compact manifold. For non-compact manifolds, we have to impose a uniform control over $x$ in order to obtain $L^2(M)$ continuity of operators of order 0 and compactness of the remainder operators if the control over $x$ is decaying. In other words, any global pseudodifferential calculus adapted to non-compact manifolds and sensitive to non-local effects needs to encode the behaviour ``at infinity" of symbols. On the Euclidean space $\R^n$, several types of pseudodifferential calculi have been defined: standard pseudodifferential calculus with uniform control over $x$ (see for instance H\"{o}rmander \cite{Hormander}, Beals \cite{Beals}, Shubin \cite{Shubin3}), isotropic calculus with simultaneous decay of the $x$ and $\th$ variables (Shubin \cite{Shubin,Shubin2}, Melrose \cite{Melrose}), and $SG$-pseudodifferential calculus with separated decay of the $x$ and $\th$ variables (Shubin \cite{Shubin2}, Parenti \cite{Parenti}, Cordes \cite{Cordes1,Cordes2}, Schrohe \cite{Schrohe}), which is invariant under a special class of diffeomorphisms and can be extended to an adapted class of manifolds, namely the $SG$-manifolds (Schrohe \cite{Schrohe}). This class of manifolds contains the non-compact manifolds ``with exits" and adapted pseudodifferential calculus has been developed (see for instance Cordes \cite{Cordes1}, Schulze \cite{Schulze}, Maniccia and Panarese \cite{Maniccia2}). Another approach, based on Lie structures at infinity, has been investigated to study the geometry of pseudodifferential operators on non-compact manifolds. Describing the geometry at infinity of the basis manifold by a Lie algebra of vector fields, an adapted pseudodifferential calculus has been constructed (see for instance Melrose \cite{Melrose2}, Mazzeo and Melrose \cite{Mazzeo}, Ammann, Lauter and Nistor \cite{Ammann}). Let us also mention the groupoid approach: by associating to any manifold with corners a smooth Lie groupoid and by building a pseudodifferential calculus on Lie groupoids, the $b$-calculus of Melrose on manifolds with corners can be generalized (see Monthubert \cite{Monthubert}).

Our purpose in this paper is to construct a global pseudodifferential calculus that generalizes the standard and $SG$ calculi on $\R^n$, on manifolds with linearization. These manifolds provide a natural geometric setting to deal simultaneously with the questions of a global isomorphism between symbols and pseudodifferential operators, and the non-local effects associated to non-compact manifolds.

The papers in organized as follows.
We define in section 2 a manifold with linearization (or exponential manifold) as a pair $(M,\exp)$ where $M$ is a smooth real finite-dimensional manifold and $\exp$ is an abstract exponential map, a smooth map from the tangent bundle onto $M$ that satisfies, besides the usual properties of an exponential map associated to a connection $\nabla$ on $TM$, the property that at each point $x\in M$, $\exp_x$ is a diffeomorphism. Any Cartan--Hadamard manifold with its canonical exponential map is an exponential manifold. These diffeomorphisms are used to define topological vector spaces of functions on the manifold (or on $TM$, $T^*M$, $M\times M$) that generalize, for instance, the notions of rapidly decaying function on $\R^n$ or of tempered distribution, provided that we add a hypothesis of ``$\O_M$-bounded geometry" on the exponential map.  
In section 3, we use linearizations in the spirit of Bokobza-Haggiag \cite{Bokobza}, to define symbol and quantization maps. This leads to topological isomorphisms between tempered distributional sections on $T^* M$ and $M\times M$, if we consider polynomially controlled (at infinity) linearizations ($\O_M$-linearizations). In particular, we extend the usual (explicit) Moyal product (or $\la$-product, for the $\la$-quantization) on any exponential manifold with $\O_M$-bounded geometry on which we set a $\O_M$-linearization. We get the following $\la$-product formula, giving a Fr\'{e}chet algebra structure to $\S(T^*M)$, 
$$ 
a\circ_\la b\,(x,\eta) = \int_{T_x(M) \times M} d\mu_{x}(\xi)d\mu(y)\int_{V^\la_{x,\xi,y}} d\mu_{x,\xi,y}^*(\th,\th')\, g^\la_{x,\xi,y}\,e^{2\pi i \om^\la_{x,\xi,y}(\eta,\th,\th')} a(y^\la_{x,\xi},\th)\,b(y^{1-\la}_{x,-\xi},\th')
$$
where $a,b\in \S(T^*M)$ and the other notations are detailed in Proposition \ref{la-product}.

In section 4, we define the symbol and amplitudes spaces for our pseudodifferential calculus. Symbol spaces can be defined in an intrinsic way on the exponential manifold with the help of "symbol-like" control ($S_\sg$-bounded geometry, see Definition \ref{ssigmadef}) of the coordinate change diffeomorphisms $\psi_{z,z'}^{\bfr,\bfr'}$ associated to the exponential map $\exp$ on $M$.
For practical reasons the definition of amplitudes here is slightly different from the usual functions of the parameters $x,y$ and $\th$. Instead, our amplitudes generalize functions of the form $(\rx,\zeta,\vth)\mapsto a(\rx,\rx+\zeta,\vth)$, where $a$ is a standard amplitude of the Euclidian pseudodifferential calculus. We establish continuity and regularity results for operators of the following form (which can be seen, for some forms of $\Ga$, as special Fourier integral operators on $\R^n$):  
$$
\langle \Op_{\Ga}(a) , u\rangle :=  \int_{\R^{3n}} e^{2\pi i \langle \vth,\zeta \rangle } \Tr\big(a(\rx,\zeta,\vth)\, \Ga(u)^*(\rx,\zeta)) \, d\zeta\,d\vth\,  d\rx
$$
where $\Ga$ is a topological isomorphism on $\S(\R^{2n},L(E_z))$ (here $E_z$ is a fixed fiber of the Hermitian bundle $E\to M$, so $L(E_z)$ can be identified with $\M_{\dim E_z}(\C)$), $a$ is in a $\O_{f,z}$ space (see Definition \ref{OFZ}) and $u\in \S(\R^n,E_z)$. In particular, results of Proposition \ref{ampliOP} and \ref{amplContinu} and Lemma \ref{noyauReste} are believed to be new.

With the help of a hypothesis of a control of symbol type over the derivative of the linearization ($S_\sg$-linearizations), we obtain in section \ref{pdosection} an intrinsic definition (Theorem \ref{lambdainv}) of pseudodifferential operators $\Psi_{\sg}^{l,m}$ on $M$.
We see in section \ref{linkstd} a condition $(H_V)$ on the linearization that entails that any pseudodifferential operator on $M$, when transferred in a frame $(z,\bfr)$, is a standard pseudodifferential operator on $\R^n$. This condition yields a $L^2$-continuity result in Proposition \ref{L2cont}. 
The last part of section 4 is devoted to the derivation of a symbol product asympotic formula for the composition of two pseudodifferential operators. The main result is Theorem \ref{compo}: under a special hypothesis $(C_\sg)$ on the linearization (see Definition \ref{Csigma}), we have the following asymptotic formula for the normal symbol (transferred in a frame $(z,\bfr)$) of the product of two pseudodifferential operators
$$
\sigma_{0}(AB)_{z,\bfr} \sim \sum_{\b,\ga \in \N^n} c_\b c_\ga \del_{\zeta,\vth}^{\ga,\ga}\big( a(\rx,\vth)\del^{\b}_{\zeta'}\big(e^{2\pi i \langle \vth,\varphi_{\rx,\zeta}(\zeta')\rangle} (\del^{\b}_{\vth'} f_b)(\rx,\zeta,\zeta',L_{\rx,\zeta}(\vth))\big)_{\zeta'=0} \tau^{-1}_{\rx,\zeta} \big)_{\zeta=0}
$$
where  $a:=\sigma_0(A)_{z,\bfr}$, $b:=\sigma_0(B)_{z,\bfr}$, and other notations are defined in section \ref{composec}.

Finally, we give in section \ref{exsec} two possible settings (besides the usual standard calculus on the Euclidian $\R^n$) in which the previous calculus applies. The first is based on the Euclidian space $\R^n$, with a ``deformed" (non-bilinear, non-flat) $S_\sg$-linearization. The second example is the hyperbolic plane (or Poincar\'{e} half-plane) $\HH$. We prove in particular that $\HH$ has a $S_1$-bounded geometry. This allows to define a global Fourier transform, Schwartz spaces $\S(\HH)$, $\S(T^*\HH)$, $\S(T\HH)$, $\B(\HH)$ and the space of symbols $S_1^{l,m}(T^*\HH)$. Moreover we can then define in an intrinsic way a global complete pseudodifferential calculus on $\HH$, and Moyal product, for any specified $S_\sg$-linearization on $\HH$.

\section{Manifolds with linearization and basic function spaces}
\subsection{Abstract exponential maps, definitions and notations}

The notion of linearization on a manifold was first introduced by Bokobza-Haggiag in \cite{Bokobza} and corresponds to a smooth map $\nu$ from $M\times M$ into $TM$ such that $\pi \circ \nu = \pi_1$, $\nu(x,x)=0$ for any $x\in M$ and $(d_y\nu)_{y=x}=\Id_{T_x M}$. In all the following, we shall work with ``global" linearizations, in the following sense:

\begin{defn} A manifold with linearization (or exponential manifold) is a pair $(M,\exp)$ where $M$ is a smooth manifold and $\exp$ a smooth map from $TM$ into $M$ such that:

\noindent $(i)$ for any $x\in M$, $\exp_x:T_x M \to M$ defined as $\exp_x(\xi):=\exp(x,\xi)$, is a global diffeomorphism between $T_x M$ and $M$,

\noindent $(ii)$ for any $x\in M$, $\exp_x(0)=x$ and $(d\exp_x)_0=\Id_{T_xM}$.

\noindent The map $\exp$ will be called the exponential map, and $(x,y)\mapsto \exp_x^{-1}(y)$ the linearization, of the exponential manifold $(M,\exp)$. We shall sometimes use the shorthand $e_x^\xi:=\exp_x (\xi)$.
\end{defn}

Note that the term ``exponential manifold" used here is not to be confused with the notion of ``exponential statistical manifold" used in stochastic analysis. 
Remark that if $\exp\in C^{\infty}(TM,M)$ satisfies $(i)$, then defining $\Exp:=\exp\circ\ T$ where $T(x,\xi):=\exp_{x}^{-1}(x)+(d\exp^{-1}_x)_x\xi$, we see that $(M,\Exp)$ is an exponential manifold.

We will say that $(M,\nabla)$ (resp. $(M,g)$) is exponential, where $M$ is a smooth manifold with connection $\nabla$ on $TM$ (resp. with pseudo-Riemannian metric $g$), if $(M,\exp)$ where $\exp$ is the canonical exponential map associated to $\nabla$ (resp. to $g$) is an exponential manifold, or in other words, if for any $x\in M$, $\exp_x$ is a diffeomorphism from $T_x M$ onto $M$. Note that $(M,\nabla)$ (resp. $(M,g)$) is exponential if and only if
\begin{itemize}
\item $M$ is geodesically complete
\item For any $x,y\in M$, there exists one and only one maximal geodesic $\ga$ such that $\ga(0)=x$ and $\ga(1)=y$.
\item For any $x\in M$, $\exp_x$ is a local diffeomorphism.
\end{itemize}

\begin{rem} $\R^n$ (with its standard metric of signature $(p,n-p)$) is an exponential manifold and any $n$-dimensional real exponential manifold is diffeomorphic to $\R^n$. In particular, an exponential manifold cannot be compact. 
A Cartan--Hadamard manifold is a Riemannian, complete, simply connected manifold with nonpositive sectional curvature. It is a consequence of the Cartan--Hadamard theorem (see for instance \cite[Theorem 3.8]{Lang}) that any Cartan--Hadamard manifold is exponential.
\end{rem}

\begin{rem} The exponential structure can be transported by diffeomorphism: if $(M,\exp_M)$ is an exponential manifold, $N$ a smooth manifold and $\varphi: M\to N$  is a diffeomorphism, then $(N,\exp_N:=\varphi\circ \exp_M \circ \  T\varphi^{-1})$ is an exponential manifold.
\end{rem}
\begin{assum} We suppose from now on that $(M,\exp)$ is an exponential $n$-dimensional real manifold.
\end{assum}

For any $x,y\in M$, we define $\ga_{xy}$ as the curve $\R\to M$, $t\mapsto \exp_{x}(t\exp_x^{-1}y)$, and $\wt\ga_{xy}(t):=\ga_{yx}(1-t)$. Note that $\ga_{xy}(0)=x$ and $\ga_{xy}(1)=y$. If the exponential map is derived from a linear connection, we have for any $t\in \R$, $\ga_{xy}(t)=\wt \ga_{xy}(t)$. In the general case, this is only true for $t=0$ and $t=1$. 

The abstract exponential map $\exp$ provides the manifold $M$ with a notion of ``points at infinity" and ``straight lines" ($\ga_{xy}$). It can be seen as a generalization to manifolds of the useful properties of $\R^n$ for the study of the behaviour of functions at infinity. The abstract exponential map $\exp$ formalizes the fact that our straight lines never stop and connect any two different points.

The diffeomorphism $\exp_z^{-1}$, for a given $z\in M$, is not stricto sensu a chart, since it maps $M$ onto $T_z M$, which is diffeormorphic but not equal to $\R^n$. In order to obtain a chart, one needs to choose a linear basis of $T_z M$. If $z\in M$ and $\bfr$ is a basis of $T_z M$ we will call the pair $(z,\bfr)$ a (normal) frame. For any frame $(z,\bfr)$, we define $n_z^{\mathfrak{b}}:=L_\bfr\circ \exp_z^{-1}$ with $L_\bfr$ the linear isomorphism from $T_z M$ onto $\R^n$ associated to $\mathfrak{b}$. As a consequence, the pair $(M,n_z^\bfr)$ is a chart which is a global diffeomorphism from $M$ onto $\R^n$.

We note $\psi_{z,z'}^{\bfr,\bfr'}:= n_z^\bfr \circ (n_{z'}^{\bfr'})^{-1}$ the normal coordinate change diffeomorphism from $\R^n$ onto $\R^n$ and
$(\del_{i,z,\bfr})_{i\in \N_n}$ and $(dx^{i,z,\bfr})_{i\in \N_n}$ (whith $\N_n:=\set{1,\cdots,n}$) the global frame vector fields and 1-forms associated to the chart $n_z^\bfr$. We also note $n^\bfr_{z,*}$ the diffeomorphism from $T^*M$ onto $\R^{2n}$ defined by $n^\bfr_{z,*}(x,\th)=(n^\bfr_{z}(x),\wt M_{z,x}^\bfr(\th))$ where $(\wt M_{z,x}^\bfr(\th)_i)_{i\in \N_n}$ are the components of $\th$ in $(dx^{i,z,\bfr}_{x})_{i\in \N_n}$ and $n_{z,T}^\bfr : (x,\xi)\to (n_z^\bfr(x), M^\bfr_{z,x}(\xi))$ the diffeomorphism from $TM$ onto $\R^{2n}$, where $( M ^\bfr_{z,x}(\xi)_i)_{i\in \N_n}$ are the coordinates of $\xi$ in the basis $({\del_{i,z,\bfr}}_x)_{i\in \N_n}$. We have $M^{\bfr}_{z,x}  = (dn_{z}^\bfr)_x$ and $\wt M^{\bfr}_{z,x}  =\, ^t(dn_{z}^\bfr)_x^{-1}$.
The diffeomorphism from $M\times M$ onto $\R^{2n}$ defined by $(x,y)\mapsto (n_z^\bfr(x),n_z^\bfr(y))$ will be noted $n_{z,M^2}^\bfr$.

We note $(\del_{i,z,\bfr})_{i\in \N_{2n}}$ the family of vector fields on $T^*M$ (resp. $TM$, $M\times M$) associated to the chart $n^\bfr_{z,*}$ (resp. $n^\bfr_{z,T}$, $n^\bfr_{z,M^2}$) onto $\R^{2n}$. 
We suppose in all the following that $\mathfrak{E}$ is an arbitrary normed finite dimensional complex vector space. If $\nu$ is a ($2n$)-multi-index, we define the following operator on $C^{\infty}(T^*M,\mathfrak{E})$ (resp. $C^{\infty}(TM,\mathfrak{E})$, $C^{\infty}(M\times M,\mathfrak{E})$): 
$$
\del_{z,\bfr}^\nu:= \prod_{k=1}^{2n}\,\del_{k,z,\bfr}^{\nu_k}.
$$
If $\a$ and $\beta$ are $n$-multi-indices, we note $(\a,\beta)$ the $2n$-multi-index obtained by concatenation. If $\a$ is a $n$-multi-index, $\del_{z,\bfr}^\a$ is a linear operator on $C^\infty(M,\mathfrak{E})$. 
We fix the shorcut $\langle \rx \rangle:=(1+\norm{\rx}^2)^{1/2}$ for any $\rx\in \R^p$, $p\in \N$. 
We will use the convention $\rx^\a := \rx_1^{\a_1} \cdots \rx_p^{\a_p}$ for $\rx\in \R^p$ and $\a$ $p$-multi-index, with $0^0:=1$. 
If $f$ is continuous function from $\R^p$ to a normed vector space and $g$ is a continuous function from $\R^p$ to $\R$, we note $f=\O(g)$ if and only if there exist $r>0$, $C>0$ such that for any $\rx\in \R^{p}\backslash B(0,r)$, $\norm{f(\rx)}\leq C |g(\rx)|$. In the case where $g$ is strictly positive on $\R^p$, this is equivalent to: there exists $C>0$ such that for any  $\rx\in \R^{p}$, $\norm{f(\rx)}\leq C g(\rx)$.  
We also introduce the following shorthands, for given $(z,\bfr)$, $x,y\in M$, $\th\in T^*_x (M)$, $\xi\in T_x(M)$:
\begin{align*}
&\langle x\rangle_{z,\bfr} := \langle{n_{z}^\bfr(x)}\rangle,\qquad
\langle \th \rangle_{z,\bfr,x} := \langle{\wt M_{z,x}^\bfr(\th)}\rangle,\qquad
\langle \xi \rangle_{z,\bfr,x} := \langle{ M_{z,x}^\bfr(\xi)}\rangle, \\
&\langle x,y\rangle_{z,\bfr}:=\langle (n_z^\bfr(x),n_z^\bfr(y))\rangle,\quad  \langle x,\th\rangle_{z,\bfr}:=\langle (n_z^\bfr(x),\wt M_{z,x}^{\bfr}(\th))\rangle, \quad \langle x,\xi\rangle_{z,\bfr}:=\langle (n_z^\bfr(x), M_{z,x}^{\bfr}(\xi))\rangle\, .
\end{align*}
If $f$ and $g$ are in $C^0(\R^p, \R^{p'})$ we note $f\asymp g$ the equivalence relation defined by: $\langle f\rangle =\O(\langle g\rangle)$ and $\langle g \rangle =\O(\langle f \rangle )$.
 
\subsection{Parallel transport on an Hermitian bundle}

Let $E$ be an hermitian vector bundle (with typical fiber $\mathbb{E}$ as a finite dimensional complex vector space) on the exponential manifold $(M,\exp)$. $E$ admits a (non-unique) connection $\nabla^E$ compatible with the hermitian metric \cite{Berline}. It is a differential operator from $C^\infty(M,E)$ (the space of smooth sections of $E\to M$) to $C^\infty(M,T^*M\otimes E)$
such that for any smooth function $f$ on $M$ and smooth $E$-sections $\psi$, $\psi'$, 
\begin{align*}
&\nabla^E(f\psi)= df\ox\psi+ f\nabla^E \psi \, ,\\
&d(\psi|\psi')= (\nabla^E \psi|\psi')+(\psi|\nabla^E\psi')\, ,
\end{align*}
where $(\psi|\psi')$ is the hermitian pairing of $\psi$ and $\psi'$. We will note $|\psi|^2:=(\psi|\psi)$. The sesquilinear form $(\cdot|\cdot)_x$ of $E_x$ is antilinear in the second variable by convention. The operator $\nabla^E$ can be (uniquely) extended as an operator acting on $E$-valued differential forms on $M$. 
If $\gamma$ is a curve on $M$ defined on an interval $J$ and $\ga^*E$ the associated pullback bundle on $J$, there exists a natural connection (the pullback of $\nabla^E$) on $\ga^*E$, noted $\nabla^{\ga^*E}$ compatible with $\nabla^E$.

Let us fix $x,y\in M$ and $\ga :\, J \to M$ a curve such that $\ga(0)=x$ and $\ga(1)=y$. For any $v\in E_x$, there exists an unique smooth section $\beta$ of $\ga^*E\to J$ such that $\beta(0)=v$ and $\nabla^{\ga^*E} \beta=0$. Clearly, $\beta(1) \in E_{y}$ and we can define a linear isomorphism $\tau_{\ga}$ from $E_x$ to $E_y$ as $\tau_{\ga}(v)=\beta(1)$.
The map $\tau_{\ga}$ is the parallel transport map associated to $\ga$ from $E_x$ to $E_y$. 
The compatibility of $\nabla^E$ with the hermitian metric entails that the maps $\tau_{\ga}$ are in fact isometries for the hermitian structures on $E_x$ and $E_y$. 

The vector bundle $L(E)\to M$, defined by $L(E)_x:=L(E_x)$ (the space of endomorphisms on $E_x$), is lifted to $T^*M$, $TM$ and $M\times M$ by setting the fiber at $(x,\th)$ to $L(E_x)$ for $T^*M$ or $TM$, and the fiber at $(x,y)$ to $L(E_y,E_x)$ for $M\times M$. The canonical projection from $T^*M$ or $TM$ to $M$ is noted $\pi$.

We note $\tau_{xy}:=\tau_{\ga_{xy}}$. Remark that $\tau_{xy}^{-1}=\tau_{\wt\ga_{yx}}$.  
We define $\tau_z : x \mapsto \tau_{zx}$ and $\tau_z^{-1} : x\mapsto \tau_{zx}^{-1}=\tau_{zx}^*$. 

If $u\in C^\infty(M,E)$ and $z\in M$, we note $u^z (x):= (\tau_{z}^{-1}\, u)(x)$ for any $x\in M$.
If $a$ is section of $L(E)\to T^*M$ or $L(E)\to TM$, we note $a^z := (\tau_z^{-1} \circ \pi) \, a\,(\tau_{z}\circ \pi)$. If $a$ is a section of $L(E)\to M\times M$, we note $a^z(x,y):= \tau_z^{-1}(x)\,a(x,y)\, \tau_{z}(y)$.
We also define $\tau^z:= (x,y)\mapsto \tau_z^{-1}(y)\tau(x,y) \tau_z(x) \in L(E_z)$. Noting $\pi_1(x,y):=x$, $\pi_2(x,y):=y$, we get $a^z= (\tau_z^{-1}\circ \pi_1)\,a\, (\tau_z\circ \pi_2) $ and $\tau^z=(\tau_z^{-1}\circ \pi_2)\, a\, (\tau_z\circ \pi_1)$.

Parallel transport on $E$ has the following smoothness property: 

\begin{lem} 
(i) 
The map $\tau : (x,y) \mapsto \tau_{xy}$ (resp. $\tau^{-1} : (x,y) \mapsto \tau_{xy}^{-1}$) is a smooth section of the vector bundle $L(E)^\vee\to M\times M$ where the fiber at $(x,y)$ is $L(E_x,E_y)$ (resp. of the vector bundle $L(E)\to M\times M$).

\noindent (ii) $\tau_z \in C^\infty(M,L(E_z,E))$ and $\tau_z^{-1} \in C^{\infty}(M,L(E,E_z))$.

\noindent (iii) $\tau^z \in C^\infty(M\times M,L(E_z))$.
\end{lem}

\begin{proof} $(i)$ The map $G: TM \to M\times M$ defined by $G(v):= (\pi(v),\exp(v))$ is a local diffeomorphism since the Jacobian of $G$ at $v_0=(x_0,\xi_0)\in TM$ is equal to the Jacobian of $\exp_{x_0}$ at $\xi_0$. Since it is also bijective (with inverse $G^{-1}(x,y):= (x,\exp_x^{-1}(y))$), it is a (global) diffeomorphism $TM\to M\times M$.
The map $b(x,y,t):= (x,t\exp_x^{-1}(y))$ is thus a smooth map from $M\times M\times \R$ to $TM$, and we get a smooth parametrization by $M\times M$ of the following family of curves: $c (x,y)\mapsto \big(\ga_{xy}: t\mapsto \exp b(x,y,t))$. This parametrization leads (see \cite[p. 17]{Dumitrescu}) to a smooth bundle homomorphism between $c^*(\cdot)(0)E\to M\times M$ and $c^*(\cdot)(1) E\to M\times M$, so a smooth section $\tau:(x,y)\mapsto \tau_{xy}$ of $L(E_x,E_y)\to M\times M$. The case of $\tau^{-1}$ is similar, by taking $b^{-1}(x,y,t):=b(x,y,1-t)$.

\noindent $(ii,iii)$ are straightforward consequences of $(i)$.
\end{proof}

\begin{cly} If $u$ is in the space $C^\infty(M,E)$, then $u^z \in C^\infty(M,E_z)$. Similarly, if $a\in C^\infty({T^*M,L(E)})$ (resp. $C^\infty(TM,L(E))$, $C^\infty(M\times M, L(E))$), then $a^z \in C^\infty({T^*M, L(E_z)})$ (resp. $C^\infty({TM, L(E_z)})$, $C^\infty(M\times M ,L(E_z))$). 
\end{cly}

\begin{rem} The vector bundle $E$ on $M$ is trivializable and the parallel transport provides a $M$-indexed family of trivializations, since for any $z\in M$, the pair $f_z :E\mapsto M \times \mathbb{E} , (x,v)\mapsto (x,\tau_{xz}(v))$, $\Id : M\mapsto M, x\mapsto x$, is a vector bundle isomorphism 
from $E\to M$ onto $M \times \mathbb{E}\to M$. Note that if $\exp$ is derived from a connection, $\tau_{xy}^{-1}=\tau_{yx}$ for any $x,y\in M$.
\end{rem}

\subsection{$\O_M$ and $S_\sigma$-bounded geometry}

Classically, in Riemannian geometry, bounded geometry hypothesis gives boundedness on the covariant derivative of the Riemann curvature of the basis manifold.
For the following pseudodifferential calculus, we shall need some hypothesis of that kind, formulated not with the curvature but with the exponential diffeomorphisms (``normal" coordinate transition maps). 
The hypothesis that we will need for pseudodifferential symbol calculus is actually not simply the boundedness condition on the derivatives of the transition maps, which is a classical consequence of bounded geometry. For symbol calculus, we will require that the $n^{\text{th}}$-derivatives are not only bounded, but decrease to zero at infinity as $\norm{x}^{-\sigma (n-1)}$ where $\sigma$ is a parameter in $[0,1]$. Or, in other words, the normal coordinate change maps behave as ``symbols" or order 1. Thus, we introduce the following

\begin{defn}
\label{ssigmadef}
Let $\sigma \in [0,1]$. The exponential manifold $(M,\exp)$ is said to have a $S_\sigma$-bounded geometry if for any $(z,\bfr)$, $(z',\bfr')$, and any $n$-multi-index $\a\neq 0$, 
\begin{align*}
&(S_\sigma 1)\qquad \del^{\a} \psi_{z,z'}^{\bfr,\bfr'}(\rx) = \mathcal{O}(\langle \rx \rangle^{-\sigma(|\a|-1)})\, ,
\end{align*}
and a $\O_M$-bounded geometry if for any $(z,\bfr)$, $(z',\bfr')$, and any $n$-multi-index $\a$, there exist $p_\a\geq 1$ such that
\begin{align*}
&(\O_M 1)\qquad \del^{\a} \psi_{z,z'}^{\bfr,\bfr'}(\rx) = \mathcal{O}(\langle \rx \rangle^{p_\a})  \, .
\end{align*}
\end{defn}

We shall be working with $\O_M$-bounded geometry for the definition of function spaces and Fourier transform and with $S_{\sigma}$-bounded geometry (for a $\sigma \in [0,1]$) for pseudodifferential symbol calculus. 

\begin{defn} The triple $(M,\exp,E)$ where $(M,\exp)$ is exponential and $E$ is a hermitian vector bundle on $M$ has a $S_\sigma$-bounded geometry if $(M,\exp)$ has a $S_\sigma$-bounded geometry and for any $(z,\bfr)$, $z',z''$, and any $n$-multi-index $\a$,  
$$
(S_\sigma 2)\qquad \del_{z,\bfr}^\a \tau_{z'}^{-1}\tau_{z''}(x) = \mathcal{O}( \langle x \rangle_{z,\bfr}^{-\sigma |\a|}) \, ,
$$
and a $\O_M$-bounded geometry if $(M,\exp)$ has a $\O_M$-bounded geometry and for any $(z,\bfr)$, $(z',\bfr')$, and any $n$-multi-index $\a$, there exist $p_\a\geq 1$ such that
\begin{align*}
&(\O_M 2)\qquad \del_{z,\bfr}^\a \tau_{z'}^{-1}\tau_{z''}(x) = \mathcal{O}(\langle x \rangle_{z,\bfr}^{p_\a})  \, .
\end{align*}
\end{defn}

Clearly, if $\sigma\leq \sigma'$, since $(S_{\sigma'} i)\Rightarrow (S_\sigma i)$, we have $S_{\sigma'} $-bounded $\Rightarrow$ $S_\sigma$-bounded $\Rightarrow$ $\O_M$-bounded. 
Note that $S_\sigma$-bounded geometry on the vector bundle entails that the derivatives of the transport transition maps $\tau_{z}^{-1}\tau_{z'}$ (smooth from $M$ to $L(E_{z'},E_z$)) are bounded (for $S_0$-bounded geometry) or decrease to zero with an order equal to the order of the derivative (for $S_1$-bounded geometry). Remark also that if $E$ is a trivial bundle and $\nabla^E=d$, then $(S_1 2)$ is automatically satisfied since the maps $\tau_{z}$ are all equal to the constant $x\mapsto Id_{\mathbb{E}}$.

\begin{lem}
\label{B1-lem}
Let $\sigma\in [0,1]$ and $(z,\bfr)$, $(z',\bfr')$ be given frames.

\noindent (i) If $(M,\exp)$ has a $S_\sigma$-bounded geometry, there exist $K,C,C'>0$ such that for any $\rx \in \R^n$, $x\in M$, $\th\in T^*_x(M)$, $\xi \in T_x(M)$,
\begin{align}
 \psi_{z,z'}^{\bfr,\bfr'} \asymp \Id_{\R^n}  \qquad \text{ and } \qquad \langle x \rangle_{z,\bfr} \leq K \langle x \rangle_{z',\bfr'}\, ,\label{i-first} \\
 \langle \th \rangle_{z,\bfr,x} \leq C \langle \th \rangle_{z',\bfr',x} \qquad \text{ and } \qquad \langle \xi \rangle_{z,\bfr,x} \leq C' \langle \xi \rangle_{z',\bfr',x}\, ,\label{i-bis}
\end{align}
and if $(M,\exp)$ has a $\O_M$-bounded geometry, there exist $K,K',K'',C,C'>0$ and $q\geq 1$ such that for any $\rx \in \R^n$, $x\in M$, $\th\in T^*_x(M)$, $\xi \in T_x(M)$,
\begin{align}
K' \langle \rx \rangle^{1/q} \leq \langle \psi_{z,z'}^{\bfr,\bfr'}(\rx)\rangle \leq K''\langle \rx \rangle^q  \qquad \text{ and } \qquad \langle x \rangle_{z,\bfr} \leq K \langle x \rangle^q_{z',\bfr'}\, ,\label{i-first-om} \\
 \langle \th \rangle_{z,\bfr,x} \leq C \langle x \rangle^q_{z',\bfr'}\langle \th \rangle_{z',\bfr',x} \qquad \text{ and } \qquad \langle \xi \rangle_{z,\bfr,x} \leq C' \langle x \rangle^q_{z',\bfr'}\langle \xi \rangle_{z',\bfr',x}\, ,\label{i-bis-om}
\end{align}

\noindent (ii) For any given $n$-multi-indice $\a$, we can write 
$$
\del_{z,\bfr}^{\a} = \sum_{0\leq |\a'|\leq |\a|} f_{\a,\a'}\, \del_{z',\bfr'}^{\a'}
$$
where the $(f_{\a,\a'})$ are smooth real functions on $M$ such that for each $n$-multi-indices $\a,\a'$,   

(a) if $(M,\exp)$ has a $S_\sigma$-bounded geometry, there exists $C_{\a}>0$ such that for any $x\in M$,
$| f_{\a,\a'} (x)| \leq C_{\a} \langle x\rangle_{z,\bfr}^{-\sigma (|\a|-|\a'|)}$,

(b) if $(M,\exp)$ has a $\O_M$-bounded geometry, there exist $C_{\a}>0$ and $q_{\a}\geq 1$ such that for any $x\in M$, $| f_{\a,\a'} (x)| \leq C_{\a} \langle x\rangle_{z,\bfr}^{q_{\a}}$.

\end{lem}
\begin{proof} $(i)$   Suppose that $(M,\exp)$ has a $S_\sigma$-bounded geometry. Taylor formula implies that 
$\norm{\psi_{z,z'}^{\bfr,\bfr'}(\rx)} \leq \norm{\psi_{z,z'}^{\bfr,\bfr'}(0)}  + C_0 \norm{\rx}$ for any $\rx\in \R^n$, where $C_0:= \sup_{\rx\in \R^n} \norm{(d\psi_{z,z'}^{\bfr,\bfr'})_\rx}$. As a consequence $\psi_{z,z'}^{\bfr,\bfr'}(\rx)=\O(\norm{\rx})$ and thus, there is $K''>0$ such that $\langle \psi_{z,z'}^{\bfr,\bfr'}(\rx) \rangle \leq K'' \langle \rx \rangle$. The same argument for $\psi_{z',z}^{\bfr',\bfr}=(\psi_{z,z'}^{\bfr,\bfr'})^{-1}$ gives $\psi_{z,z'}^{\bfr,\bfr'} \asymp \Id_{\R^n}$ and $\langle x \rangle_{z,\bfr} \leq K \langle x \rangle_{z',\bfr'}$ follows immediately. 
Since $x\mapsto \norm{\wt M_{z,x}^\bfr ( \wt M_{z',x}^{\bfr'})^{-1}}=\norm{(d \psi_{z',z}^{\bfr',\bfr})_{n_z^\bfr(x)}}$ and $x\mapsto \norm{ M_{z,x}^\bfr (  M_{z',x}^{\bfr'})^{-1}}=\norm{(d \psi_{z,z'}^{\bfr,\bfr'})_{n_{z'}^{\bfr'}(x)}}$ are bounded functions, (\ref{i-bis}) follows. The case where $(M,\exp)$ has a $\O_M$-bounded geometry is similar.

\noindent $(ii)$ We have for any $f\in C^\infty(M,\mathfrak{E})$, 
$$
\del_{z,\bfr}^\a (f) = \del^\a(f\circ (n_{z}^\bfr)^{-1}) \circ n_z^\bfr = \del^\a(f\circ (n_{z'}^{\bfr'})^{-1} \circ \psi_{z',z}^{\bfr',\bfr}) \circ n_z^\bfr\, .
$$
We now apply the multivariate Faa di Bruno formula obtained by G.M. Constantine and T.H. Savits in \cite{Constantine}, that we reformulated for convenience in Theorem \ref{FaaCS}. This formula entails that for any $n$-multi-index $\a\neq 0$,
$$
\del^\a(f\circ (n_{z'}^{\bfr'})^{-1} \circ \psi_{z',z}^{\bfr',\bfr}) =  \sum_{1\leq |\a'|\leq |\a|} P_{\a,\a'}(\psi_{z',z}^{\bfr',\bfr})\,(\del^{\a'} f\circ (n_{z'}^{\bfr'})^{-1} )\circ \psi_{z',z}^{\bfr',\bfr}\  
$$
and thus 
$$
\del_{z,\bfr}^\a =  \sum_{1\leq |\a'|\leq |\a|} (P_{\a,\a'}(\psi_{z',z}^{\bfr',\bfr})\circ n_z^\bfr) \ \del_{z',\bfr'}^{\a'} =:\sum_{1\leq |\a'|\leq |\a|} f_{\a,\a'} \ \del_{z',\bfr'}^{\a'}
$$ 
where $P_{\a,\a'}(g)$ is a linear combination of terms of the form $\prod_{j=1}^{s} (\del^{l^j} g)^{k^j} $, where $1\leq s\leq |\a|$ and  
the $k^j$ and $l^j$ are $n$-multi-indices with $|k^j|>0$, $|l^j|>0$, $\sum_{j=1}^s |k^j| = |\a'|$ and $\sum_{j=1}^s|k^j||l^j|=|\a|$. In the case where $(M,\exp)$ has a $S_\sigma$-bounded geometry, for each $s,(k^j),(l^j)$, there is $K>0$ such that for any $\rx\in \R^n$,
$$
|\prod_{j=1}^{s} (\del^{l^j} \psi_{z',z}^{\bfr',\bfr})^{k^j}(\rx)|\leq K \langle \rx \rangle^{-\sigma \sum_{j=1}^s(|l^j|-1)|k^j|} = K \langle \rx \rangle^{-\sigma (|\a|-|\a'|)}  
$$
which gives the result. The case where $(M,\exp)$ has a $\O_M$-bounded geometry is similar.
\end{proof}

\begin{thm}\cite{Constantine}
\label{FaaCS} 
Let $f\in C^\infty(\R^p,\mathfrak{E})$ and $g\in C^\infty(\R^n,\R^p)$. Then for any $n$-multi-index $\nu\neq 0$,
$$
\del^\nu (f\circ g) = \sum_{1\leq |\la|\leq |\nu|} (\del^\la f)\circ g\  \sum_{s=1}^{|\nu|} \sum_{p_s(\nu,\la)}\nu! \prod_{j=1}^s \tfrac{1}{k^j!(l^j!)^{|k^j|}} (\del^{l^j} g)^{k^j}
$$
where $p_s(\nu,\la)$ is the set of $p$-multi-indices $k^j$ and $n$-multi-indices $l^j$ ($1\leq j\leq s$) such that $0\prec l^1 \prec \cdots\prec l^s$ ($l\prec l'$ being defined as ``$|l|<|l'|$ or $|l|=|l'|$ and $l<_Ll'$" where $<_L$ is the strict lexicographical order), $|k^j|>0$, $\sum_{j=1}^s k^j = \la$ and $\sum_{j=1}^s |k^j| l^j= \nu$.
\end{thm}

Note that by Lemma \ref{B1-lem}, if $(M,\exp)$ satisfies $(S_\sigma 1)$ (resp. $(\O_M 1)$), then $(S_\sigma 2)$ (resp. $(\O_M 2)$) is equivalent to: for any $z',z''\in M$, there exists a frame $(z,\bfr)$ such that $
\del^\a_{z,\bfr} \tau_{z'}^{-1}\tau_{z''} (x) =\O(\langle x \rangle_{z,\bfr}^{-\sigma |\a|})$ (resp. $\O(\langle x \rangle_{z,\bfr}^{p_{\a}})$ for a $p_\a\geq 1$)
for any $n$-multi-index $\a$. 

As the following proposition shows, $S_\sigma$ or $\O_M$-bounded geometry properties can be transported by any diffeomorphism.

\begin{prop}
\label{ssigmDiff}
If $(M,\exp_M)$ has a $S_\sigma$ (resp. $\O_M$) bounded geometry, $N$ a smooth manifold and $\varphi: M\to N$  is a diffeomorphism, then $(N,\exp_N:=\varphi\circ \exp_M \circ \  d\varphi^{-1})$ has a $S_\sigma$ (resp. $\O_M$) bounded geometry.
\end{prop}
\begin{proof} Let us note $\psi_{z,z',N}^{\bfr,\bfr'}:=n_{z,N}^\bfr \circ (n_{z',N}^{\bfr'})^{-1}$ where $n_{z,N}^\bfr:= L_\bfr \circ \exp_{N,z}^{-1}$ and $(z,\bfr)$, $(z',\bfr')$ are two frames on $N$.
Since $\exp_{z',N}=\varphi \circ \exp_{M,\varphi^{-1}(z')}\circ (d\varphi^{-1})_{z'}$ and $\exp_{N,z}^{-1}= (d\varphi^{-1})_z^{-1}\circ \exp_{M,\varphi^{-1}(z)}^{-1}\circ\ \varphi^{-1}$, we obtain $\psi_{z,z',N}^{\bfr,\bfr'} = \psi_{\varphi^{-1}(z),\varphi^{-1}(z'),M}^{\bfr_z,\bfr'_{z'}}$ where $\bfr_z$ is the basis of $T_{\varphi^{-1}(z)}(M)$ such that $L_{\bfr_z}= L_\bfr\circ (d\varphi)_{\varphi^{-1}(z)}$. The result follows.
\end{proof}

The following technical lemma will be used for Fourier transform and the definition of rapidly decreasing section spaces over the tangent and cotangent bundle in section 3. It will also give the behaviour of symbols under coordinate change. 

\begin{lem}
\label{cchange} Let $(z,\bfr)$, $(z',\bfr')$ be given frames.

\noindent (i) We can express $\del_{z,\bfr}^{(\a,\b)}$ as an operator on $C^{\infty}(T^*M,\mathfrak{E})$ (resp. $C^{\infty}(TM,\mathfrak{E})$) , where $(\a,\b)$ is a $2n$-multi-index,  with the following finite sum:
$$
\del_{z,\bfr}^{(\a,\b)} = \underset{|\b'|\geq |\b|}{\sum_{0\leq |(\a',\b')|\leq |(\a,\b)|}}\, f_{\a,\b,\a',\b'}\,   \del_{z',\bfr'}^{(\a',\b')} 
$$
where the $f_{\a,\b,\a',\b'}$ are smooth functions on $T^*M$ (resp. $TM$) such that

(a) if $(M,\exp)$ has a $S_\sigma$-bounded geometry for a given $\sigma\in [0,1]$, there exists $C_{\a,\b}>0$ such that for any $(x,\th)\in T^*M$ (resp. $TM$),

\begin{equation}
\label{fabab}
|f_{\a,\b,\a',\b'}(x,\th)|\leq    C_{\a,\b}\langle x \rangle_{z,\bfr }^{\sigma(|\a'|-|\a|)}\langle \th \rangle_{z,\bfr,x}^{|\b'|-|\b|}.
\end{equation}

(b) if $(M,\exp)$ has a $\O_M$-bounded geometry, there exist $C_{\a,\b}>0$ and $q_{\a,\b}\geq 1$ such that for any $(x,\th)\in T^*M$ (resp. $TM$),

\begin{equation}
\label{fabab2}
|f_{\a,\b,\a',\b'}(x,\th)|\leq    C_{\a,\b}\langle x \rangle_{z,\bfr }^{q_{\a,\b}}\langle \th \rangle_{z,\bfr,x}^{|\b'|-|\b|}.
\end{equation}

\noindent (ii) We can express $\del_{z,\bfr}^{(\a,\b)}$ as an operator on $C^{\infty}(M\times M,\mathfrak{E})$,  with the following finite sum:
$$
\del_{z,\bfr}^{(\a,\b)} = \underset{0\leq |\b'|\leq |\b|}{\sum_{0\leq |\a'|\leq |\a|}}\, f_{\a,\b,\a',\b'}\,   \del_{z',\bfr'}^{(\a',\b')} 
$$
where the $f_{\a,\b,\a',\b'}$ are smooth functions on $M\times M$ such that 

(a) if $(M,\exp)$ has a $S_\sigma$-bounded geometry for a given $\sigma\in [0,1]$, there exists $C_{\a,\b}>0$ such that for any $(x,y)\in M\times M$,

\begin{equation}
\label{fabab3}
|f_{\a,\b,\a',\b'}(x,y)|\leq    C_{\a,\b}\,\langle x \rangle_{z,\bfr }^{\sigma(|\a'|-|\a|)}\,\langle y \rangle_{z,\bfr}^{\sigma(|\b'|-|\b|)}.
\end{equation}

(b) if $(M,\exp)$ has a $\O_M$-bounded geometry, there exist $C_{\a,\b}>0$ and $q_{\a},q_{\b}\geq 1$ such that for any $(x,y)\in M\times M$,

\begin{equation}
\label{fabab4}
|f_{\a,\b,\a',\b'}(x,y)|\leq    C_{\a,\b}\,\langle x \rangle_{z,\bfr}^{q_{\a}}\,\langle y \rangle_{z,\bfr}^{q_{\b}}.
\end{equation}

\end{lem}
\begin{proof} 
$(i)$ Suppose that $(M,\exp)$ has a $S_\sigma$-bounded geometry.
Let us note $\psi_*:=n_{z',*}^{\bfr'}\circ (n_{z,*}^{\bfr})^{-1}$ and $\psi_T:=n_{z',T}^{\bfr'}\circ (n_{z,T}^{\bfr})^{-1}$.
We have $\psi_*=(\psi_{z',z}^{\bfr',\bfr}\circ \pi_1,L)$ where $\pi_1$ is the projection from $\R^{2n}$ onto the first copy of $\R^n$ in $\R^{2n}$ and $L$ is the smooth map from $\R^{2n}$ to $\R^n$ defined as $L(\rx,\vth):=\, ^t(d\psi_{z',z}^{\bfr',\bfr})^{-1}_{\rx}(\vth)=\, ^t(d\psi_{z,z'}^{\bfr,\bfr'})_{\psi_{z',z}^{\bfr',\bfr}(\rx)}(\vth)$. Noting $(L_i)_{1\leq i\leq n}$ the components of $L$, we have $L_i(\rx,\vth)=\sum_{1\leq p\leq n} L_{i,p}(\rx)\,\vth_{p}$, where $L_{i,p}:=(\del_i \psi_{z,z'}^{\bfr,\bfr'})_p\circ \psi_{z',z}^{\bfr',\bfr}$. As a consequence, for $1\leq i\leq n$ and $\a,\b$, $n$-multi-indices such that $|(\a,\b)|>0$
\begin{align*}
&(\del^{(\a,\b)}\psi_*)_i=  \delta_{\b,0}(\del^{\a}\psi_{z',z}^{\bfr',\bfr})_i \circ \pi_1\, ,\qquad (\del^{(\a,\b)}\psi_*)_{n+i}=(\del^{(\a,\b)}L)_i\, , \\
&(\del^{(\a,\b)}L)_i(\rx,\vth) = \sum_{1\leq p\leq n} (\del^{\a} L_{i,p})(\rx)\  F_{\b,p}(\vth)\,,\\ 
&\del^{\a} L_{i,p} = \sum_{1\leq |\a'|\leq |\a|} P_{\a,\a'}(\psi_{z',z}^{\bfr',\bfr})\ ((\del^{\a'+e_i} \psi_{z,z'}^{\bfr,\bfr'})_p\circ \psi_{z',z}^{\bfr',\bfr})\quad \text{ if } \quad |\a|>0 \, ,
\end{align*}
where $F_{\b,p}(\vth)$ is equal to $\vth_p$ if $\b=0$, to $\delta_{p,r}$ if $\b=e_r$, and to 0 otherwise. We get from the proof of Lemma \ref{B1-lem} that (for $1\leq |\a'|\leq |\a|$) $P_{\a,\a'}(\psi_{z',z}^{\bfr',\bfr})(\rx)=\O(\langle \rx \rangle^{-\sigma(|\a|-|\a'|)})$. As a consequence, using (\ref{i-first}), we see that $\del^{\a} L_{i,p}(\rx)=\O(\langle \rx \rangle^{-\sigma|\a|})$. Thus, if $|\b|>1$, $\del^{(\a,\b)}\psi_* = 0$ and 
\begin{align*}
&\text{ if } \b=0\, , \quad \ (\del^{(\a,\b)}\psi_*)_i(\rx,\vth) = \O(\langle \rx \rangle^{-\sigma(|\a|-1)}) \quad  \text{ and }  \quad  (\del^{(\a,\b)}\psi_*)_{n+i}(\rx,\vth) = \O(\langle\rx\rangle^{-\sigma |\a|}\langle \vth \rangle )\, , \\
&\text{ if } |\b|=1\, , \quad (\del^{(\a,\b)}\psi_*)_i =0 \quad \hspace{3cm} \text{ and } \quad   (\del^{(\a,\b)}\psi_*)_{n+i}(\rx,\vth) = \O(\langle\rx\rangle^{-\sigma |\a|})\, .
\end{align*}
Similar results hold for $\psi_T$, the only difference is that we just have to take $\wt L:=(d\psi_{z',z}^{\bfr',\bfr})_{\rx}(\vth)$ instead of $L$.

\noindent We have for any $f\in C^\infty(T^*M,\mathfrak{E})$, 
$$
\del_{z,\bfr}^{\nu} (f) = \del^{\nu}(f\circ (n_{z,*}^\bfr)^{-1}) \circ n_{z,*}^\bfr = \del^{\nu}(f\circ (n_{z',*}^{\bfr'})^{-1} \circ \psi_*) \circ n_{z,*}^\bfr\, .
$$
Using again the Faa di Bruno formula in Theorem \ref{FaaCS}, we get 
$$
\del_{z,\bfr}^{\nu} =  \sum_{1\leq |\nu'|\leq |\nu|} (P_{\nu,\nu'}(\psi_*)\circ n_{z,*}^\bfr) \ \del_{z',\bfr'}^{\nu'} =:\sum_{1\leq |\nu'|\leq |\nu|} f_{\nu,\nu'} \ \del_{z',\bfr'}^{\nu'}
$$ 
where $P_{\nu,\nu'}(\psi_*)$ is a linear combination of terms of the form $\prod_{j=1}^{s} (\del^{l^j} \psi_*)^{k^j} $, where $1\leq s\leq |\nu|$,  
the $k^j$ and $l^j$ are $2n$-multi-indices with $|k^j|>0$, $|l^j|>0$, $\sum_{j=1}^s k^j = \nu'$ and $\sum_{j=1}^s|k^j|l^j=\nu$. 

Let us note $l^{j}=:(l^{j,1},l^{j,2})$, $k^{j}=:(k^{j,1},k^{j,2})$ where $l^{j,1},l^{j,2},k^{j,1},k^{j,2}$ are $n$-multi-indices. Thus,
 $$
(\del^{l^j}\psi_*)^{k^j}= \prod_{i=1}^n ((\del^{l^j}\psi_*)_i)^{k^{j,1}_i}\ ((\del^{l^j}\psi_*)_{n+i})^{k^{j,2}_i}
$$
and we get, for a given $s$, $(l^{j})$, $(k^j)$ such that $(\del^{l^j} \psi_*)^{k^j} \neq 0$ for all $1\leq j\leq s$,
\begin{align*}
&\text{ if } l^{j,2}=0\, ,  \qquad   (\del^{l^j}\psi_*)^{k^j} = \O(\langle \rx\rangle^{-\sigma(|l^{j}|-1)|k^{j}| - \sigma |k^{j,2}|}\langle \vth\rangle^{|k^{j,2}|})\, , \\ 
&\text{ if } |l^{j,2}|=1\, ,\qquad k^{j,1}=0 \text{ and } \    (\del^{l^j}\psi_*)^{k^j} = \O(\langle \rx \rangle^{-\sigma (|l^j|-1)|k^j|})\, . 
\end{align*}

Since $k^j\neq 0$ and $(\del^{l^j} \psi_*)^{k^j} \neq 0$, $l^{j,2}$ always satisfies $|l^{j,2}|\leq 1$.
By permutation on the $j$ indices, we can suppose that for $1\leq j\leq j_1-1$, we have $l^{j,2}=0$, for $j_1\leq j\leq s$, we have $|l^{j,2}|=1$, where $1\leq j_1 \leq s+1$. Thus,  
$$
\prod_{j=1}^s (\del^{l^j} \psi_*)^{k^j} = \O(\langle\rx \rangle^{-\sigma(\sum_{j=1}^{s}(|l^j|-1)|k^j| + \sum_{j=1}^{j_1-1}|k^{j,2}|)}\langle \vth \rangle^{\sum_{j=1}^{j_1-1}|k^{j,2}|})\, .
$$
Since, with $\nu=(\a,\b)$, $\nu'=(\a',\b')$, 
$$
\sum_{j=1}^{j_1-1} |k^{j,2}| = \sum_{j=1}^s |k^{j,2}| - \sum_{j=j_1}^{s} |k^{j,2}| = |\b'|-\sum_{j=j_1}^{s} |k^{j}||l^{j,2}| =|\b'|-|\b|\, , 
$$
(\ref{fabab}) follows. If we set $f_{0,0,0,0}:=1$ and $f_{\a,0,0,0}:=0$ if $\a\neq 0$, then for any $2n$-multi-index $(\a,\b)$,
$$
\del_{z,\bfr}^{(\a,\b)} = \underset{|\b'|\geq |\b|}{\sum_{0\leq |(\a',\b')|\leq |(\a,\b)|}}\, f_{\a,\b,\a',\b'}\,   \del_{z',\bfr'}^{(\a',\b')} 
$$
and the estimate (\ref{fabab}) holds for any $f_{\a,\b,\a',\b'}$. In the case of $\O_M$-bounded geometry, the proof is similar, and we obtain for a $r_\nu\geq 1$, $\prod_{j=1}^s (\del^{l^j} \psi_*)^{k^j} = \O(\langle\rx \rangle^{r_\nu}\langle \vth \rangle^{|\b'|-|\b|})$, which gives the result.

\noindent $(ii)$ Replacing $\psi_{*}$ by $\psi_{z',z,M^2}^{\bfr',\bfr}:=n_{z',M^2}^{\bfr'}\circ (n_{z,M^2}^\bfr)^{-1}$ in $(i)$, we obtain the result by similar arguments.
\end{proof}

\subsection{Basic function and distribution spaces}

We suppose in this section that $E$ is an hermitian vector bundle on the exponential manifold $(M,\exp)$. Recall that if $u\in C^\infty(M;E)$ (resp. $C^\infty_c(M;E)$) the Fr\'{e}chet space of smooth sections (resp. the $LF$-space of compactly supported smooth sections) of $E\to M$, we have for any $z\in M$, $u^z:=\tau_z^{-1} u \in C^\infty(M,E_z)$ (resp. $C^\infty_c(M,E_z)$). We define for any frame $(z,\bfr)$ on $M$,
$$
T_{z,\bfr}(u):= u^z \circ (n_z^\bfr)^{-1}.
$$
Thus, $T_{z,\bfr}$ sends sections of $E\to M$ to functions from $\R^n$ to $E_z$ and is in fact a topological isomorphism from $C^\infty(M;E)$ (resp. $C^\infty_c(M;E)$) onto $C^{\infty}(\R^n,E_z)$ (resp. $C^{\infty}_c(\R^n,E_z)$). 

In the following, a density (resp. a codensity) is a smooth section of the complex line bundle over $M$ defined by the disjoint union over $x\in M$ of the complex lines formed by the 1-twisted forms on $T_x M$ (resp. $T_x^*(M)$). Recall that a 1-twisted form on a $n$-dimensional vector space $V$ is a function on $F$ on $\Lambda_n V\backslash \{0\}$ such that 
$$
F(cv)= |c| F(v) \qquad \text{for all } v\in \Lambda_n V\backslash \{0\} \text{ and } c\in \R^*.
$$
For a given frame $(z,\bfr)$, let us note $|dx^{z,\bfr}|$ the density associated to the volume form on $M$: $dx^{z,\bfr}:=dx^{1,z,\bfr}\wedge \cdots \wedge dx^{n,z,\bfr}$ and $|\del_{z,\bfr}|$ the codensity defined as $|\del_{1,z,\bfr}\wedge \cdots \wedge \del_{n,z,\bfr}|$.

Any density (resp. codensity) is of the form $c|dx^{z,\bfr}|$ (resp. $c|\del_{z,\bfr}|$) where $c$ is a smooth function on $M$, and by definition is strictly positive if $c(x)>0$ for any $x\in M$.
 For a given strictly positive density $d\mu$, we also note by $d\mu$ its associated (positive, Borel--Radon, $\sigma$-finite) measure on $M$.  
This allows to define the following Banach spaces of (equivalence classes of) functions on $M$: $L^p(M,d\mu)$ ($1\leq p\leq \infty$). Actually, $L^{\infty}(M):=L^{\infty}(M,d\mu)$ does not depend on the chosen $d\mu$, since the null sets for $d\mu$ are exactly the null sets for any other strictly positive density $d\mu'$ on $M$.

For a given $z\in M$, we note $L^p(M,E_z,d\mu)$ ($1\leq p<\infty$) and $L^\infty(M,E_z)$ the Bochner spaces on $M$ with values in $E_z$.
$E_z$ is a hermitian complex vector space, so we can identify $E_z$ with its antidual $E'_z$. There is a natural anti-isomorphism between $E'_z$ and the dual of $E_z$ but there is in general no canonical way to identify $E_z$ with its dual with a {\it linear} isomorphism. 
Thus, we shall use antiduals rather than duals in the following. However, $E_z$ is anti-isomorphic with its dual by complex conjugaison on $E'_z$. We shall note $\ol x$ the image under this anti-isomorphism of $x\in E_z$ and $\ol E_z$ the dual of $E_z$.


We note 
$L^p(M;E,d\mu):=\set{\psi \text{ section of }E\to M\text{ such that }|\psi|^p\in L^1(M,d\mu)}/\sim_{a.e.}$ and 
$L^\infty(M;E):= \set{\psi \text{ section of }E\to M\text{ such that }|\psi|\in L^\infty(M) }/\sim_{a.e.}$
where $\sim_{a.e.}$ the standard ``almost everywhere" equivalence relation.
Since the $\tau_{xy}$ maps are isometries, for any $z\in M$, the map $\psi \to \tau_z^{-1} \psi$ defines linear isometries: $L^p(M;E,d\mu)\simeq L^p(M,E_z,d\mu)$, and $L^\infty(M;E)\simeq L^\infty(M,E_z)$.
In particular, $L^p(M;E,d\mu)$ and $L^\infty(M;E)$ are Banach spaces and $L^2(M;E,d\mu)$ a Hilbert space.
Moreover, we can define for any $\psi \in L^1(M;E,d\mu)$ and $z\in M$ the following Bochner integral $\int \tau_z^{-1}\psi \in E_z$. We can canonically identify $L^\infty(M; E)$ as the antidual of $L^1(M;E,d\mu)$ and 
$L^2(M;E,d\mu)$ as its own antidual. 
The (strong) antiduals of $C_c^\infty(M;E)$ and $C^\infty(M;E)$ are noted respectively $\D' (M;E)$ and $\E'(M;E)$. 

We define $G_{\sigma}(\R^p,\mathfrak{E})$ (resp. $S_\sigma(\R^p)$), where $\sg\in [0,1]$, as the space of smooth functions $g$ from $\R^{p}$ into $\mathfrak{E}$ (resp. $\R$) such that for any $p$-multi-index $\nu\neq 0$ (resp. any $p$-multi-index $\nu$), there exists $C_\nu>0$
such that $\norm{\del^{\nu} g (\rv)}\leq C_{\nu} \langle \rv \rangle^{-\sigma (|\nu|-1)}$ (resp. $|\del^{\nu} g (\rv)|\leq C_{\nu} \langle \rv \rangle^{-\sigma |\nu|}$) for any $\rv\in \R^{p}$. We note $\O_M(\R^p,\mathfrak{E})$ the space of smooth $\mathfrak{E}$-valued functions with polynomially bounded derivatives.
We use the shorcuts $G_\sigma(\R^p):=G_\sigma(\R^p,\R^p)$ and $\O_M(\R^p):=\O_M(\R^p,\R)$.

We have the following lemma which will give an equivalent formulation of $S_\sigma$ or $\O_M$-bounded geometry.

\begin{lem}
\label{Ssigma} 
(i) Let $f\in G_{\sigma}(\R^p,\mathfrak{E})$ (resp. $S_\sigma(\R^p)$) and $g\in G_\sigma(\R^{n},\R^p)$ such that, if $\sigma>0$, there exists $\eps>0$ such that $\langle g(\rv) \rangle \geq \eps \langle \rv \rangle$ for any $\rv\in \R^n$. Then $f\circ g \in G_\sigma(\R^n,\mathfrak{E})$ (resp. $S_\sigma(\R^n)$).

\noindent (ii) The set $G_\sigma^\times(\R^p)$ of diffeomorphisms $g$ on $\R^p$ such that $g$ and $g^{-1}$ are in $G_\sigma(\R^p)$ is a subgroup of $\mathrm{Diff}(\R^p)$ and contains $GL_p(\R)$ as a subgroup.  

\noindent (iii) We have $\O_M(\R^p,\mathfrak{E})\circ \O_M(\R^n,\R^p) \subseteq \O_M(\R^n,\mathfrak{E})$. In particular, the space $\O_M(\R^p,\R^p)$ is a monoid under the composition of functions. The set of inversible elements of the monoid $\O_M(\R^p,\R^p)$, noted $\O_M^\times(\R^p,\R^p)$, is a subgroup of $\mathrm{Diff}(\R^p)$ and contains $G_\sigma^\times(\R^p)$ as a subgroup. 

\noindent (iv)
$(M,\exp)$ has a $S_\sigma$ (resp. $\O_M$)-bounded geometry if and only if there exists a frame $(z_0,\bfr_0)$ such that for any frame $(z,\bfr)$, $\psi_{z_0,z}^{\bfr_0,\bfr}\in G_{\sigma}^\times(\R^n)$ (resp. $\O_M^\times(\R^n,\R^n)$).

\noindent (v) The set, noted $S_\sigma^\times(\R^p)$ (resp. $\O_M^\times(\R^p)$), of smooth functions $f:\R^p\to \R^*$ such that $f$ and $1/f$ are in $S_\sigma(\R^p)$ (resp. $\O_M(\R^p)$) is a commutative group under pointwise multiplication of functions. Moreover, $S_\sigma^\times(\R^p) \leq S_{\sigma'}^\times(\R^p) \leq \O_M^\times(\R^p)$ if $1\geq \sigma \geq \sigma'\geq 0$.  

\noindent (vi) If $g\in G_\sigma^\times (\R^p)$ (resp. $\O_M^\times(\R^p,\R^p)$) then its Jacobian determinant $J(g)$ is in $S_\sigma^\times(\R^p)$ (resp. $\O_M^\times(\R^p)$).
\end{lem}

\begin{proof} 
$(i)$ The Faa di Bruno formula yields for any $n$-multi-index $\nu\neq 0$,
\begin{equation}
\del^\nu (f\circ g) = \sum_{1\leq |\la|\leq |\nu|} (\del^\la f)\circ g\  P_{\nu,\la}(g) \label{delnufg}
\end{equation}
where $P_{\nu,\la}(g)$ is a linear combination (with coefficients independant of $f$ and $g$) of functions of the form $\prod_{j=1}^s (\del^{l^j} g)^{k^j}$ where $s\in \set{1,\cdots ,|\nu|}$.  
The $k^j$ are $p$-multi-indices and the $l^j$ are $n$-multi-indices (for $1\leq j\leq s$) such that $|k^j|>0$, $|l^j|>0$, $\sum_{j=1}^s k^j = \la$ and $\sum_{j=1}^s |k^j| l^j= \nu$. As a consequence, since $g\in G_\sigma(\R^{n},\R^p)$, for each $\nu,\la$ with $1\leq |\la|\leq |\nu|$ there exists $C_{\nu,\la}>0$ such that for any $\rv\in \R^n$,
\begin{equation}\label{pnulag1}
|P_{\nu,\la}(g) (\rv)|\leq C_{\nu,\la} \langle \rv \rangle^{-\sigma(|\nu|-|\la|)}\, . 
\end{equation}
Moreover, if $f\in G_\sigma(\R^p,\mathfrak{E})$ (resp. $S_\sigma(\R^p)$), there is $C'_\la>0$ such that for any $\rv\in \R^n$, $\norm{(\del^\la f) \circ g(\rv)}\leq C'_\la \langle \rv\rangle ^{-\sigma(|\la|-1)}$ (resp. $|(\del^\la f) \circ g(\rv)|\leq C'_\la \langle \rv\rangle ^{-\sigma|\la|}$). The result now follows from (\ref{delnufg}) and (\ref{pnulag1}).

\noindent $(ii)$ Let $f$ and $g$ in $G_\sigma^\times(\R^p)$. We have $\del_i g^{-1} = \O(1)$ for any $i\in \set{1,\cdots,p}$. Taylor--Lagrange inequality of order 1 entails that $\langle g^{-1}(v)\rangle =\O(\langle v\rangle)$ and thus there is $\eps>0$ such that $\langle g(\rv) \rangle \geq \eps \langle \rv \rangle$ for any $\rv\in \R^n$. With $(i)$, we get $f\circ g\in G_\sigma(\R^p)$. The same argument shows that $g^{-1}\circ f^{-1} \in G_\sigma(\R^p)$.

\noindent $(iii)$ Direct consequence of Theorem \ref{FaaCS}.

\noindent $(iv)$
The only if part is obvious. Suppose then that for any frame $(z,\bfr)$, $\psi_{z_0,z}^{\bfr_0,\bfr}\in G_{\sigma}^\times(\R^n)$ (resp. $\O_M^\times(\R^n,\R^n)$. Let $(z,\bfr)$, $(z',\bfr')$ be two frames. We have $\psi_{z,z'}^{\bfr,\bfr'}=\psi_{z,z_0}^{\bfr,\bfr_0}\circ \psi_{z_0,z'}^{\bfr_0,\bfr'}$. So, by $(ii)$ (resp. $(iii)$),  $\psi_{z,z'}^{\bfr,\bfr'}\in G_{\sigma}^\times(\R^n)$ (resp. $\O_M^\times(\R^n,\R^n)$), which yields the result. 

\noindent $(v)$ By Leibniz rule, the spaces $S_\sigma(\R^p)$ and $\O_M(\R^p)$ are $\R$-algebras under the pointwise product of functions. The result follows.

\noindent $(vi)$ Consequence of $(ii)$, $(iii)$, $1/J(g) = J(g^{-1})\circ g$ and the fact that $S_\sigma(\R^p)$ (resp. $\O_M(\R^p)$) is stable under the pointwise product of functions.
\end{proof}

Remark that for any $g \in G_\sigma^\times(\R^p)$, we have $g\asymp\Id_{\R^p}$.
The multiplication by a function in $\O_M^\times (\R^n)$ is a topological isomorphism from the Fr\'{e}chet space of rapidly decaying $E_z$-valued functions $\S(\R^n,E_z)$ onto itself. 
If we note $J_{z,z'}^{\bfr,\bfr'}$ the Jacobian of $\psi_{z,z'}^{\bfr,\bfr'}$, then $1/{J_{z,z'}^{\bfr,\bfr'}}=J_{z',z}^{\bfr',\bfr}\circ \psi_{z,z'}^{\bfr,\bfr'}$ and $J_{z,z'}^{\bfr,\bfr'}\circ n_{z'}^{\bfr'}(x)= dx^{z,\bfr}/dx^{z',\bfr'} (x)= \det M_{z,x}^\bfr ( M_{z',x}^{\bfr'})^{-1}=\det( M_{z',x}^{\bfr'})^{-1}  M_{z,x}^\bfr$.
We deduce from Lemma \ref{Ssigma} $(vi)$ that if $(M,\exp)$ has a $S_{\sigma}$ (resp. $\O_M$) bounded geometry then $J_{z,z'}^{\bfr,\bfr'}$ is in $S^\times_{\sigma}(\R^n)$ (resp. $\O_M^\times (\R^n)$). 

\begin{defn} Any smooth function $f$ is in $S_\sigma$ (resp. $\O_M$) if for any frame $(z,\bfr)$, $f\circ (n_z^\bfr)^{-1}\in S_\sigma(\R^n)$ (resp. $\O_M(\R^n)$). Similarly, any smooth function $f$ is in $S_\sigma^\times$ (resp. $\O_M^\times$) if for any frame $(z,\bfr)$, $f\circ (n_z^\bfr)^{-1}\in S_\sigma^\times(\R^n)$ (resp. $\O_M^\times(\R^n)$).
\end{defn}

\begin{lem}
\label{redSOM}
If $(M,\exp)$ has a $S_\sigma$-bounded geometry then a smooth function $f$ on $M$ is in $S_\sigma$ (resp. $S_\sigma^\times$) if there exists a frame $(z,\bfr)$ such that $f\circ (n_z^\bfr)^{-1}\in S_\sigma(\R^n)$ (resp. $S_\sigma^\times(\R^n)$). Similarly, If $(M,\exp)$ has a $\O_M$-bounded geometry then $f$ is in $\O_M$ (resp. $\O_M^\times$) if there exists a frame $(z,\bfr)$ such that $f\circ (n_z^\bfr)^{-1}\in \O_M(\R^n)$ (resp. $\O_M^\times(\R^n)$). 
\end{lem}
\begin{proof} Let $(z',\bfr')$ be a frame such that $f\circ (n_{z'}^{\bfr'})^{-1}\in S_\sigma(\R^n)$, and let $(z,\bfr)$ be another frame. By Lemma \ref{B1-lem} $(ii)$, if $(M,\exp)$ has a $S_\sigma$-bounded geometry then for any $n$-multi-index $\a$, 
$$
\del^\a (f\circ (n_z^\bfr)^{-1}) = \sum_{0\leq |\a'|\leq |\a|} f_{\a,\a'}\circ (n_{z}^{\bfr})^{-1} \, (\del^{\a'} f\circ (n_{z'}^{\bfr'})^{-1}) \circ \psi_{z',z}^{\bfr',\bfr}
$$
where $(f_{\a,\a'}\circ (n_{z}^{\bfr})^{-1})(\rx) = \O(\langle \rx\rangle^{-\sigma(|\a|-|\a'|)})$. As a consequence $\del^\a (f\circ (n_z^\bfr)^{-1})(\rx)=\O(\langle \rx\rangle^{-\sigma|\a|})$ and the result follows. The case of $\O_M$ bounded geometry is similar.
\end{proof}

\begin{defn}
A smooth strictly positive density $d\mu$ is a $S_\sigma^\times$-density (resp. $\O_M^\times$-density) if for any frame $(z,\bfr)$, the unique smooth strictly positive function $f_{z,\bfr}$ such that $d\mu = f_{z,\bfr} |dx^{z,\bfr}|$ is in $S_\sigma^\times$ (resp. $\O_M^\times$). In this case, we shall note $\mu_{z,\bfr}$ the smooth stricly positive function in $S_\sigma^\times(\R^n)$ (resp. $\O_M^\times(\R^n)$) such that $d\mu =(\mu_{z,\bfr}\circ n_z^\bfr)\, |dx^{z,\bfr}|$. 
\end{defn}

We shall say that $(M,\exp,d\mu)$ has a $S_\sg$ (resp. $\O_M$) bounded geometry if $(M,\exp)$ has a $S_\sg$ (resp. $\O_M$) bounded geometry  and $d\mu$ is a $S^\times_\sg$(resp. $\O^\times_M$) density.

\begin{lem} If $(M,\exp)$ has a $S_\sigma$ (resp. $\O_M$) bounded geometry then any density of the form $ u\circ n_{z'}^{\bfr'} |dx^{z,\bfr}|$ where $u$ is a smooth strictly positive function in $S_\sigma^\times(\R^n)$ (resp. $\O_M(\R^n)$) and $(z,\bfr)$, $(z',\bfr')$ are frames, is a $S_\sigma^\times$-density (resp. $\O_M^\times$-density).
\end{lem}
\begin{proof} Let $(z'',\bfr'')$ be an arbitrary frame. Noting $d\mu:=u\circ n_{z'}^{\bfr'} |dx^{z,\bfr}|$, we get $d\mu = (u\circ n_{z'}^{\bfr'})| J_{z,z''}^{\bfr,\bfr''}|\circ n_{z''}^{\bfr''} |dx^{z'',\bfr''}|$.  We already saw that the function $J_{z,z''}^{\bfr,\bfr''}$ is in $S^\times_\sigma(\R^n)$ (resp. $\O_M^\times(\R^n))$. By Lemma \ref{redSOM}, $(u\circ n_{z'}^{\bfr'})( |J_{z,z''}^{\bfr,\bfr''}|\circ n_{z''}^{\bfr''})$  is in $S_\sigma^\times$ (resp. $\O_M^\times$).
\end{proof}

\begin{rem} By taking $u:=\rx\mapsto 1$ in the previous lemma, we see that for any exponential manifold $(M,\exp)$ with $S_\sigma$ (resp. $\O_M$) bounded geometry, we can define a canonical family of $S^\times_\sigma$-densities (resp. $\O_M^\times$-densities) on $M$: $\mathcal{D}:=(|dx^{z,\bfr}|)_{(z,\bfr)\in I}$ where $I$ is the set of frames on $M$. If the map $\exp$ is the exponential map associated to a pseudo-Riemannian metric $g$ on $M$, we can also define a canonical subfamily of $\mathcal{D}$ by $\mathcal{D}_g:=(|dx^{z}|)_{z\in M}$ where $|dx^z|:=|dx^{z,\bfr}|$ with $\bfr$ any orthonormal basis (in the sense $g_z(\bfr_i,\bfr_j)=\eta_i\delta_{i,j}$ where $\eta_i=1$ for $1\leq i \leq m$ and $\eta_i=-1$ for $i>m$, where $g$ has signature $(m,n-m)$) of $T_z(M)$ ($|dx^z|$ is then independant of $\bfr$). A priori, the Riemannian density does not belong to the canonical $M$-indexed family $\mathcal{D}_g$. 
\end{rem}

We shall need integrations over tangent and cotangent fibers and manifolds. We thus define $d\mu^*:=(\mu^{-1}_{z,\bfr}\circ n_z^\bfr)\,|\del_{z,\bfr}|$ the codensity associated to $d\mu$, where $\mu^{-1}_{z,\bfr}:=\tfrac{1}{\mu_{z,\bfr}}$ and 
$(z,\bfr)$ is a frame. Note that since $|\del_{z,\bfr}|/|\del_{z',\bfr'}|=|dx^{z',\bfr'}|/|dx^{z,\bfr}|= (\mu_{z,\bfr}\circ n_{z}^{\bfr})/(\mu_{z',\bfr'}\circ n_{z'}^{\bfr'})$, $d\mu^*$ is independant of $(z,\bfr)$.
For a given $x\in M$, the density on $T_x(M)$ associated to $d\mu$ is $d\mu_x:=(\mu_{z,\bfr}\circ n_z^\bfr(x))\, |dx^{z,\bfr}_x|$ and the associated density on $T_x^*(M)$ is $d\mu_x^*:=(\mu^{-1}_{z,\bfr}\circ n_z^\bfr(x))\,|{\del_{z,\bfr}}_x|$. For a function $f$ defined on $T_x(M)$ or $T^*_x(M)$, we have formally:
\begin{align*}
&\int_{T_x(M)} f(\xi)\, d\mu_x(\xi) = \mu_{z,\bfr}\circ n_z^\bfr(x)\int_{\R^n}f\circ (  M_{z,x}^{\bfr})^{-1}(\zeta)\, d\zeta \,, \\
&\int_{T_x^*(M)} f(\th)\, d\mu_x^*(\th) = \mu_{z,\bfr}^{-1}\circ n_z^\bfr(x)\int_{\R^n}f\circ (\wt M_{z,x}^{\bfr})^{-1}(\vth)\, d\vth\, ,
\end{align*}
and it is straightforward to check that these integrals are independant of the chosen frame $(z,\bfr)$.

\subsection{Schwartz spaces and operators}

\begin{assum} We suppose in this section and in section \ref{fouriersec} that $(M,\exp,E,d\mu)$, where $E$ is an hermitian vector bundle on $M$, has a $\O_M$-bounded geometry.
\end{assum}

The main consequence of the exponential structure is the possibility to define Schwartz functions on $M$.

\begin{defn} A section $u\in C^\infty(M,E)$ is rapidly decaying at infinity if for any $(z,\bfr)$, any $n$-multi-index $\a$ and $p\in \N$, there exists $K_{\a,p}>0$ such that the following estimate
\begin{equation}\label{SB-est}
\norm{\del_{z,\bfr}^{\a} u^z (x)}_{E_z}< K_{\a,p} \langle x\rangle_{z,\bfr}^{-p} 
\end{equation}
holds uniformly in $x\in M$. We note $\S(M,E)$ the space of such sections.
\end{defn}

With the hypothesis of $\O_M$-bounded geometry, we see that the requirement ``any $(z,\bfr)$" can be reduced to a simple existence:

\begin{lem}
\label{indepzbS}
A section $u\in C^\infty(M,E)$ is in $\S(M,E)$ if and only if there exists a frame $(z,\bfr)$ such that (\ref{SB-est}) is valid. 
\end{lem}
\begin{proof} Suppose that (\ref{SB-est}) is valid for $(z',\bfr')$ and let $(z,\bfr)$ be another frame.  Thus, with Lemma \ref{B1-lem} $(ii)$ and  Leibniz rule,
\begin{align}\label{del-coord-ch}
\del_{z,\bfr}^\a u^z(x)  =  \sum_{0\leq |\a'|\leq |\a| }\, \sum_{\b\leq \a' }\ \ f_{\a,\a'} \,\tbinom{\a'}{\b}\,\del_{z',\bfr'}^{\a'-\b} (\tau_z^{-1}\tau_{z'})\ \del_{z',\bfr'}^{\b} u^{z'}(x).
\end{align}
Moreover $|f_{\a,\a'} \,\tbinom{\a'}{\b}\,\del_{z',\bfr'}^{\a'-\b} (\tau_z^{-1}\tau_{z'})| \leq C_{\a} \langle x \rangle_{z,\bfr}^{q_\a}$ for a $C_\a>0$ and a $q_\a\geq 1$. Now (\ref{SB-est}) and (\ref{i-first-om}) entail that for any $p\in \N$, there is $K>0$ such that
$\norm{\del_{z,\bfr}^\a u^z(x)}_{E_z} \leq K \langle x \rangle_{z,\bfr}^{-p}$. The result follows.
\end{proof}

\begin{rem} Let $u\in C^\infty(M,E)$ and $(z,\bfr)$ a frame. Then $u\in \S(M,E)$ if and only if $(\tau_z^{-1} u)\circ (n_{z}^{\bfr})^{-1} \in \S(\R^n,E_z)$. In other words, if $v\in \S(\R^n,E_z)$ then $\tau_z(v\circ n_z^\bfr) \in \S(M,E)$. 
\end{rem}

The following lemma shows that we can define canonical Fr\'{e}chet topologies on $\S(M,E)$.

\begin{lem} Let $(z,\bfr)$ a frame. Then 

\noindent (i) The following set of semi-norms: 
$$
q_{\a,p}(u):= \sup_{x\in M} \langle x \rangle_{z,\bfr}^{p} \norm{\del_{z,\bfr}^{\a} u^z (x)}_{E_z}\,.
$$
defines a locally convex metrizable topology $\mathcal{T}$ on $\S(M,E)$.

\noindent (ii) The application $T_{z,\bfr}$ is a topological isomorphism from the space $\S(M,E)$ onto $\S(\R^n,E_z)$.

\noindent (iii) The topology $\mathcal{T}$ is Fr\'{e}chet and independent of the chosen frame $(z,\bfr)$.
\end{lem}
\begin{proof}
$(i)$ and $(ii)$ are obvious.

\noindent $(iii)$ Since $T_{z,\bfr}$ is a topological isomorphism, $\mathcal{T}$ is complete. Following the arguments of the proof of Lemma \ref{indepzbS}, we see that there is $r\in \N^*$ such that for any $n$-multi-index $\a$ and $p\in \N$, there exist $C_{\a,p}>0$, $r_{\a,p}\in \N^*$, such for any $u\in \S(M,E)$,
$$
q_{\a,p}^{(z,\bfr)} (u) \leq C_{\a,p} \sum_{|\b|\leq |\a|} q^{(z',\bfr')}_{\b,r_{\a,p}}(u)\, .
$$
The independance on $(z,\bfr)$ follows.
\end{proof}

\begin{rem}
\label{BME}
If $(M,\exp,E,d\mu)$ has a $S_0$-bounded geometry, then it is possible to define the Fr\'{e}chet space of smooth sections with bounded derivatives $\B(M,E)$ by following the same procedure of $\S(M,E)$, with Lemma \ref{B1-lem}.
\end{rem}

Classical results of distribution theory \cite{Treves} and the previous topological isomorphisms $T_{z,\bfr}$ entail the following diagrams of continuous linear injections ($(M;E)$ ommitted and $1\leq p<\infty$):

\[\xymatrix{
    C_c^\infty \ar[r]\ar[d] & \S \ar[r]\ar[d]  & C^\infty \ar[d]\\
    \E' \ar[r] &  \S' \ar[r] & \D'\\
  } \hskip1cm
\xymatrix{
     & \B \ar[rr]    &&  L^\infty \ar[rd] & \\
    \S \ar[rr]\ar[ru] & & L^p(d\mu) \ar[rr] && \S' 
  \ .} \] 
  
The injections $\S\to \B\to L^\infty$ are valid in the case where $M$ has a $S_0$-bounded geometry. In the case of a general $\O_M$-bounded geometry, only the injection $\S\to L^\infty$ holds a priori.  
The injection from functions into distribution spaces is given here by $u\mapsto \langle u,\cdot \rangle$ where $\langle u,v\rangle :=\int (u | v)\, d\mu$. Note that the following continuous injections $\S \to \S' $ and $ \S\to L^{p}(d\mu) \to \S'$, $(1\leq p<\infty)$ have a dense image.

Using the same principles of the definition of $\S$ together with the $\O_M$-bounded geometry hypothesis and Lemma \ref{cchange} $(ii)$, we define the Fr\'{e}chet space $\S(M\times M, L(E))$ such that for any $(z,\bfr)$ the applications $T_{z,\bfr,M^2}:=K\mapsto  K^z \circ (n_{z,M^2}^\bfr)^{-1}$
are topological isomorphisms from $\S(M\times M , L(E))$ onto $\S(\R^{2n},L(E_z))$. Noting $j_{M^2}$ the continous dense injection from $\S(M\times M,L(E))$ into its antidual $\S'(M\times M,L(E))$ defined as $\langle j_{M^2}(K),K'\rangle= \int_{M\times M} \Tr(K(x,y) (K'(x,y))^*)\,d\mu\ox d\mu(x,y)$, we have the following commutative diagram, where $j$ is the classical continuous dense inclusion from $\S(\R^{2n},L(E_z))$ into its antidual, and $M_{\mu\otimes\mu}$ is the multiplication operator from $\S(\R^{2n},L(E_z))$ onto itself by the $\O_M^\times(\R^{2n})$ function $\mu_{z,\bfr}\otimes \mu_{z,\bfr}$:

\[\xymatrix{
    \S(M\times M,L(E)) \ar[rr]^{j_{M^2}} \ar[d]_{T_{z,\bfr,M^2}} &  \ar[r] &\S'(M\times M,L(E))  \\
    \S(\R^{2n},L(E_z)) \ar[r]_{M_{\mu\ox\mu}} &  \S(\R^{2n},L(E_z)) \ar[r]_j & \S'(\R^{2n},L(E_z))\ . \ar[u]^{T^*_{z,\bfr,M^2}}   \\
  }\]

\noindent Since $\S$ is nuclear, $L(\S,\S')\simeq \S'(M\times M,L(E))$ and $\S(M\times M, L(E)) \simeq \S\ \wox\  \ol\S$ where $\ol \S :=\S(M,\ol E)$.  
Thus, $\S'(M\times M,L(E))\simeq \S'\wox\ \ol\S'$, where $\ol \S'$ is the dual of $\S$ which is also the antidual of $\ol \S$. Note that the isomorphism $L(\S,\S')\simeq \S'(M\times M,L(E))$ is given by 
$$
\langle A_K (v) ,u\rangle = K(u\otimes \ol v) 
$$
where $A_K$ is operator associated to the kernel $K$, $u,v\in \S$, and $\ol v(y):= \ol{v(y)}$. Formally, 
$$
\langle A_K (v), u\rangle= \int_{M\times M}(K(x,y) v(y) | u(x)) d\mu\otimes d\mu (x,y)\, ,  \qquad (A_K v)(x)= \int_M K(x,y) v(y) d\mu(y).
$$  
Thus any continuous linear operator $A:\S \to \S'$ is uniquely determined by its kernel $K_A \in  \S'(M\times M , L(E))$.
The transfert of $A$ into the frame $(z,\bfr)$ is the operator $A_{z,\bfr}$ from $\S(\R^n,E_z)$ into $\S'(\R^{n},E_z)$ such that 
$$
\langle A_{z,\bfr}(v),u\rangle := \langle A(T_{z,\bfr}^{-1}(v)), T_{z,\bfr}^{-1}(u) \rangle.
$$
Thus, if $K_A$ is the kernel of $A$, we have $K_{A_{z,\bfr}}:= \wt T_{z,\bfr,M^2}(K_A)$ as the kernel of $A_{z,\bfr}$, where $\wt T_{z,\bfr,M^2}$ here is the inverse of the adjoint of $T_{z,\bfr,M^2}$. $\wt T_{z,\bfr,M^2}$ is a topological isomorphism from $\S'(M\times M , L(E))$ onto  $\S'(\R^{2n} , L(E_z))$.

\begin{defn} An operator $A\in L(\S,\S')$ is regular if $A$ and its adjoint $A^\dag$ send continously $\S$ into itself. An isotropic smoothing operator is an operator with kernel in $\S(M\times M,L(E))$. The space regular operators and the space of isotropic smoothing operators are respectively noted $\Re(\S)$ and $\Psi^{-\infty}$.
\end{defn}

Note that this definition of isotropic smoothing operators differs from the standard smoothing operators one where only local effects are taken into account, since in this case, a smoothing operator is just an operator with smooth kernel.  Clearly, $A$ is regular if and only if for any frame $(z,\bfr)$, $A_{z,\bfr}$ is regular as an operator from $\S(\R^{n},E_z)$ into $\S'(\R^{n},E_z)$. Remark that the space of regular operators  forms a $*$-algebra under composition and the space of isotropic smoothing operators $\Psi^{-\infty}$ is a $*$-ideal of this algebra.

Let us record the following important fact:

\begin{prop} Any isotropic smoothing operator extends (uniquely) as a Hilbert--Schmidt operator on $L^2(d\mu)$.
\end{prop}
\begin{proof} An isotropic smoothing operator $A$ (with kernel $K$) extends as a continous linear operator from $\S'$ to $\S$, and thus it also extends as a bounded operator on $L^2(d\mu)$. Let $(z,\bfr)$ be a frame. If $U$ is the unitary associated to the isomorphism $L^2(d\mu)$ onto $\H_{z,\bfr}:=L^2(\R^n,E_z,\mu_{z,\bfr}\, dx)$ we have $A = U^* A_{z,\bfr} U$ where $A_{z,\bfr}$ is a bounded operator on $\H_{z,\bfr}$ given by the kernel $K^z\circ (n_z^\bfr,n_z^\bfr)^{-1}$. Since this kernel is in $\H_{z,\bfr}\otimes \ol \H_{z,\bfr} = L^2(\R^{2n},E_z\otimes \ol E_z, (\mu_{z,\bfr}\, dx)^{\ox 2} )$, it follows that $A_{z,\bfr}$ is Hilbert--Schmidt on $\H_{z,\bfr}$, which gives the result. 
\end{proof}

\subsection{Fourier transform}\label{fouriersec}

Fourier transform is the fundamental element that will allow the passage from operators to their symbols. 
In our setting, it is natural to extend the classical Fourier transform on $\R^n$ to Schwartz spaces of rapidly decreasing sections 
on the tangent and cotangent bundles of $M$, and use the fibers $T_x(M)$, $T_x^*(M)$ as support of integration. 

\begin{defn} A smooth section $a \in C^\infty(T^*M,L(E))$ is in $\S(T^*M,L(E))$ if for any $(z,\bfr)$, any $2n$-multi-index $\nu$ and any $p\in \N$, there exists $K_{p,\nu}>0$ such that 
\begin{equation}
\label{STM-est}
\norm{\del_{z,\bfr}^{\nu} a^z(x,\th)}_{L(E_z)}\leq K_{p,\nu} \langle x,\th \rangle_{z,\bfr}^{-p}
\end{equation}
uniformly in $(x,\th)\in T^*M$. A similar definition is set for $\S(TM,L(E))$.
\end{defn}

Following the same technique as for the space $\S(M,E)$, using the coordinate invariance given by Lemma \ref{cchange} we obtain the
\begin{prop} 
(i)  A section $u\in C^\infty(T^*M,L(E))$ is in $\S(T^*M,L(E))$ if and only if there exists a frame $(z,\bfr)$ such that (\ref{STM-est}) is valid. A similar property holds for $\S(TM,L(E))$.

\noindent (ii) There is a Fr\'{e}chet topology on $\S(T^*M,L(E))$ such that each 
$$
T_{z,\bfr,*}: a \mapsto a^z \circ (n_{z,*}^{\bfr})^{-1}
$$
is a topological isomorphism from $\S(T^*M,L(E))$ onto $\S(\R^{2n},L(E_z))$. A similar property holds for $\S(TM,L(E))$ and the applications $T_{z,\bfr,T}:=a\mapsto a^z \circ (n_{z,T}^\bfr)^{-1}$.
\end{prop}
\begin{proof}
\noindent $(i,ii)$ 
Suppose that (\ref{STM-est}) is valid for $(z',\bfr')$ and $a\in C^\infty(T^*M,L(E))$ and let $(z,\bfr)$ another frame.  With Lemma \ref{cchange} and  Leibniz rule, noting $\nu=(\a,\b)$, $\nu'=(\a',\b')$, $\la=(\la^1,\la^2)$ and $\rho=(\rho^1,\rho^2)$, we get
\begin{align}\label{del-cchange}
\del_{z,\bfr}^{\nu} a^z  =  \underset{|\b'|\geq |\b| }{\sum_{0\leq |\nu'|\leq |\nu|}} \sum_{\rho\leq \la \leq \nu'}\ \ f_{\nu,\nu'} C_{\nu',\la,\rho}\,\del_{z',\bfr'}^{\a'-\la^1}(\tau_z^{-1}\tau_{z'})\, \del_{z',\bfr'}^{(\rho^1,\b')}(a^{z'}) \,\del_{z',\bfr'}^{\la^1-\rho^1}(\tau_{z'}^{-1}\tau_{z})\,
\end{align}
where $C_{\nu',\la,\rho}= \delta_{\b',\la^2}\delta_{\b',\rho^2} \tbinom{\nu'}{\la} \tbinom{\la}{\rho}$. Using now the fact that for any $\rx,\vth \in \R^n$,
$\langle \rx \rangle^{1/2} \langle \vth \rangle^{1/2} \leq \langle(\rx,\vth)\rangle \leq \langle  \rx \rangle \langle \vth \rangle$, and (\ref{i-first-om}), (\ref{i-bis-om}), we see that for any $2n$-multi-index $\nu$, and $p\in \N$, there is $r_{\nu,p}\in \N^*$ and $C_{\nu,p}>0$ such that $q_{\nu,p}^{(z,\bfr)}(a) \leq C_{\nu, p} \sum_{|\rho|\leq |\nu|} q^{(z',\bfr')}_{\rho,r_{\nu,p}}(a)$, where 
$$
q_{\nu,p}^{(z,\bfr)}(a):=  \sup_{(x,\th)\in T^*M} \langle x,\th \rangle_{z,\bfr}^{p} \norm{\del_{z,\bfr}^{\nu} a^z (x,\th)}_{L(E_z)}\, .
$$
The results follow, as in the case of $\S(M,E)$, by taking the topology given by the seminorms $q_{\nu,p}^{z,\bfr}$ for an arbitrary frame $(z,\bfr)$.
\end{proof}

\begin{rem} If $(M,\exp,E)$ has a $S_0$-bounded geometry, we saw in Remark \ref{BME} that a coordinate free (independant of the frame $(z,\bfr)$) definition of a space of smooth $E$-sections on $M$ with bounded derivatives is possible. However, a similar definition cannot be given in the same manner for $L(E)$-sections on $TM$ or $T^*M$ with bounded derivatives, due to the fact that the change of coordinates of Lemma \ref{cchange} impose an increasing power of $\langle\th \rangle$ (when $|\b'|>|\b|$). However, the independance over $(z,\bfr)$ would still hold for the space of smooth sections of $L(E)\to T^*M$ (resp. $TM$) with polynomially bounded derivatives.
\end{rem}

We note $\S'(T^*M,L(E))$ and $\S'(TM,L(E))$ the strong antiduals of $\S(T^*M,L(E))$ and $\S(TM,L(E))$, respectively.
We have the following continuous inclusion with dense image 
$$
j_{T^*M}:\S(T^*M,L(E))\to \S'(T^*M,L(E)) \qquad \text{\big(resp. } j_{TM}:\S(TM,L(E)) \to \S'(TM,L(E))\big)
$$
defined by
$$
\langle j_{T^*M}(a),b\rangle := \int_{TM^*} \Tr(ab^*) d\mu^*\qquad \text{\big(resp. }\langle j_{TM}(a),b\rangle := \int_{TM} \Tr(ab^*) d\mu^T \,\big)
$$
where $d\mu^*$ is the measure on $T^*M$ given by $d\mu^*(x,\th):= d\mu^*_x(\th)d\mu(x)$ and $d\mu^T$ is the measure on $TM$ given by $d\mu^T(x,\xi):= d\mu_x(\xi)d\mu(x)$. Note that for any $(z,\bfr)$, $d\mu^*(x,\th)= |{\del_{z,\bfr}}_x|(\th) |dx^{z,\bfr}|(x)$ (this is the Liouville measure on $T^*M$) and $d\mu^T(x,\th)= \mu^2_{z,\bfr}\circ n_z^\bfr(x)|dx^{z,\bfr}_x|(\xi) |dx^{z,\bfr}|(x)$. We have the following commutative diagram, where $M_{\mu^2}$ is the multiplication operator by the $\O_M^\times(\R^{2n})$ function $(\rx,\zeta)\mapsto \mu_{z,\bfr}^2(\rx)$,  
\[\xymatrix{
    \S(TM,L(E)) \ar[rr]^{j_{TM}} \ar[d]_{T_{z,\bfr,T}} &  \ar[r] &\S'(TM,L(E))  \\
    \S(\R^{2n},L(E_z)) \ar[r]_{M_{\mu^2}} &  \S(\R^{2n},L(E_z)) \ar[r]_j & \S'(\R^{2n},L(E_z)) \ar[u]_{T^*_{z,\bfr,T}}   \\
  }\]
and, in the case of $\S(T^*M,L(E))$ a similar diagram is valid if $M_{\mu^2}$ is replaced by the identity.
\begin{defn} The Fourier transform of $a\in \S(TM,L(E))$ is
$$
\F(a) : (x,\th)\mapsto \int_{T_x(M)} e^{-2\pi i \langle \th,\xi \rangle }\, a(x,\xi)\, d\mu_x(\xi)\, .
$$ 
\end{defn}

\begin{prop} $\F$ is a topological isomorphism from $\S(TM,L(E))$ onto $\S(T^*M,L(E))$ with inverse 
$$
\ol \F (a):= (x,\xi) \mapsto \int_{T^*_x(M)} e^{2\pi i \langle \th,\xi \rangle}\, a(x,\th)\, d\mu^*_x(\th)\, .
$$ 
The adjoint $\ol \F^*$ of $\ol \F$ coincides with $\F$ on $\S(TM,L(E))$, so we still note $\ol \F^*$ by $\F$ and $\F^*$ by $\ol\F$.
\end{prop}
\begin{proof} Let $(z,\bfr)$ be a frame. It is straightforward to check that the following diagram commutes
\[\xymatrix{
    S(TM,L(E)) \ar[r]^{\F} \ar[d]_{T_{z,\bfr,T}} & \S(T^*M,L(E)) \\
    \S(\R^{2n},L(E_z)) \ar[r]_{\F_{z,\bfr}} &  \S(\R^{2n},L(E_z))\ar[u]_{T_{z,\bfr,*}^{-1}}  \\
  }\]
where $\F_{z,\bfr}= \F_P \circ M_{\mu} = M_{\mu}\circ \F_P$, with $M_{\mu}$ the multiplication operator on $\S(\R^{2n},L(E_z))$ defined by $M_{\mu}(a):=(\rx,\zeta)\mapsto \mu_{z,\bfr}(\rx)\, a(\rx,\zeta)$ and $\F_P$ the partial Fourier transform on the space $\S(\R^{2n},L(E_z))$ (only the variables in the second copy of $\R^{n}$ in $\R^{2n}$ being Fourier transformed). It is clear that $\F_{z,\bfr}$ is a topological isomorphism from $\S(\R^{2n},L(E_z))$ onto itself with inverse $\F_{z,\bfr}^{-1}=M_{1/\mu}\circ \ol \F_P$.
The fact that $\ol \F^*$ coincides with $\F$ on $\S(TM,L(E))$ is a consequence of the following equality 
$$
\int_{TM} \Tr(a (\ol \F(b)) ^*)\,d\mu^T = \int_{T^*M}\Tr(\F(a)b^*)\,d\mu^*
$$
for any $a\in \S(TM,L(E))$ and $b\in \S(T^*M,L(E))$, that is a direct consequence of the Parseval formula for $\F_P$.
\end{proof}

\section{Linearization and symbol maps}

\subsection{Linearization and the $\Phi_{\la}$, $\Ups_t$ diffeomorphisms}

Recall that a linearization (Bokobza-Haggiag \cite{Bokobza}) on a smooth manifold $M$ is defined as a smooth map $\nu$ from $M\times M$ into $TM$ such that $\pi \circ \nu = \pi_1$, $\nu(x,x)=0$ for any $x\in M$ and $(d_y\nu)_{y=x}=\Id_{T_x M}$. Using this map, it is then possible by restricting $\nu$ on a small neighborhood of the diagonal of $M\times M$, to obtain a diffeomorphism onto a neighborhood of the zero section of $TM$ and obtain an isomorphism between symbols (with a local control of the x variables on compact) and pseudodifferential operators modulo smoothing ideals. These isomorphisms depend on the linearization, as shown in \cite[Proposition V.3]{Bokobza}. We follow here the same idea, with a global point of view, since we are interested in the behavior at infinity. We thus consider, on the exponential manifold $(M,\exp,E,d\mu)$ a fixed linearization $\ol \psi$ that comes from an (abstract) exponential map $\psi$ on $M$ (also called linearization map in the following), so that $\ol \psi(x,y) = \psi_x^{-1} y$, and $\psi_x$ is a diffeomorphism from $T_x M$ onto $M$, with $\psi_x (0)=x$, $(d\psi_x)_0=\Id_{T_{x} M}$. For example, $\psi$ may be the exponential map $\exp$.

Let $\la \in [0,1]$ and $\Phi_\la$ be the smooth map from $TM$ onto $M\times M$ defined by 
$$
\Phi_\la: (x,\xi) \mapsto \big(\psi_x(\la \xi), \psi_x(-(1-\la)\xi) \big)\, .
$$
\begin{assum}
\label{assumH}
We suppose from now on that whenever the parameters $\la$, $\la'$, are in $]0,1[$, it is implied that the linearization map $\psi$ satisfies for any $x,y\in M$ and $t\in \R$, $\psi_{x}(t\psi_x^{-1}(y))= \psi_{y}((1-t)\psi_y^{-1}(x))$. This hypothesis, called $(H_\psi)$ in the following, is automatically satisfied if the linearization is derived from a exponential map of a connection on the manifold. 
\end{assum}
A computation shows that $\Phi_\la$ is a diffeomorphism with the following inverse $\Phi_\la^{-1}:(x,y)\mapsto \a'_{yx}(1-\la)$ for $\la\neq 0$ and $\Phi_0^{-1}(x,y):\mapsto -\a'_{xy}(0)$, where $\a_{xy}(t):=\psi_x(t\psi_x^{-1}(y))$.
Noting $\Phi_\la^{-1}(x,y)=:(m_\la(x,y),\xi_\la(x,y))$, we see that $m_\la(x,y)=\a_{xy}(\la)$ and, if $\la\neq 0$, $\xi_{\la}(x,y)=\tfrac{1}{\la}\psi_{m_{\la}(x,y)}^{-1}(x)$, while $\xi_0(x,y)=-\psi_x^{-1}(y)$.
In all the following, we shall use the symbol $W$ (for Weyl) for the value $\la=\half$, so that $m_W:=m_{\half}$, $\Phi_W:=\Phi_\half$, and similar conventions for the other mathematical symbols containing $\la$.  
Note that $m_\la$ is a smooth function from $M\times M$ onto $M$, with $m_\la(x,x)=x$ for any $x\in M$. Moreover, for any $x,y\in M$, $m_{\la}(x,y)=m_{1-\la}(y,x)$, $m_W(x,y)=m_W(y,x)$ (the ``middle point`` of $x$ and $y$), $\xi_{\la}(x,y)=-\xi_{1-\la}(y,x)$, $\xi_W (x,y)=-\xi_W (y,x)$ and $x\mapsto \Phi_\la^{-1} (x,x)$ is the zero section of $TM\to M$. Noting $j$ the involution on $M\times M$ : $(x,y)\mapsto (y,x)$, we have $\Phi_\la = j\circ \Phi_{1-\la}\circ -\Id_{TM}$. 

For any $t\in [-1,1]$ (with the convention that if $(H_\psi)$ is not satisfied, we are restricted to $t\in \set{-1,0,1}$), we define, 
$$
\Ups_t :(x,\xi)\mapsto \big(\psi_x (t\xi), \tfrac{-1}{t} \psi^{-1}_{\psi_x(t\xi)}(x)\big)
$$
with the convention $\tfrac{-1}{t} \psi^{-1}_{\psi_x(t\xi)}(x):= \xi$ if $t=0$, so that $\Ups_0=\Id_{TM}$.
A computation shows that $\Ups_{t}^{-1}=\Ups_{-t}$. The $\Phi_\la$ and $\Ups_t$ diffeomorphisms are related by the following property: for any $\la,\la'\in [0,1]$, $\Phi_{\la}^{-1}\circ \Phi_{\la'} = \Ups_{\la'-\la}$. We will use the shorthand $\Ups_{t,T}(x,\xi):=\tfrac{-1}{t} \psi^{-1}_{\psi_x(t\xi)}(x)$, so that $\Ups_{t}=(\psi\circ\ t\Id_{TM}, \Ups_{t,T} ).$

\begin{rem}
\label{Hhyp}Note that $(H_\psi)$ entails that
$(\Ups_t)_{t\in \R}$ is a one parameter subgroup of $\Diff(TM)$. 
\end{rem}

\begin{rem}
\label{remTP}
Suppose that $\psi$ is the exponential map associated to a connection $\nabla$ on $TM$, and $\a_{x,\xi}$ the unique maximal geodesic such that $\a'_{x,\xi}(0)=(x,\xi).$ It is a standard result of differential geometry (see for instance \cite[Theorem 3.3, p.206]{Lang}) that for any $v:=(x,\eta)\in TM$, and $\xi\in T_x(M)$, there exists an unique curve $\beta_{v}^\xi:\R \to TM$ such that 
$\nabla_{\a'_{v}} \beta_{v}^\xi =0$, $\pi\circ \beta_{v}^{\xi} = \a_{v}$ (in other words, $\beta_{v}^\xi$ is $\a_{v}$-parallel lift of $\a_{v}$)
and $\beta_{v}^\xi(0) = (x,\xi)$. By definition of geodesics, $\beta^{\eta}_{x,\eta}=\a'_{x,\eta}$. Moreover, $\beta_{x,\eta}^{\xi}(1)\in T_{\psi_x^\eta}(M)$, so we can define the following linear isomorphism of tangent fibers:
$P_{x,\eta}\ :\ T_x(M) \to T_{\psi_x^\eta}(M),   \quad \xi\mapsto \beta_{x,\eta}^\xi(1) \, .$
Note that $P_{x,\eta}^{-1}= P_{\psi_x^\eta,\psi_{\psi_x^\eta}^{-1}(x)}= P_{-\Ups_{1}(x,\eta)} = P_{\Ups_{-1}(x,-\eta)}$. The $P_{x,\xi}$ are the parallel transport maps along geodesics on the tangent bundle. These maps are related to the $\Ups_t$ diffeomorphisms, since a computation shows that for any $(x,\eta)\in TM$ and $t\in \R$,  $P_{x,t\eta}(\eta) = \Ups_{t,T}(x,\eta)$.
\end{rem}

If $(z,\bfr)$ is a frame, we define $\Phi_{\la,z,\bfr}:= n_{z,M^2}^{\bfr}\circ \Phi_\la \circ (n_{z,T}^\bfr)^{-1}$ and we note $J_{\la,z,\bfr}$ its Jacobian. We also define $\Ups_{t,z,\bfr}=n_{z,T}^{\bfr}\circ \Ups_{t}\circ (n_{z,T}^{\bfr})^{-1}$ and the smooth maps from $\R^{2n}$ to $\R^{n}$: 
\begin{align*}
&\psi_z^\bfr : (\rx,\zeta) \mapsto n_z^\bfr \circ \psi \circ\, (n_{z,T}^\bfr)^{-1}(\rx,\zeta)\, ,\\
&\ol{\psi_z^\bfr}: (\rx,\ry) \mapsto  M_{z,(n_z^\bfr)^{-1}(\rx)}^\bfr\circ \psi^{-1}_{(n_z^\bfr)^{-1}(\rx)}\circ (n_{z}^\bfr)^{-1}(\ry).
\end{align*}
Noting $\psi_{z,\rx}^{\bfr}(\zeta):=\psi_{z}^\bfr(\rx,\zeta)$ and $\ol{\psi_{z,\rx}^{\bfr}}(\ry):=\ol{\psi_{z}^\bfr}(\rx,\ry)$, we have $(\psi_{z,\rx}^{\bfr})^{-1}=\ol{\psi_{z,\rx}^{\bfr}}$.
A computation shows that for any $(\rx,\zeta,\ry)\in \R^{3n}$,
\begin{equation}\label{Phi_la}
\Phi_{\la,z,\bfr}(\rx,\zeta) = (\psi_{z}^\bfr(\rx,\la\zeta),\psi_z^\bfr(\rx,-(1-\la)\zeta))  \, , \qquad
\Phi_{\la,z,\bfr}^{-1}(\rx,\ry) = (m_{\la,z,\bfr}(\rx,\ry),\xi_{\la,z,\bfr}(\rx,\ry))\, 
\end{equation}
where we defined the following functions: $m_{\la,z,\bfr}(\rx,\ry):= \psi_{z}^\bfr(\rx,\la\,\ol{\psi_{z}^\bfr}(\rx,\ry))$, $\xi_{0,z,\bfr}:= -\ol{\psi_z^\bfr}$ and for $\la\neq 0$, $\xi_{\la,z,\bfr}(\rx,\ry):= \tfrac{1}{\la} \ol{\psi_z^\bfr}(m_{\la,z,\bfr}(\rx,\ry),\rx)$. We also obtain for $t\in [-1,1]$, $(\rx,\zeta)\in \R^{2n}$, 
\begin{equation}\label{Ups_t}
\Ups_{t,z,\bfr}(\rx,\zeta) = \big(\psi_{z}^\bfr(\rx,t\zeta),\tfrac{-1}{t} \ol{\psi_z^\bfr}(\psi_z^\bfr(\rx,t\zeta),\rx)) =:  (\psi_{z}^\bfr(\rx,t\zeta),\Ups^{z,\bfr}_{t,T}(\rx,\zeta) \big)  \, ,
\end{equation}
and $\Ups_{0,z,\bfr}=\Id_{\R^{2n}}$. Note that $\Ups_{t,z,\bfr}(\rx,0)= (\rx,0)$ for any $\rx\in \R^n$ and $\Ups_{t,T}^{z,\bfr}=\tfrac{1}{t}\Ups^{z,\bfr}_{1,T}\circ I_{1,t}$ where $I_{r,r'}$ is the diagonal matrix with coefficients $I_{ii} = r$ for $i\leq n$ for $1\leq i\leq n$ and $I_{ii}=r'$ for $n+1\leq i \leq 2n$.

\subsection{$\O_M$-linearizations}

We intent to use the linearization to define topological isomorphisms between rapidly decaying section on $TM$ and $M\times M$. We thus need a control at infinity over the derivatives of the linearization $\psi$.

We note $\tau^{z,\bfr}=\tau^z\circ (n_{z,M^2}^\bfr)^{-1} \in C^\infty(\R^{2n},L(E_z))$. Remark that for any $(\rx,\ry)\in \R^{2n}$, $\tau^{z,\bfr}(\rx,\ry)$ is an unitary operator on $E_z$. We will also need the following functions parametrized by $t\in \R$: $\tau_t(x,\eta):=\tau_x(\psi_x(t\eta))$ for any $(x,\eta)\in TM$ and $\tau_{t}^{z,\bfr}(\rx,\zeta):=\tau^{z,\bfr}(\rx,\psi_z^\bfr(\rx,t\zeta))$. 

\begin{defn} A linearization $\psi$ on the exponential manifold $(M,\exp,E,d\mu)$ is said to be a $\O_M$-linearization if for any frame $(z,\bfr)$ the functions $\psi_z^\bfr$ and $\ol{\psi_z^\bfr}$ are in in $\O_M(\R^{2n},\R^n)$ and the functions $\tau_{1}^{z,\bfr}$ and $(\tau_1^{z,\bfr})^{-1}$ are in $\O_M(\R^{2n},L(E_z))$. We will say that $(M,\exp,E,d\mu, \psi)$ has a $\O_M$-bounded geometry, if it the case of $(M,\exp,E,d\mu)$ and $\psi$ is a $\O_M$-linearization.
\end{defn}

\begin{lem} 
\label{OMlem}
Suppose that $\psi$ is a $\O_M$-linearization. Then for any frame $(z,\bfr)$, $\la \in [0,1]$ and $t\in [-1,1]$,

\noindent (i) $\Phi_{\la,z,\bfr} \in \O_M^\times(\R^{2n},\R^{2n})$ and $J_{\la,z,\bfr} \in \O_M^\times(\R^{2n})$,

\noindent (ii) $\Ups_{t,z,\bfr} \in \O_M^\times(\R^{2n},\R^{2n})$ and $J(\Ups_{t,z,\bfr}) \in \O_M^\times(\R^{2n})$,

\noindent (iii) $\tau_t^{z,\bfr}$ and $(\tau_t^{z,\bfr})^{-1}$ are in $\O_M(\R^{2n},L(E_z))$.
\end{lem}
\begin{proof}
\noindent $(i)$ By (\ref{Phi_la}), we have $\Phi_{\la,z,\bfr} = (\psi_{z}^\bfr\circ I_{1,\la},\psi_z^\bfr\circ I_{1,\la-1})$ and $\Phi_{\la,z,\bfr}^{-1}=(m_{\la,z,\bfr},\xi_{\la,z,\bfr}) $ where $m_{\la,z,\bfr}=\psi_z^\bfr\circ I_{1,\la}\circ (\pi_1,\ol{\psi_z^\bfr})$ and 
if $\la \neq 0$, $\xi_{\la}=\tfrac{1}{\la} \ol{\psi_z^\bfr}\circ (m_{\la,z,\bfr},\pi_1)$, while $\xi_{0,z,\bfr}=-\ol{\psi_z^\bfr}$. Thus, the result is a consequence Lemma \ref{Ssigma} $(iii)$ and $(vi)$.

\noindent $(ii)$ By (\ref{Ups_t}), we have for $t\neq 0$, $\Ups_{t,z,\bfr}=(\psi_{z}^\bfr\circ I_{1,t}, \tfrac{-1}{t}\ol{\psi_z^\bfr}\circ (\psi_z^\bfr\circ I_{1,t},\pi_1))$. The result follows again from Lemma \ref{Ssigma} $(iii)$ and $(vi)$. 

\noindent $(iii)$ We have $\tau_{t}^{z,\bfr}=\tau_1^{z,\bfr}\circ I_{1,t}$ and $(\tau_{t}^{z,\bfr})^{-1}=(\tau_1^{z,\bfr})^{-1}\circ I_{1,t}$ so the result follows from Lemma \ref{Ssigma} $(iii)$.  
\end{proof}

The following lemma shows that we can obtain topological isomorphisms on spaces of rapidly decaying functions from the functions $\tau_t$ and $\Phi_{\la}$.

\begin{lem}
\label{Stopolisom}
Let $p\in \N^*$, $\tau\in \O_M^\times(\R^p,GL(E_z))$ and $\Phi \in \O_M^\times (\R^p,\R^p)$. Then the maps $L_{\tau}:= u\mapsto \tau u$, $R_{\tau}:= u\mapsto u\tau$ and $C_\Phi:= u\mapsto u\circ \Phi$ are topological isomorphisms of $\S(\R^p, L(E_z))$.
\end{lem}
\begin{proof} Since $L_{\tau}^{-1}= L_{\tau^{-1}}$, $R_{\tau}^{-1}=R_{\tau^{-1}}$ and $C_\Phi^{-1}=C_{\Phi^{-1}}$, we only need to check the continuity of $L_{\tau}$, $R_{\tau}$ and $C_\Phi$. The continuity of $L_{\tau}$ and $R_{\tau}$ is a direct application of Leibniz formula. Let $\nu$ be a $p$-multi-index and $r\in \N$. Theorem \ref{FaaCS} implies that for any $u\in \S(\R^p,L(E_z))$, 
$$
q_{\nu,N}(u\circ \Phi) \leq \sum_{|\la|\leq |\nu|} \sup_{\rx\in \R^p} \langle \rx \rangle ^{N}|P_{\nu,\la}(\Phi)(\rx)| \norm{(\del^\la u)\circ\Phi (\rx) }_{L(E_z)}
$$
where the functions $P_{\nu,\la}(\Phi)$ are such that $|P_{\nu,\la}(\Phi)(\rx)| \leq C_\nu \langle \rx\rangle ^{q_\nu}$ for a $q_\nu \in \N^*$ and a $C_\nu>0$. Since $\langle \Phi^{-1} (\rx) \rangle \leq C \langle \rx\rangle ^r$ for a $r\in \N^*$ and a $C>0$, we see that there is $C'_\nu>0$ such that $q_{\nu,N}(u\circ \Phi) \leq C'_\nu \sum_{|\la|\leq |\nu|} q_{\la,(q_\nu+N)r}(u)$, which gives the result. 
\end{proof}

\begin{lem} If $(M,\exp,E,d\mu)$ has a $\O_M$-bounded geometry and $\psi$ is a linearization such that there exists $(z_0,\bfr_0)$ such that the functions $\psi_{z_0}^{\bfr_0}$, $\ol \psi_{z_0}^{\bfr_0}$ are in $\O_M(\R^{2n},\R^n)$ and $\tau_1^{z_0,\bfr_0}$, $(\tau_{1}^{z_0,\bfr_0})^{-1}$ are in $\O_M(\R^{2n},L(E_{z_0}))$, then $\psi$ is a $\O_M$-linearization.
\end{lem}
\begin{proof} The result is a direct consequence of the formulas $\psi_{z}^\bfr=\psi_{z,z_0}^{\bfr,\bfr_0}\circ \psi_{z_0}^{\bfr_0}\circ \psi_{z_0,z,T}^{\bfr_0,\bfr}$, $\ol\psi_{z,\rx}^\bfr(\ry)=(d\psi_{z_0,z}^{\bfr_0,\bfr})^{-1}_\rx \, \ol\psi_{z_0}^{\bfr_0}\circ \psi_{z_0,z,M^2}^{\bfr_0,\bfr}(\rx,\ry)$ and $\tau^{z,\bfr}=(\tau_z^{-1}\tau_{z_0})\circ \pi_2 \circ (n_{z,M^2}^{\bfr})^{-1}\,\tau^{z_0,\bfr_0}\circ \psi_{z_0,z,M^2}^{\bfr_0,\bfr} \, 
(\tau_{z_0}^{-1}\tau_{z})\circ \pi_1 \circ (n_{z,M^2}^{\bfr})^{-1}$.
\end{proof}

\subsection{Symbol maps and $\la$-quantization} \label{quantisec}

\begin{assum} We suppose in this section and in section \ref{moyalsection} that $(M,\exp,E,d\mu, \psi)$ has a $\O_M$-bounded geometry.
\end{assum}

The operator $\F$ is a topological isomorphism from $\S'(TM,L(E))$ onto $\S'(T^*M,L(E))$.
We shall now introduce a topological isomorphism between $\S'(M\times M,L(E))$ and $\S'(TM,L(E))$. 
We define the linear application $\Ga_\la$ from $C^\infty(M\times M,L(E))$ into $C^\infty(TM,L(E)))$:
$$
\Ga_\la(K): v \mapsto K^{\pi(v)}\circ \Phi_\la(v)\, .
$$
 As a consequence, $\Ga_\la(K)= \tau^{-1}_{\la}\,(K\circ \Phi_\la)\,\tau_{\la-1}$ and $\Ga_\la^{-1}(a)=(\tau_{\la}\,a\,\tau_{\la-1}^{-1})\circ \Phi_\la^{-1}$. For a given frame $(z,\bfr)$, we note $\Ga_{\la,z,\bfr}:=T_{z,\bfr,T}\circ \Ga_\la \circ T_{z,\bfr,M^2}^{-1}$. A computation shows that for any smooth function $u\in C^{\infty}(\R^{2n},L(E_z))$, $\Ga_{\la,z,\bfr}(u)=(\tau_{\la}^{z,\bfr})^{-1}(u\circ \Phi_{\la,z,\bfr}) \tau_{\la-1}^{z,\bfr}$.

Let us define the smooth strictly positive functions on $\R^{2n}$ and $M\times M$ respectively:
\begin{equation}
\label{mu_la_def}
\mu_{\la,z,\bfr}(\rx,\ry):= \tfrac{\mu_{z,\bfr}(\rx)\mu_{z,\bfr}(\ry)}{\mu_{z,\bfr}^2(m_{\la,z,\bfr}(\rx,\ry))}\, |J_{\la,z,\bfr}|\circ \Phi_{\la,z,\bfr}^{-1}(\rx,\ry) \qquad \mu_{\la}:=\mu_{\la,z,\bfr}\circ (n_{z}^\bfr,n_z^\bfr).
\end{equation}
It is straithtforward to check that $\mu_{\la}$ is indeed independent of $(z,\bfr)$. Note that $\mu_{1-\la}(x,y)=\mu_\la(y,x)$.
Since $\mu_{\la,z,\bfr}\in \O_M^\times(\R^{2n})$, the operator of multiplication $M_{\mu_{\la}}$ is a topological isomorphism on $\S(M\times M,L(E))$. Note also that $\Ga_\la \circ M_{\mu_\la} = M_{\mu_\la\circ \Phi_\la}\circ \Ga_\la$.

\begin{prop} $\Ga_\la$ is a topological isomorphism from $\S(M\times M,L(E))$ onto $\S(TM,L(E))$. Moreover,
$\wt \Ga_\la \circ j_{M^2}= j_{TM}\circ \Ga_\la \circ M_{\mu_\la}$, where $\wt \Ga_\la:={\Ga_\la^{-1}}^*$.
\end{prop}
\begin{proof} Let $(z,\bfr)$ be a frame. It suffices to prove that $\Ga_{\la,z,\bfr}$
is a topological isomorphism from $\S(\R^{2n},L(E_z))$ onto itself. Since $\Ga_{\la,z,\bfr}=L_{(\tau_\la^{z,\bfr})^{-1}}\circ R_{\tau_{\la-1}^{z,\bfr}}\circ C_{\Phi_{\la,z,\bfr}}$, the result follows from Lemma \ref{Stopolisom} and Lemma \ref{OMlem} $(i)$ and $(iii)$.
Let $u,v\in \S(\R^{2n},L(E_z))$. We have (with $j$ the canonical inclusion from $\S(\R^{2n},L(E_z))$ into $\S'(\R^{2n},L(E_z))$: 
\begin{align*}
(\wt\Ga_{\la,z,\bfr}\circ j (u))(v)&= \int_{\R^{2n}} \Tr\big(u(\rx,\ry)(\Ga_{\la,z,\bfr}^{-1}(v)(\rx,\ry))^*\big)\,d\rx\, d\ry \,  \\
&=\int_{\R^{2n}} \Tr\big((\tau_{\la}^{z,\bfr})^{-1}\circ \Phi_{\la,z,\bfr}^{-1}(\rx,\ry)\  u(\rx,\ry)\  \tau_{\la-1}^{z,\bfr}\circ \Phi_{\la,z,\bfr}^{-1}(\rx,\ry)\ \\
&\hspace{3cm}  v^*\circ \Phi_{\la,z,\bfr}^{-1}(\rx,\ry)\big)\, d\rx\, d\ry \\ 
&=\int_{\R^{2n}} \Tr(\Ga_{\la,z,\bfr}(u)(m,\zeta) v^*(m,\zeta)) |J_{\la,z,\bfr}|(m,\zeta)\, dm\, d\zeta \\
&= (j\circ M_{|J_{\la,z,\bfr}|}\circ \Ga_{\la,z,\bfr}(u)) (v)
\end{align*} where we used the following change of variables $(m,\zeta):=\Phi_{\la,z,\bfr}^{-1}(\rx,\ry)$.
Thus, we have $\wt\Ga_{\la,z,\bfr}\circ j = j\circ M_{|J_{\la,z,\bfr}|}\circ \Ga_{\la,z,\bfr}$. The relation $\wt \Ga_\la \circ j_{M^2}= j_{TM}\circ \Ga_\la \circ M_{\mu_\la}$ now follows since $M_{|J_{\la,z,\bfr}|}\circ \Ga_{\la,z,\bfr} = \Ga_{\la,z,\bfr}\circ M_{|J_{\la,z,\bfr}|\circ \Phi_{\la,z,\bfr}^{-1}}$, $T_{z,\bfr,T}^*\circ j\circ M_{\mu^2_{z,\bfr}}=j_{TM}\circ T_{z,\bfr,T}^{-1}$ and $T_{z,\bfr,M^2}^*\circ j \circ M_{\mu_{z,\bfr}\ox \mu_{z,\bfr}}=j_{M^2}\circ T_{z,\bfr,M^2}^{-1}$. 
\end{proof}

As a consequence, $\wt \Ga_\la$ is a topological isomorphism from the space of tempered distributional $L(E)$-sections on $M\times M$, $\S'(M\times M,L(E))$ onto $\S'(TM,L(E))$ and when restricted (in the sense of the previous continous inclusions) to $\S(M\times M,L(E))$, is equal to $\Ga_\la \circ M_{\mu_\la}^{-1}$, so provides a topological isomorphism from $\S(M\times M,L(E))$ onto $\S(TM,L(E))$. Fourier transform coupled with $\wt\Ga_\la$ lead us to the following natural isomorphism from $\S'(M\times M,L(E))$ onto $\S'(T^*M,L(E))$. 

\begin{defn} Let $\la\in [0,1]$. The $\la$-symbol map is the topological isomorphism from $\S'(M\times M,L(E))$ onto $\S'(T^*M,L(E))$: $\sigma_\la:=\F\circ \wt \Ga_\la$. The $\la$-quantization map is the inverse of $\sigma_\la$, noted $\mathfrak{Op}_\la$.
\end{defn}

Thus, the data of a tempered distributional section on the cotangent bundle (i.e. an element of $\S'(T^*M,L(E))$) determines in an unique way (for a given $\la$), an operator continuous from $\S$ to $\S'$, and vice versa. Remark that $\sigma_\la \circ j_{M^2}= j_{T^*M}\circ \F\circ \Ga_\la\circ M_{\mu_\la}$ and $\Op_\la\circ j_{T^*M}=j_{M^2}\circ M_{1/\mu_{\la}}\circ \Ga_\la^{-1}\circ \ol\F$. If $(z,\bfr)$ is a frame then, noting $\Op_{\la,z,\bfr}:= \wt T_{z,\bfr,M^2} \circ \Op_\la\circ \wt T_{z,\bfr,*}^{-1}$, we obtain $\Op_{\la,z,\bfr}=\Ga_{\la,z,\bfr}^*\circ M_{\mu_{z,\bfr}}^*\circ \F_P^*$ so that for any $u\in \S(\R^{2n},L(E_z))$ and $b\in \O_M(\R^{2n},L(E_z))$, 
\begin{equation}
\langle \Op_{\la,z,\bfr}(b),u\rangle= \int_{\R^{3n}} e^{2\pi i \langle \zeta,\vth\rangle} \Tr\big(\mu b(\rx,\vth) (\Ga_{\la,z,\bfr}(u))^*(\rx,\zeta)\big)\,d\zeta\,d\vth\,d\rx\, .\label{eq-oplazb}
\end{equation}
where $\mu b:(\rx,\vth)\mapsto  \mu_{z,\bfr}(\rx)\,b(\rx,\vth)$.

\subsection{Moyal product} \label{moyalsection}

The applications $\Op_{0}$, $\Op_{1}$, $\mathfrak{Op}_W:=\Op_\half$ are respectively the normal, antinormal and Weyl quantization maps. Remark that for any $T\in \S'(T^*M,L(E))$, $\mathfrak{Op}_\la(T^*) = (\mathfrak{Op}_{1-\la}(T))^\dag$. In particular
$$
\mathfrak{Op}_0(T^*) = (\mathfrak{Op}_1(T))^\dag\, ,\qquad \mathfrak{Op}_W(T^*) = (\mathfrak{Op}_W(T))^\dag 
$$
where $\dag$ is the topological isomorphism of $\S'(M\times M,L(E))$ defined as $\langle K^\dag, u\rangle:=\ol{\langle K , u^*\circ j\rangle}$ with $j$ the diffeomorphism on $M\times M : (x,y)\mapsto (y,x)$ and $u\in \S(M\times M,L(E))$. The kernel of the adjoint $A^\dag$ of any operator $A\in L(\S,\S')$ is $(K_A)^\dag$.
As a consequence, $\sigma_\la$ is a linear topological isomorphism (and a $*$-isomorphism in the case of the Weyl quantization) from the algebra $\Re(\S)=L(S,S)\cap L(S',S')$ of regular operators onto its image $\mathfrak{M}_\la:=\sigma_\la(\Re(\S))$. We can transport the operator composition in the world of functions, by defining the $\la$-product on $\mathfrak{M}_\la$ as
$$
T\circ_\la T':= \sigma_\la(\mathfrak{Op}_\la(T)\,\mathfrak{Op}_\la(T'))
$$
so that $\mathfrak{M}_\la$ forms an algebra, and $\mathfrak{M}_\la^*=\mathfrak{M}_{1-\la}$. In the case of $\la=\half$, we recover the Moyal $*$-algebra $\mathfrak{M}_W$ and the Moyal product $\circ_W$. The space $\Psi^{-\infty}(M)\simeq \S(M\times M,L(E))$ of isotropic smoothing operators being an $*$-ideal of $\Re(\S)$, the space $\S(T^*M,L(E))=\sigma_{\la}(\Psi^{-\infty}(M))$ forms an ideal of $\mathfrak{M}_\la$. 
Since we will focus on the pseudodifferential calculus over $M$, we shall not investigate in this paper the full analysis of the Moyal product over $T^*M$. Note however the following property on $\S(T^*M):=\S(T^*M,L(M\times \C)) \simeq \S(T^*M,\C)$:

\begin{prop}
\label{la-product}
$(\S(T^*M),\circ_\la)$ is a (noncommutative, nonunital) Fr\'{e}chet algebra. Moreover,
$$ 
a\circ_\la b\,(x,\eta) = \int_{T_x(M) \times M} d\mu_{x}(\xi)d\mu(y)\int_{V^\la_{x,\xi,y}} d\mu_{x,\xi,y}^*(\th,\th')\, g^\la_{x,\xi,y}\,e^{2\pi i \om^\la_{x,\xi,y}(\eta,\th,\th')} a(y^\la_{x,\xi},\th)\,b(y^{1-\la}_{x,-\xi},\th')
$$
where $y_{x,\xi}^{\la}:=m_\la(\psi_x^{\la\xi},z)$, $\ol y_{x,\xi}^{\la}:=\xi_{\la}(\psi_x^{\la\xi},z)$ and 
\begin{align*}
&V_{x,\xi,y}^\la:= T^*_{y_{x,\xi}^\la}(M)\times T^*_{y^{1-\la}_{x,-\xi}}(M)\, ,\qquad d\mu_{x,\xi,y}^*(\th,\th'):=d\mu^*_{y^\la_{x,\xi}}(\th)\,d\mu^*_{y^{1-\la}_{x,-\xi}}(\th')\, , \\
&g_{x,\xi,y}^\la:= \tfrac{\mu_\la(\psi_x^{\la \xi},\psi_x^{(\la-1)\xi})}{\mu_\la(\psi_x^{\la\xi},y)\,\mu_\la(y,\psi_x^{(\la-1)\xi})}\, , \\
&\om_{x,\xi,y}^\la(\eta,\th,\th'):=\langle \th,\ol y_{x,\xi}^\la\rangle -  \langle \th', \ol y_{x,-\xi}^{1-\la}\rangle - \langle \eta ,\xi \rangle \, .
\end{align*}

\end{prop}
\begin{proof} The product $a\circ_\la b$ on $\S(T^*M)$ is obtained by computation of $\F\circ \Ga_\la\circ M_{\mu_\la}\circ \big( (M_{\mu_\la}^{-1}\circ \Ga_{\la}^{-1}\circ \ol \F(a))\circ_V (M_{\mu_\la}^{-1}\circ \Ga_\la^{-1}\circ \ol \F(b))\big)$, where $\circ_V$ is the Volterra product of kernels. Since $\sigma_\la$ is a topological isomorphism between $\S(M^2)$ and $\S(T^*M)$, the continuity of the Moyal product is equivalent to the continuity of  $\circ_V$, which is equivalent to the continuity of the following product on $\S(\R^{2n})$: 
$$
K\cdot K' (\rx,\ry):= \int_{\R^{n}} K(\rx,\rt)K(\rt,\ry) \mu_{z,\bfr}(\rt) d\rt.
$$
The continuity of this product is obtained by the following estimates
$$
q_{p,(\a,\b)}(K\cdot K') \leq C \, q_{2(p+r),(\a,0)}(K)\,q_{p,(0,\b)}(K'),\qquad  q_{p,\nu}(K):=\sup_{(\rx,\ry)\in \R^{2n}} \langle(\rx,\ry) \rangle^p |\del^\nu K (\rx,\ry)|
$$
where $|\mu_{z,\bfr}(\rt)|\leq C_1 \langle \rt \rangle^{r-n-1}$ and $C:=C_1\int_{\R^n} \langle\rt \rangle^{-(n+1)} d\rt$. 
\end{proof}

\begin{rem} $(\S(T^*M),\circ_W)$ is a $*$-algebra since $(a\circ _W b)^*= b^* \circ_W a^*$ for any $a,b\in \S(T^* M)$. We can also construct another $*$-algebra on $\S(T^*M)$ with the product $a \star b := \half( a\circ_0 b + a\circ_1 b)$. This proves that when $(H_\psi)$ (see Assumption \ref{assumH}) is not satisfied (so that no middle point exist in the classical world) we can still have a canonical star-product on $\S(T^*M)$ which satisfies $(a\star b)^*= b^* \star a^*$.
\end{rem}

\section{Symbol calculus of pseudodifferential operators}

\subsection{Symbols}

\begin{assum} Let $\sigma\in [0,1]$. We suppose in this section that $(M,\exp,E)$ has a $S_{\sigma}$-bounded geometry.   
\end{assum}

The algebra $\Re(\S)$ and $\Psi^{-\infty}$ are respectively too big and too small to develop a satisfactory pseudodifferential calculus that allows an efficient utilization of symbol maps. We shall in this section define some spaces of symbols that will be used to define later special algebras of pseudodifferential operators that lie between $\Re(\S)$ and $\Psi^{-\infty}$.

\begin{defn} A symbol of degree $(l,m)\in \R^2$ of type $\sigma$, on $M$ is a smooth section $a\in C^\infty(T^*M,L(E))$ such that for any $(z,\bfr)$ and any $n$-multi-indices $\a$, $\beta$, there exists $K>0$ such that
\begin{equation}\label{symbolIneq}
\norm{ \del^{(\a,\b)}_{z,\bfr} a^{z} (x,\th)}_{L(E_z)}\leq K \langle x\rangle^{\sigma(l- |\a|)}_{z,\bfr}\,\langle \th\rangle^{m-|\b|}_{z,\bfr,x}
\end{equation}
is valid for all $(x,\th)\in T^*M$. The space of symbols of degree $(l,m)$ and type $\sigma$ is noted $S^{l,m}_\sigma$. 
\end{defn}

Remark that $S^{l,m}_0$ is independant of $l$, so we note this space $S^m_0$. We note $S^{-\infty}_\sigma:=\cap_{l,m} S^{l,m}_\sigma$ and in the case $\sigma>0$, we define $S^{-\infty}:= S^{-\infty}_\sigma=\S(T^*M,L(E))$ (it is independant of $\sigma>0$). We set $S^\infty_\sigma := \cup_{l,m} S^{l,m}_\sigma$. We define similarly $S^{l,m}_{\sg,z}:=S^{l,m}_\sigma(\R^{2n},L(E_z))$, without reference to a frame.  

Since $M$ has a $S_{\sigma}$-bounded geometry, we get the following coordinate independance of the previous definition:

\begin{prop} Let $a\in C^\infty(T^*M,L(E))$. Then $a \in S^{l,m}_\sigma$ if and only if there exists a frame $(z,\bfr)$ such that 
$a$ satisfies (\ref{symbolIneq}).
\end{prop}
\begin{proof} 
Suppose that (\ref{symbolIneq}) is satisfied for $(z',\bfr')$ and let ($z,\bfr$) be another frame. For $(x,\th)\in T^*M$ and $\a,\b$ two $n$-multi-indices with $\nu=(\a,\b)\neq 0$, we get from Equation (\ref{del-cchange}) and Lemma \ref{cchange},
\begin{align*} 
\norm{\del_{z,\bfr}^{\nu} a^z(x,\th)}_{L(E_z)}  \leq K \sum_{\a',\b'} \sum_{\rho\leq \la \leq \nu'}\ 
\langle x\rangle_{z,\bfr}^{\sigma(|\a'|-|\a|)}\langle\th \rangle_{z,\bfr,x}^{|\b'|-|\b|} \langle x \rangle_{z',\bfr'}^{\sigma(|\la^1|-|\a'|)}\\
\hspace{3cm} \times \langle x\rangle_{z',\bfr'}^{\sigma(l-|\rho^1|)}\langle \th \rangle_{z',\bfr',x}^{m-|\b'|}\langle x \rangle_{z',\bfr'}^{\sigma(|\rho^1|-|\la^1|)}.
\end{align*} 
Using (\ref{i-first}), (\ref{i-bis}) and the fact that $|\a|\geq |\rho^1|$, we get the result.
\end{proof} 

\begin{cly} If $a\in C^\infty(T^*M,L(E))$, then $a\in S^{l,m}_\sigma$ if and only if for any $(z,\bfr)$, $a^z\circ (n^\bfr_{z,*})^{-1} \in S^{l,m}_\sigma(\R^{2n},L(E_z))$, or equivalently, there exists $(z,\bfr)$ such that $a^z\circ (n^\bfr_{z,*})^{-1} \in S^{l,m}_\sigma(\R^{2n},L(E_z))$.
\end{cly}

We see that $S^{l,m}_\sigma \cdot S^{l',m'}_{\sigma} \subseteq S^{l+l',m+m'}_{\sigma}$ where $\cdot$ is the composition of sections induced by the matrix product on the fibers of $L(E)$. Moreover, $S^{l,m}_\sigma \subseteq S^{l',m'}_{\sigma}$ for $m\leq m'$ and $l\leq l'$. Thus, $S^{\infty}_\sigma$ is a $*$-algebra, which is bigraduated for $\sigma>0$ and graduated for $\sigma=0$. Remark also that $S^{-\infty}\cdot S^{m}_0$ and $S^{m}_0\cdot S^{-\infty}$ are included in $S^{-\infty}$. Note that if $f\in S^{l,m}_\sigma(T^*M)$ (a symbol where $M$ has its trivial bundle $M\times \C$), then $a_f(x,\th):=f(x,\th)I_{L(E_x)}$ defines a symbol in $S^{l,m}_\sigma$. Such symbols will be called scalar symbols.
Note also that if $a\in S^{l,m}_\sigma$, then $\del_{z,\bfr}^{(\a,\b)}a:=(\tau_z\circ \pi )(\del^{(\a,\b)}_{z,\bfr} a^z)(\tau_z^{-1}\circ\pi)\in S^{l-|\a|,m-|\b|}_{\sigma}$.

If $f\in S_\sigma(\R^n)$ then $(\rx,\vth)\mapsto f(\rx) \Id_{L(E_z)} \in S_\sigma^{0,0}(\R^n,L(E_z))$. In particular  $(\rx,\vth)\mapsto \mu_{z,\bfr}^{\pm 1}(\rx) \Id_{L(E_z)} \in S_\sigma^{0,0}(\R^n,L(E_z))$ if $d\mu$ is a $S_\sigma^\times$-density.

\begin{rem} We note $PS^{l,m}_\sigma(\R^{2n},L(E_z))$ the subspace of $S^{l,m}_\sigma(\R^{2n},L(E_z))$ consisting of functions of the form $\sum_{1\leq i\leq (\dim E_z)^2} P_i e_i$ where $(e_i)$ is a linear basis of $L(E_z)$ and $P_i$ are of the form $\sum_{\beta} c_{i,\b}(\rx)\vth^{\b} $ (finite sum over the $n$-multi-indices $\b$), where for any $i,\b$, $\del^\a c_{i,\b}(\rx)=\O(\langle \rx\rangle^{\sigma(l-|\a|)})$ for any $n$-multi-indices $\a$, and $m=\max_i \deg_\vth P_i$. We check that this definition is independant of the chosen basis $(e_i)$. 

We call polynomial symbol of order $l,m$ and type $\sigma$ any section of the form  $(\tau_{z}\circ \pi) (P\circ n_{z,*}^{\bfr})(\tau_z^{-1}\circ\pi)$ where $P\in PS^{l,m}_\sigma(\R^{2n},L(E_z))$ and $(z,\bfr)$ is a frame. This definition is independant of $(z,\bfr)$. We note $PS^{l,m}_\sigma$ the subspace of $S^{l,m}_\sigma$ consisting of polynomial symbols of order $l,m$ and type $\sigma$. Remark that the section $I:(x,\th)\mapsto I_{L(E_x)}$ 
is in $PS^{0,0}_1$.
\end{rem}

We now topologize the symbol spaces:

\begin{lem}
\label{toposymbol}
The following semi-norms on $S_\sigma^{l,m}$, for $N\in \N$, 
$$
q_{(\a,\b)}(a):= \sup_{(x,\th)\in T^*M} \langle x \rangle_{z,\bfr}^{\sigma(|\a|-l)} \langle \th \rangle_{z,\bfr,x}^{|\b|-m}\norm{\del^{(\a,\b)}_{z,\bfr}a^z(x,\th)}_{L(E_z)}
$$
determine a Fr\'{e}chet topology on $S_\sigma^{l,m}$, which is independant of $(z,\bfr)$.
The applications $T_{z,\bfr,*}$ are topological isomorphisms from $S_\sigma^{l,m}$ onto $S_\sigma^{l,m}(\R^{2n},L(E_z))$. The following inclusions are continous for these topologies: $S^{l,m}_\sigma \cdot S^{l',m'}_{\sigma} \subseteq S^{l+l',m+m'}_{\sigma}$, $S^{l,m}_\sigma \subseteq S^{l',m'}_{\sigma}$ ($m\leq m'$ and $l\leq l'$) and $S^{-\infty}_\sigma\subseteq S^{l,m}_\sigma$. Moreover, the last inclusion is dense when  $S^{l,m}_\sigma$ has the topology of $S^{l',m'}_\sigma$ for $m< m'$ and $l<l'$.
\end{lem}
\begin{proof} The independance of the topology for $(z,\bfr)$ follows from the easily checked estimate for any $(\a,\b)$,
$$
q_{(\a,\b)}^{(z,\bfr)}(a) \leq K_{\a,\b} \underset{|\b'|\geq |\b|\,,\ga\leq \a'}{\sum_{0\leq |(\a',\b')|\leq |(\a,\b)|}} q_{(\ga,\b')}^{(z',\bfr')}(a) .
$$ 
where $K_{\a,\b}>0$. By construction the applications $T_{z,\bfr,*}$ are clearly topological isomorphisms from $S_\sigma^{l,m}$ onto $S_\sigma^{l,m}(\R^{2n},L(E_z))$. The continuity of $S^{l,m}_\sigma \cdot S^{l',m'}_{\sigma} \subseteq S^{l+l',m+m'}_{\sigma}$, $S^{l,m}_\sigma \subseteq S^{l',m'}_{\sigma}$ ($m\leq m'$ and $l\leq l'$) and $S^{-\infty}_\sigma\subseteq S^{l,m}_\sigma$ are straightforward. Following \cite{Melrose}, to prove the density result, we shall prove the stronger property: for any $a\in S^{l,m}_{\sigma}(\R^{2n},L(E_z))$ the sequence 
$$
a_p(\rx,\vth):= (\rho(\rx/p))^{1-\delta_{\sigma,0}}\, \rho(\vth/p)\, a(\rx,\vth) 
$$
converges to $a$ for the topology of $S^{l',m'}_{\sigma}(\R^{2n},L(E_z))$ where $m'>m$ and $l'>l$. Here $\rho\in C^\infty_c(\R^{n},[0,1])$ with $\rho=1$ on $B(0,1)$ and $\rho=0$ on $\R^{n}\backslash B(0,2)$. First, it is clear that $a_p \in S^{-\infty}_\sigma(\R^{2n},L(E_z))$. Noting $R_p(\rx,\vth):=  \langle \rx \rangle^{\sigma(|\a|-l')} \langle \vth \rangle^{|\b|-m'}\norm{\del^{(\a,\b)}(a-a_p)(\rx,\vth)}_{L(E_z)}$ for a given $2n$-multi-index $\nu:=(\a,\b)$, we get with Leibniz rule, for a $K>0$ (by convention $\nu'<\nu$ if and only if $\nu'\leq \nu$ and $\nu'\neq \nu$):
$$
\tfrac{1}{K}\,R_p (\rx,\vth)\leq \Delta_p(\rx,\vth) \langle \rx \rangle^{\sigma(l-l')} \langle \vth \rangle^{m-m'}+ \sum_{\nu'< \nu}   |\del^{\nu-\nu'}\Delta_p(\rx,\vth)|\langle \rx \rangle^{\sigma(l-l'+|\a|-|\a'|)} \langle \vth \rangle^{m-m'+|\b|-|\b'|}
$$
where $\Delta_p(\rx,\vth):=1-(\rho(\rx/p))^{1-\delta_{\sigma,0}} \rho(\vth/p)$. Suppose that $\sigma=0$. In that case, 
$|\Delta_p(\rx,\vth)|\leq 1_{[p,+\infty[}(\vth)$ and if $\nu'<\nu$, 
\begin{equation}
\label{Deltasigma0}
|\del^{\nu-\nu'}\Delta_p(\rx,\vth)|\leq \delta_{\a,\a'}\, K_{\b}\, p^{-|\b|+|\b'|}\, 1_{[p,2p]}(\vth)
\end{equation}
where $1_{[r,r']}$ is the characteristic function of the annulus $A_{r,r'}:=\set{\vth\in \R^{n} \ : \ r\leq \norm{\vth} \leq r'}$ and $K_{\b}:=\sup_{\b'<\b}\norm{\del^{\b-\b'} \rho}_\infty$. As a consequence, for $K'>0$,
$$
\tfrac{1}{K}\,R_p (\rx,\vth)\leq \langle p \rangle^{m-m'} +K_\b \sum_{\nu'< \nu} \delta_{\a,\a'}\, 1_{[p,2p]}(\vth)\ p^{-|\b|+|\b'|}\, \langle \vth \rangle^{m-m'+|\b|-|\b'|} \leq K' \langle p \rangle^{m-m'}
$$ 
and the result follows. Suppose now $\sigma\neq 0$. In that case $|\Delta_p(\rx,\vth)|\leq 1_{F_p}(\rx,\vth)$ where $F_p:=\R^{2n}-B(0,p)^2$ and if $\nu'<\nu$, for a constant $K_\nu>0$
\begin{equation}
\label{Deltasigma1}
|\del^{\nu-\nu'}\Delta_p(\rx,\vth)|\leq K_\nu\, 1_{[\sgn(\a-\a')p,2p]}(\rx)\  1_{[\sgn(\b-\b')p,2p]}(\vth)\, p^{-|\nu|+|\nu'|}\, .
\end{equation}
As a consequence, for $K',K''>0$, and with $r:=\max \{m-m',\sigma(l-l')\}<0$,
$$
\tfrac{1}{K}\,R_p (\rx,\vth)\leq \langle p \rangle^{r} +K'\sum_{\nu'< \nu}  1_{[\sgn(\a-\a')p,2p]}(\rx)\  1_{[\sgn(\b-\b')p,2p]}(\vth)\, \langle \rx \rangle^{\sigma(l-l')}\langle \vth \rangle^{m-m'} \leq K'' \langle p \rangle^{r}
$$ 
and the result follows.
\end{proof}

Note that $S^{-\infty}:=\cap_{l,m} S^{-\infty}_{\sigma>0}=\S(T^*M,L(E))$ and the equality is also valid for the topologies.
The following lemma shows that the symbols of $S^{l,m}_\sigma$ are tempered distributional sections on $T^*M$.

\begin{lem}
\label{slmdistr}
The application 
$j_{T^*M}$ is injective and continuous from $S^{l,m}_\sigma$ into $\S'(T^*M,L(E))$.
\end{lem}
\begin{proof} Since we have the following commutative diagram
\[\xymatrix{
    S^{l,m}_\sigma \ar[rr]^{j_{T^*M}} \ar[d]_{T_{z,\bfr,*}} &&  \S'(T^*M,L(E))  \\
    S^{l,m}_\sigma(\R^{2n},L(E_z))  \ar[r]_i  &\O_M(\R^{2n},L(E_z)) \ar[r]_j & \S'(\R^{2n},L(E_z)) \ar[u]_{T^*_{z,\bfr,*}}   \\
  }\]
where $T^*_{z,\bfr,*}$ is the adjoint of $T_{z,\bfr,*}$ on $\S(T^*M,L(E))$ and $\O_M(\R^{2n},L(E_z))$ is the locally convex complete Hausdorff space of $L(E_z)$-valued functions on $\R^{2n}$ with polynomially bounded derivatives, it is sufficient to check that the natural injection $i$ is continuous from $S^{l,m}_\sigma(\R^{2n},L(E_z))$ into $\O_M(\R^{2n},L(E_z))$. This is obtained by the following estimate, for any $\varphi\in \S(\R^{2n})$ and $\nu=(\a,\b)$ $2n$-multi-index,
$$
\sup_{(\rx,\vth)\in \R^{2n}} \norm{ \varphi\, \del^{\nu} a (\rx,\vth)}_{L(E_z)} \leq K_{\varphi,\nu}\, q_{\nu}(a) 
$$
where $K_{\varphi,\nu}:=\sup_{(\rx,\vth)\in \R^{2n}} |\varphi(\rx,\vth)\langle \rx \rangle^{\sigma(l-|\a|)} \langle \vth \rangle^{m-|\b|}|$.
\end{proof}

\begin{defn} Let $(a_j)_{j\in \N^*}$ be a sequence in $S^{l_j,m_j}_\sigma$ where $(l_j)$ and $(m_j)$ are real strictly decreasing sequences such that $\lim_{j\to \infty} l_j= \lim_{j\to \infty} m_j = -\infty$. We say that $a$ is an asymptotic expansion of $(a_j)_{j\in \N^*}$ and we note
$$
a \sim \sum_{j=1}^\infty a_j
$$
if $a\in C^{\infty}(T^*M,L(E))$ is such that $a - \sum_{j=1}^{k-1} a_j \in S^{l_k,m_k}_{\sigma}$ for any $k\in \N$ with $k\geq 2$. In particular, we have $a\in S^{l_1,m_1}_\sigma$.
\end{defn}

We need asymptotic summation of symbols modulo $S^{-\infty}_\sigma$. The following result of asymptotic completeness is based on a classical method \cite{Shubin} of approximation of series by weightening summands $a_j(\rx,\th)$ with functions which ``cut" a neighborhood of zero in the domain of $x$ (if $\sigma\neq0$) and $\th$. The idea is that the part we cut is bigger and bigger when $j\to\infty$ so that convergence occurs. 

\begin{lem}
\label{asympt}
Let $(a_j)_{j\in \N^*}$ be a sequence in $S^{l_j,m_j}_\sigma$ where $(l_j)$ and $(m_j)$ are real strictly decreasing sequences such that $\lim_{j\to \infty} l_j= \lim_{j\to \infty} m_j = -\infty$. Then 

\noindent (i) There exists $a\in S^{l_1,m_1}_\sigma$ such that $a \sim \sum_{j=1}^\infty a_j$.

\noindent (ii) If another $a'$ satisfies  $a' \sim \sum_{j=1}^\infty a_j$, then $a-a'\in S^{-\infty}_\sigma$.
\end{lem}
\begin{proof} $(ii)$ is obvious. Let us prove $(i)$ for a sequence $(a_j)_{j\in \N^*}$ in $S^{l_j,m_j}_\sigma(\R^{2n},L(E_z))$ and with $a \sim \sum_{j=1}^\infty a_j \in S^{l_1,m_1}_\sigma(\R^{2n},L(E_z))$. The result will then follows for a sequence $(b_j)$ in $S^{l,m}_\sigma$ by taking $b:=T_{z,\bfr,*}^{-1}(a)$ where $a_j:=T_{z,\bfr,*}(b_j)$. Define
$$
a'_j(\rx,\vth):= \Delta_{p_j}(\rx,\vth)\,a_j(\rx,\vth)
$$
where $\Delta_{p_j}$ is defined in the proof of Lemma \ref{toposymbol} and $(p_j)$ is a real sequence in $[1,+\infty[$. For any $j\in \N$, $a'_j-a_j\in S^{-\infty}_\sigma(\R^{2n},L(E_z))$. Thus, the result will follow if we prove that for a specified sequence $(p_j)$ and for any $N\geq 0$, there exists $k_0(N)\geq 2$ such that for any $k\geq k_0(N)$, 
\begin{equation}
\label{asympteq1}
\sum_{j=k+1}^\infty q_{N,l_k,m_k}(a'_j)<\infty 
\end{equation}
where $q_{N,l_k,m_k}:=\sup_{|\nu|\leq N} q_{\nu,l_k,m_k}$, and $q_{\nu,l_k,m_k}$ are the semi-norms of $S^{l_k,m_k}_\sigma(\R^{2n},L(E_z))$. Indeed, with $\norm{\del^\nu a'_j}_{\infty}\leq q_{|\nu|,l_k,m_k}(a'_j)$ for $k\geq k_1(\nu)$, $a':=\sum_{j=1}^\infty a'_j$ is a well defined smooth function and we have then $a'-\sum_{j=1}^{k-1} a_j \in S^{l_k,m_k}_{\sigma}(\R^{2n},L(E_z))$.
Using Leibniz rule, we see that for any $2n$-multi-index $\nu:=(\a,\b)$, and any $j\in \N^*$, there is $K_{\nu,j}>0$ such that 
\begin{align*}
\tfrac{1}{K_{\nu,j}}\norm{\del^{\nu}a'_j(\rx,\vth)}_{L(E_z)}&\leq \Delta_p(\rx,\vth) \langle \rx \rangle^{\sigma (l_j-|\a|)} \langle \vth \rangle^{m_j-|\b|}\\
&\hspace{2cm}+\sum_{\nu' <\nu}   |\del^{\nu-\nu'}\Delta_p(\rx,\vth)|\langle \rx \rangle^{\sigma(l_j-|\a'|)} \langle \vth \rangle^{m_j-|\b'|}.
\end{align*}
Let us suppose that $\sigma=0$. The estimate (\ref{Deltasigma0}) yields for any $N\geq 0, k\geq 2$, $j\geq k+1$,
$$
q_{N,l_k,m_k}(a'_j)\leq K_{N,j} \langle p_j \rangle^{m_j-m_{j-1}}  
$$
for a constant $K_{N,j}>0$. If we now fix $p_j$ as $p_j = (2^j \sup_{N\leq j} \set{K_{N,j},1})^{1/(m_{j-1}-m_j)}$, then we see that for any $N\geq 0$, $k\geq N+2$, $j\geq k+1$, we have $q_{N,l_k,m_k}(a'_j)\leq 2^{-j}$ and (\ref{asympteq1}) is satisfied. Suppose now $\sigma\neq 0$. The estimate (\ref{Deltasigma1}) yields for any $N\geq 0, k\geq 2$, $j\geq k+1$,
$$
q_{N,l_k,m_k}(a'_j)\leq K'_{N,j} \langle p_j \rangle^{r_j}  
$$
for a constant $K'_{N,j}>0$ and with $r_j:=\max \{m_j-m'_{j-1},\sigma(l_j-l'_{j-1})\}<0$. If we now fix $p_j$ as $p_j = (2^j \sup_{N\leq j} \set{K'_{N,j},1})^{-r_j^{-1}}$, then we see that for any $N\geq 0$, $k\geq N+2$, (\ref{asympteq1}) is satisfied as for the case $\sigma=0$.
\end{proof}

\subsection{Amplitudes and associated operators on $\S(\R^n,E_z)$}

We shall see in this section amplitudes as generalizations of symbols of the type $S^{l,m}_{\sigma,z}:=S^{l,m}_\sigma(\R^{2n},L(E_z))$ where $z\in M$ is fixed. For each amplitude, a continuous operator from $\S(\R^n,E_z)$ into itself will be defined. Here the spaces $L(E_z)$ and $E_z$ can simply be considered as $\M_n(\C)$ and $\C^n$. The results in this section will be important for pseudodifferential operators on $M$ in the next section.

\begin{defn} An amplitude of order $l,w,m$ and type $\sigma\in[0,1]$, $\kappa\geq 0$, is a smooth function $a\in C^\infty(\R^{3n},L(E_z))$ such that for any $3n$-multi-index $\nu=(\a,\b,\ga)$, there exists $C_\nu>0$ such that
\begin{equation}
\label{amplest}
\norm{\del^{(\a,\b,\ga)} a (\rx,\zeta,\vth)}_{L(E_z)} \leq C_\nu\, \langle \rx\rangle^{\sigma(l-|\a+\b|)}\, \langle\zeta\rangle^{w+\kappa|\a+\b|}\, \langle \vth \rangle^{m-|\ga|}
\end{equation}
for any $(\rx,\zeta,\vth)\in \R^{3n}$. We note $\Pi_{\sigma,\kappa,z}^{l,w,m}:=\Pi_{\sigma,\kappa}^{l,w,m}(\R^{3n},L(E_z))$ the space of amplitudes of order $l,w,m$ and type $\sigma,\kappa$.
\end{defn} 

Remark that $\Pi^{l,w,m}_{0,\kappa,z}$ is independant of $l$, we note this space $\Pi^{0,w,m}_{0,\kappa,z}$. We note $\Pi^{-\infty,w}_{\sigma,\kappa,z}:=\cap_{l,m} \Pi^{l,w,m}_{\sigma,\kappa,z}$. We set $\Pi^\infty_{\sigma,\kappa,z} := \cup_{l,w,m} \Pi^{l,w,m}_{\sigma,\kappa,z}$ and $\Pi_{\sg,z}^{-\infty}:= \cap_{l,m} \cup_{w,\ka} \Pi_{\sg,\ka,z}^{l,w,m}$. 
We see that $\Pi^{l,w,m}_{\sigma,\kappa,z} \cdot \Pi^{l',w',m'}_{\sigma,\kappa,z} \subseteq \Pi^{l+l',w+w',m+m'}_{\sigma,\kappa,z}$ and $\Pi^{l,w,m}_{\sigma,\kappa,z} \subseteq \Pi^{l',w',m'}_{\sigma,\kappa, z}$ for $m\leq m'$, $w\leq w'$, and $l\leq l'$. Thus, $\Pi^{\infty}_{\sigma,\kappa,z}$ is a $*$-algebra, which is trigraduated for $\sigma>0$ and bigraduated for $\sigma=0$. 
Note also that if $a\in \Pi^{l,w,m}_{\sigma,\kappa, z}$, then $\del^{(\a,\b,\ga)}a \in \Pi^{l-|\a+\b|,w+\kappa|\a+\b|,m-|\ga|}_{\sigma,\kappa,z}$. 

Amplitudes and symbols in $S^{l,m}_{\sigma, z}$ are related by the following lemma:

\begin{lem} 
\label{Amplisymb}
(i) For any $a\in \Pi^{l,w,m}_{\sigma,\kappa,z}$ we have $a_{\zeta=0}:=(\rx,\vth)\mapsto a(\rx,0,\vth)$ in $S^{l,m}_{\sigma,z}$. 

\noindent (ii) For any $s\in S^{l,m}_{\sigma,z}$, the function $(\rx,\zeta,\vth)\mapsto s(\rx,\vth)$ is in $\Pi^{l,0,m}_{\sigma,0,z}$.

\noindent (iii) For any $f\in S_{\sigma}(\R^n)$, the function $(\rx,\zeta,\vth)\mapsto f(\rx)\Id_{L(E_z)}$ is in $\Pi_{\sigma,0,z}^{0,0,0}$.
\end{lem}
\begin{proof}
$(i)$ follows from the fact that $\del^\nu (a \circ P) = (\del^{P(\nu)} a) \circ P$ where $P(\rx,\vth):=(\rx,0,\vth)$.

\noindent $(ii)$ Noting $Q(\rx,\zeta,\vth):=(\rx,\vth)$, the result follows from $\del^{\a,\b,\ga} (s\circ Q) = \delta_{\b,0} (\del^{\a,\ga} s)\circ Q$.

\noindent $(iii)$ follows from $(ii)$ and the fact that $(\rx,\vth)\mapsto f(\rx)\Id_{L(E_z)} \in S^{0,0}_{\sigma,z}$. 
\end{proof}

As the spaces of symbols, the $\Pi_{\sigma,\kappa,z}^{l,w,m}$ are naturally Fr\'{e}chet spaces:

\begin{lem}
\label{topoampli}
The following semi-norms on $\Pi_{\sigma,\kappa,z}^{l,w,m}$: 
$$
q_{(\a,\b,\ga)}^{l,w,m}(a):= \sup_{(\rx,\zeta,\vth)\in \R^{3n}} \langle \rx \rangle^{\sigma(|\a+\b|-l)} \langle \zeta\rangle^{-w-\kappa|\a+\b|}\langle \vth \rangle^{|\ga|-m}\norm{\del^{(\a,\b,\ga)} a (\rx,\zeta,\vth)}_{L(E_z)}
$$
determine a Fr\'{e}chet topology on $\Pi_{\sigma,\kappa,z}^{l,w,m}$. The following inclusions are continous for these topologies: $\Pi^{l,w,m}_{\sigma,\kappa,z} \cdot \Pi^{l',w',m'}_{\sigma,\kappa,z} \subseteq \Pi^{l+l',w+w',m+m'}_{\sigma,\kappa,z}$, $\Pi^{l,w,m}_{\sigma,\kappa,z} \subseteq \Pi^{l',w',m'}_{\sigma,\kappa,z}$ ($m\leq m'$, $w\leq w'$ and $l\leq l'$) and $\Pi^{-\infty,w}_{\sigma,\kappa,z}\subseteq \Pi^{l,w,m}_{\sigma,\kappa,z}$. Moreover, the last inclusion is dense when  $\Pi^{l,w,m}_{\sigma,\kappa,z}$ has the topology of $\Pi^{l',w,m'}_{\sigma,\kappa,z}$ for $m< m'$ and $l<l'$.
\end{lem}
\begin{proof} The continuity results are straightforward. For the density result, we prove as in Lemma \ref{toposymbol}, that for any $a\in \Pi^{l,w,m}_{\sigma,\kappa,z}$ the sequence
$$
a_p(\rx,\zeta,\vth):= (\rho(\rx/p))^{1-\delta_{\sigma,0}}\, \rho(\vth/p)\, a(\rx,\zeta,\vth) =:(1-\Delta_{p}(\rx,\vth))\, a(\rx,\zeta,\vth)
$$
converges to $a$ for the topology of $\Pi^{l',w,m'}_{\sigma,\kappa}(\R^{2n},L(E_z))$ where $m'>m$ and $l'>l$. First note that the application $(\rx,\zeta,\vth)\mapsto  (\rho(\rx/p))^{1-\delta_{\sigma,0}}\, \rho(\vth/p)\, \, \Id_{L(E_z)}$ is an amplitude in $\Pi^{-\infty,0}_{\sigma,0,z}$. Thus, $(a_p)_{p\in \N^*}$ is a sequence in  $\Pi_{\sigma,\kappa,z}^{-\infty,w}$. We define the function $R_p$ such that $q_{(\a,\b,\ga)}^{l',w,m'}(a-a_p)=\sup_{(\rx,\zeta,\vth)\in \R^{3n}}R_p(\rx,\zeta,\vth)$, where $m'>m$ and $l'>l$. For a given $3n$-multi-index $\nu:=(\a,\b,\ga)$, we get with Leibniz rule, for a $K>0$,
\begin{align*}
\tfrac{1}{K}\,R_p (\rx,\zeta,\vth)\leq &\ \Delta_p(\rx,\vth)\, \langle \rx \rangle^{\sigma(l-l')}\, \langle \vth \rangle^{m-m'} + \sum_{\nu'< \nu}   |\del^{\nu-\nu'}\Delta_p(\rx,\vth)|\\
&\hspace{1cm}\times\langle \rx \rangle^{\sigma(l-l'+|\a+\b|-|\a'+\b'|)}\langle \zeta \rangle^{\kappa(|\a'+\b'|-|\a+\b|)} \langle \vth \rangle^{m-m'+|\ga|-|\ga'|}\, .
\end{align*}
Suppose that $\sigma=0$. In that case, 
$|\Delta_p(\rx,\vth)|\leq 1_{[p,+\infty[}(\vth)$ and if $\nu'<\nu$, 
\begin{equation*}
|\del^{\nu-\nu'}\Delta_p(\rx,\vth)|\leq \delta_{\a,\a'}\, \delta_{\b,\b'}\,K_{\ga}\, p^{-|\ga|+|\ga'|}\, 1_{[p,2p]}(\vth)\,.
\end{equation*}
As a consequence we find
$R_p (\rx,\zeta,\vth) = \O_{p\to \infty} (\langle p \rangle^{m-m'})$, as in Lemma \ref{toposymbol}. 
Suppose now $\sigma\neq 0$. In that case $|\Delta_p(\rx,\vth)|\leq 1_{F_p}(\rx,\vth)$ where $F_p:=\R^{2n}-B_{n}(0,p)\times B_n(0,p)$ and if $\nu'<\nu$, for a constant $K_\nu>0$
\begin{equation*}
|\del^{\nu-\nu'}\Delta_p(\rx,\vth)|\leq \delta_{\b-\b',0} K_\nu\, 1_{[\sgn(\a-\a')p,2p]}(\rx)\  1_{[\sgn(\ga-\ga')p,2p]}(\vth)\, p^{-|\nu|+|\nu'|}\, .
\end{equation*}
As a consequence, we find
$R_p (\rx,\zeta,\vth) = \O_{p\to \infty} (\langle p \rangle^{r})$ where $r:=\max \{m-m',\sigma(l-l')\}<0$ and the result follows.
\end{proof}

We shall note $\Delta_\zeta$ the differential operator $\sum_{i=1}^n \del_{\zeta_i}^2$. 
The following formula is valid for any $\vth,\zeta\in \R^n$ and $p\in \N$,
\begin{equation}
\label{e2pii}
\langle \vth\rangle^{2p}e^{2\pi i \langle \vth,\zeta\rangle} = (1-(2\pi)^{-2} \Delta_{\zeta})^{p}\,e^{2\pi i \langle \vth,\zeta\rangle}=: L_{\zeta}^p \,e^{2\pi i \langle \vth,\zeta\rangle}
\end{equation} 
A computation shows that $(1-(2\pi)^{-2} \Delta_{\zeta})^{p} = \sum_{0\leq |\b|\leq p} c_{p,\b}\,\del^{2\b}_\zeta$, where the summation is on $n$-multi-indices $\b$ and $c_{p,\b}:= \tbinom{p}{|\b|}(-1)^{|\b|}(2\pi)^{-2|\b|}\b!$. 
We shall also use the following useful formula valid for any $\vth\in \R^n$, $\zeta\in \R^{n}\backslash\set{0}$ and $p\in \N$,
\begin{equation}
\label{Mformula}
e^{2\pi  i \langle \vth,\zeta\rangle} = \sum_{|\b|=p} \la_\b\,\tfrac{\zeta^\b}{\norm{\zeta}^{2p}}\, \del^{\b}_\vth \,e^{2\pi  i \langle \vth,\zeta\rangle} =: M_{\vth}^{p,\zeta} \,e^{2\pi  i \langle \vth,\zeta\rangle} 
\end{equation}  
where $\la_\b :=\b!(2\pi)^{-|\b|}i^{|\b|}$. We define $^t M_{\vth}^{p,\zeta}:= \sum_{|\b|=p} \la_\b (-1)^p \tfrac{\zeta^\b}{\norm{\zeta}^{2p}} \del^{\b}_\vth$. 

\begin{defn}
\label{OFZ}
We note $\O_{f,z}$, where $f_1,f_2,f_3:\N^{3n}\to \R$, and $f:=(f_1,f_2,f_3)$, the space of smooth functions in $C^{\infty}(\R^{3n},L(E_z))$ such that for any $3n$-multi-index $\nu=(\a,\b,\ga)$, there is $C_\nu>0$ such that 
$$
\norm{\del^\nu a (\rx,\zeta,\vth)}_{L(E_z)} \leq C_\nu \langle \rx \rangle^{f_1(\nu)}\langle \zeta\rangle ^{f_2(\nu)} \langle \vth\rangle ^{f_3(\nu)}
$$
uniformly in $(\rx,\zeta,\vth)\in \R^{3n}$. 
\end{defn}

The vector space $\O_{f,z}$ has a natural family of seminorms $q_{\nu}^{f}$ given by the best constants $C_\nu$ in the previous estimate. With this family, $\O_{f,z}$ is a Fr\'{e}chet space. Obviously, amplitudes in $\Pi_{\sg,\ka,z}^{l,w,m}$ form an $\O_{f,z}$ space where $f_1(\nu):=\sg(l-|\a+\b|)$, $f_2(\nu):=w+\kappa|\a+\b|$ and $f_3(\nu):=m-|\ga|$. For a given triple $f:=(f_1,f_2,f_3)$ and $\rho\in \R$, we will note $f_{3,\rho,\a,\ga}:=\sup_{\b} f_3(\a,\b,\ga)-\rho|\b|$, $f_{2,\rho,\a,\b}:=\sup_{\ga} f_2(\a,\b,\ga)-\rho|\ga|$ and $f_{1,\rho,\a,\b}:=\sup_{\ga} f_1(\a,\b,\ga)-\rho|\ga|$.

\begin{prop}
\label{ampliOP} Let $\Ga$ a continuous linear operator on the space $\S(\R^{2n},L(E_z))$, and $f:=(f_1,f_2,f_3)$ a triple such that there exists $\rho<1$ such that $f_{3,\rho,0,0}<\infty$.

\noindent (i) For any function $a\in \O_{f,z}$ the following antilinear form on $\S(\R^{2n},L(E_z))$ 
$$
\langle \Op_{\Ga}(a) , u\rangle :=  \int_{\R^{3n}} e^{2\pi i \langle \vth,\zeta \rangle } \Tr(a(\rx,\zeta,\vth)\, \Ga(u)^*(\rx,\zeta)) \, d\zeta\,d\vth\,  d\rx
$$
is in $\S'(\R^{2n},L(E_z))$.

\noindent (ii) For any given $u\in \S(\R^{2n},L(E_z))$, the linear form $L_{u,\Ga}:= a\mapsto \langle \Op_{\Ga}(a) , u\rangle$ is continuous on $\O_{f,z}$. In particular $L_{u,\Ga}$ is continuous on any amplitude space $\Pi_{\sigma,\kappa,z}^{l,w,m}$. 
\end{prop}
\begin{proof} $(i)$ We have $\Op_{\Ga}(a) = I(a) \circ \Ga$, where $I(a)$ is the antilinear form on  
$\S(\R^{2n},L(E_z))$:
$$
\langle I(a) , u\rangle :=  \int_{\R^{3n}} e^{2\pi i \langle \vth,\zeta \rangle } \Tr(a(\rx,\zeta,\vth)\, u^*(\rx,\zeta)) \,  d\zeta\, d\vth\, d\rx\,.
$$
We shall prove that $I(a)\in \S'(\R^{2n},L(E_z))$, which will give the result. Let $u\in \S(\R^{2n},L(E_z))$ and let us fix for now $\rx$ and $\vth \in \R^n$. We can check that the map $\zeta\mapsto a(\rx,\zeta,\vth)\, u^*(\rx,\zeta)$ is in $\S(\R^n, L(E_z))$. As a consequence, with (\ref{e2pii}) 
and integration by parts, we get with $R(\rx,\vth):= \int_{\R^n} e^{2\pi i \langle \vth,\zeta \rangle } a(\rx,\zeta,\vth)\,u^*(\rx,\zeta) \,d\zeta$,
\begin{align*}
 R(\rx,\vth)&= \int_{\R^n} e^{2\pi i \langle \vth,\zeta \rangle }\langle \vth\rangle^{-2p}  (1-(2\pi)^{-2} \Delta_{\zeta})^{p}\, a(\rx,\zeta,\vth)\, u^*(\rx,\zeta)\, d\zeta \\
 &= \sum_{0\leq |\beta| \leq p}\,\sum_{\b'\leq 2\b} c_{p,\b}\,\tbinom{2\b}{\b'}\langle \vth\rangle^{-2p}\int_{\R^n} e^{2\pi i \langle \vth,\zeta \rangle }(\del^{(0,\b',0)}\, a(\rx,\zeta,\vth))\,(\del^{(0,2\b-\b')} u^*(\rx,\zeta))\, d\zeta\, .
\end{align*}
Thus, for any $\rx,\vth\in \R^n$, we get by fixing $p$ such that $2(\rho-1)p+f_{3,\rho,0,0}\leq -2n$ (this is possible since $\rho<1$) that for any $N\in \N$,
$$
\norm{R(\rx,\vth)}_{L(E_z)} \leq C_p \langle \vth\rangle^{-2n} \int_{\R^{n}}\langle \rx,\zeta \rangle^{-N+r_p}\, d\zeta \sum_{0\leq |\beta|\leq p}\sum_{\b'\leq 2\b} q_{0,\b',0}^{f}(a)\, q_{N,(0,2\b-\b')}(u)\,   
$$
for a $C_p>0$, where $r_p:=\max_{|\b'|\leq 2p} |f_1(0,\b',0)| + |f_2(0,\b',0)|$.  If we now fix $N$ such that $-N+r_p\leq -4n$, we see, using the inequality $\langle \rx,\zeta \rangle^{-2} \leq \langle \rx \rangle^{-1} \langle \zeta\rangle^{-1}$, that there is $C_{\rho,f}>0$ such that 
\begin{equation}
\label{Iaueq}
|\langle I(a), u\rangle |\leq C_{\rho,f} \sum_{0\leq |\beta|\leq p}\sum_{\b'\leq 2\b}  q_{0,\b',0}^{f}(a)\,q_{N,(0,2\b-\b')}(u)
\end{equation}
which yields the result. 

\noindent $(ii)$ The continuity of $L_{u,\Ga}$ on $\O_{f,z}$ follows directly from (\ref{Iaueq}) since $L_{u,\Ga}(a)=\langle I(a),\Ga(u)\rangle$. Since $\Pi^{l,w,m}_{\sigma,\kappa,z}= \O_{f,z}$ for a triple $f=(f_1,f_2,f_3)$ such that $f_{3,0,0,0}<\infty$, $L_{u,\Ga}$ is continous on any amplitude space.
\end{proof}

For any amplitude $a$, we will also note $\Op_{\Ga}(a)$ the continous linear map from $\S(\R^{n},E_z)$ into $\S'(\R^n,E_z)$, associated to the tempered distribution $u\mapsto \langle \Op_{\Ga}(a),u\rangle$. 

\begin{rem} If $(M,\exp,E,d\mu,\psi)$ has a $\O_M$-bounded geometry, we saw that for any frame $(z,\bfr)$ and $\la\in [0,1]$, the $\Ga_{\la,z,\bfr}$ maps are topological isomorphisms on $\S'(\R^{2n},L(E_z))$. Thus, Lemma \ref{ampliOP} implies that for a given $a\in \Pi_{\sigma,\kappa,z}^{l,w,m}$, we can define a family indexed by $\la\in [0,1]$ of operators $\Op_{\Ga_{\la,z,\bfr}}(a)$ which are continous from  $\S(\R^{n},E_z)$ into $\S'(\R^n,E_z)$.
\end{rem}

\begin{rem}
\label{Oplien} Suppose that $(M,\exp,E,d\mu)$ has a $\S_{\sigma}$ bounded geometry and that $\psi$ is a $\O_M$-linearization. We deduce from (\ref{eq-oplazb}) that if $s$ is a symbol in $S^{l,m}_\sigma$ and $\la\in [0,1]$, we have $(\Op_\la(s))_{z,\bfr}=\Op_{\Ga_{\la,z,\bfr}}(\mu s_{z,\bfr})$ where $(z,\bfr)$ is a frame, $s_{z,\bfr}:=T_{z,\bfr,*}(s)$ and $\mu s_{z,\bfr}:=(\rx,\zeta,\vth)\mapsto \mu_{z,\bfr}(\rx)\, s_{z,\bfr}(\rx,\vth) \in \Pi^{l,0,m}_{\sigma,0,z}$. We will also note $\mu^{-1} s_{z,\bfr}(\rx,\zeta,\vth):= \mu_{z,\bfr}^{-1}(\rx) s_{z,\bfr}(\rx,\vth) \in \Pi^{l,0,m}_{\sigma,0,z}$.
\end{rem}

We now establish a sufficient condition on $\Ga$ and $a$ in order to have $\Op_\Ga(a)$ stable (and continuous) on $\S(\R^n,E_z)$.
The result will be used to establish regularity of pseudodifferential operators.

\begin{lem}
\label{amplContinu} Let $\Ga$ be a continuous linear operator on $\S(\R^{2n},L(E_z))$ of the form $\Ga= L_{\tau_1}\circ R_{\tau_2}\circ C_{\Phi}$, where $\tau_i\in \O_M(\R^{2n},L(E_z))$ (for $1\leq i\leq 2$), and $\Phi:=(\pi_1,\psi)\in C^{\infty}(\R^{2n},\R^{2n})$ is such that $\psi\in \O_M(\R^{2n},\R^n)$ and there exist $c,\eps,r >0$, such that for any $(\rx,\zeta)\in \R^{2n}$, $\langle \psi(\rx,\zeta)\rangle\geq c \langle \rx \rangle^{\eps} \langle \zeta \rangle^{-r}$ and for any $\rx\in \R^n$, there is $c_\rx >0$ such that $\langle\psi(\rx,\zeta)\rangle\geq c_\rx \langle \zeta \rangle^{\eps}$ uniformly in $\zeta\in \R^n$.

Suppose that $f=(f_1,f_2,f_3)$ is such that there exist $(\rho_1,\rho_2,\rho_3)\in \R^3$ such that $\rho_3<1$, $(r/\eps)\rho_1+\rho_2<1$ and for any $2n$-multi-index $\mu$, $f_{1,\rho_1,\mu}<\infty$, $f_{2,\rho_2,\mu}<\infty$, $f_{3,\rho_3,\mu}<\infty$ and for any $n$-multi-index $\a$ $f_{3,\rho_3,\a}:=\sup_{\ga}f_{3,\rho_3,\a,\ga}<\infty$. Then for any function $a\in \O_{f,z}$, the operator $\mathfrak{Op}_{\Ga}(a)$ is continuous from $\S(\R^{n},E_z)$ into itself. In particular, this is the case for any amplitude $a\in \Pi^{l,w,m}_{\sigma,\kappa,z}$.
\end{lem}
\begin{proof} Let $u,v\in \S(\R^{n},E_z)$. By definition, $\langle\mathfrak{Op}_{\Ga}(a) (v) ,u\rangle = \Op_{\Ga}(a)(u\otimes \ol v)$ and $\Ga(K)= \tau_1\, (K\circ \Phi)\, \tau_2$. Noting $a'(\rx,\zeta,\vth):=\tau_1^*(\rx,\zeta)\,a(\rx,\zeta,\vth)\,\tau_2^*(\rx,\zeta)$, we obtain
\begin{align*}
\langle\mathfrak{Op}_{\Ga}(a) (v) ,u\rangle&:=  \int_{\R^{3n}} e^{2\pi i \langle \vth,\zeta \rangle } \big(\,a'(\rx,\zeta,\vth)\, \,v(\psi(\rx,\zeta))\, \big| \,u(\rx)\,\big) \,  d\zeta\,d\vth\,d\rx \, \\
& =  \int_{\R^{n}} \big(\,g(\rx) \big| \,u(\rx)\,\big) \,  d\rx\ 
\end{align*}
where $g(\rx):=\int_{\R^{2n}} e^{2\pi i \langle \vth,\zeta \rangle } \,a'(\rx,\zeta,\vth)\, \,v\circ \psi(\rx,\zeta)\,d\zeta\,d\vth$.

A computation with the Faa di Bruno formula shows that for any $2n$-multi-index $\nu$, any $N\in \N$ and any $\rx\in \R^n$ there is $C_{\rx,N,\nu}>0$ such that $\norm{\del^\nu (v\circ \psi) (\rx,\zeta)}_{E_z}\leq C_{\rx,N,\nu}\langle \zeta\rangle^{-N}$ uniformly in $\zeta\in \R^n$. As a consequence, the map $\zeta \mapsto \del^{\a',0}a'(\rx,\zeta,\vth)\,\del^{\a-\a'}(v\circ \psi)(\rx,\zeta)$ is in $\S(\R^n,E_z)$. We can thus successively integrate by parts in $g(\rx)$ so that for any $p\in \N^*$,
\begin{align*}
&g(\rx) =\int_{\R^{2n}}e^{2\pi i \langle \vth,\zeta \rangle } \langle\vth \rangle^{-2p}L_{\zeta}^p (a'(v\circ \psi))(\rx,\zeta,\vth)\,d\zeta\,d\vth \, .
\end{align*}
By taking $p$ such that $(\rho_3-1)2p +c_0 \leq -2n$ where $c_\a:=\sup_{\a'\leq \a} f_{3,\rho_3,\a'}$, we see that the previous integrand is absolutely integrable, and we can permute the order of integrations $d\zeta d\vth\to d\vth d\zeta$. Since all the successive $\vth$-derivatives of $\langle\vth \rangle^{-2p}L_{\zeta}^p (a'(v\circ \psi))(\rx,\zeta,\vth)$ converges to 0 when $\langle \vth\rangle$ goes to infinity, we can then integrate by parts in $\vth$ so that for any $q\in \N$ and $p\geq p_0$
$$
g(\rx) =\int_{\R^{2n}}e^{2\pi i \langle \vth,\zeta \rangle }\langle \zeta\rangle^{-2q} L_{\vth}^q(\langle\vth \rangle^{-2p}L_{\zeta}^p (a'(v\circ \psi)))(\rx,\zeta,\vth)\,d\zeta\,d\vth \, .
$$
Noting $h_{p,q}$ the previous integrand, we see that for any $n$-multi-index $\a$,  $\del^\a h_{p,q}$ is a linear combination of terms of the form
$$
e^{2\pi i\langle \vth,\zeta\rangle} \langle \zeta\rangle^{-2q} \langle \vth\rangle^{-2p-|\ga-\ga'|} \del^{\a',\b',\ga'} a' \del^{\a-\a',\b-\b'} v\circ \psi
$$
where $|\ga|\leq 2p$, $\ga'\leq \ga$, $|\b|\leq 2q$, $\b'\leq \b$ and $\a'\leq \a$. A computation with the Faa di Bruno formula shows that for any $2n$-multi-index $\nu$ there is $r_{\nu}\in\N^*$ such that for any $N>0$, there is $C_{\nu,N}>0$ such that for any $w\in \S(\R^n,E_z)$ and any $(\rx,\zeta)\in \R^{2n}$, $\norm{\del^\nu (w\circ \psi)(\rx,\zeta)}_{E_z}\leq C_{\nu,N}\langle\rx,\zeta\rangle^{r_\nu-N}\langle \zeta\rangle^{r_\nu+(r/\eps)N}\sum_{|\nu'|\leq |\nu|} q_{[N/\eps]+1,\nu'}(w)$. Moreover, we check that there is $K_{\a,p}>0$ such that 
\begin{equation*}
\norm{\del^{(\a',\b',\ga')} a'(\rx,\zeta,\vth)}_{L(E_z)} \leq C_{\a,p,q} \langle \rx\rangle^{K_{\a,p}+ \rho_1 2q}\langle \zeta\rangle^{K_{\a,p}+\rho_2 2q} \langle \vth\rangle^{c_\a +\rho_3 2p} \, .
\end{equation*}
As a consequence, we get the estimate
$$
\norm{\del^\a h_{p,q}} \leq C_{\a,p,q,N} \langle\rx \rangle^{K'_{\a,p}+\rho_1 2q -N} \langle \zeta\rangle^{K'_{\a,p}+(\rho_2-1)2q+(r/\eps)N} \langle \vth\rangle^{c_\a+(\rho_3-1)2p} \sum_{|\nu'|\leq |\nu|} q_{[N/\eps]+1,\nu'}(v)\, .
$$
or equivalently, replacing $K'_{\a,p}+\rho_1 2q -N$ by $-N$, 
\begin{align*}
&\norm{\del^\a h_{p,q}} \leq C_{\a,p,q,N} \langle\rx \rangle^{-N} \langle \zeta\rangle^{K''_{\a,p}+(\rho_2-1+(r/\eps)\rho_1)2q+(r/\eps)N} \langle \vth\rangle^{c_\a+(\rho_3-1)2p}\\ 
&\hspace{3cm} \sum_{|\nu'|\leq |\nu|} q_{[N+K'_{\a,p}+\rho_1 2q/\eps]+1,\nu'}(v)\, .
\end{align*}
Fixing now, for a given $N$, $p$ such that $(\rho_3-1)2p+ c_\a \leq -2n$ and $q$ such that $K''_{\a,p}+(\rho_2-1+(r/\eps)\rho_1)2q+(r/\eps)N \leq -2n$, we obtain the result.  
\end{proof}

The following lemma gives a characterization of smoothing kernels in the cases $\sigma=0$ and $\sigma\neq 0$. If $s$ is in a space of symbols and $\Ga$ is a continuous linear map on $\S(\R^{2n},L(E_z))$, we will note $\Op_\Ga(s):=\Op_\Ga((\rx,\zeta,\vth)\mapsto s(\rx,\vth))$. We shall use the Fr\'{e}chet space $\O_{\sg,f,z}^{l,m}$ of smooth functions $a$ in $C^{\infty}(\R^{3n},L(E_z))$ such that for any $\nu:=(\mu,\ga)\in \N^{2n}\times \N^n$ 
$$
\norm{\del^\nu a (\rx,\zeta,\vth)}_{L(E_z)}\leq C_\nu \langle \rx\rangle^{\sg(l+f_1(\mu))} \langle \zeta\rangle^{f_2(\nu)} \langle \vth\rangle^{m+f_3(\mu)}\, .
$$
We will note $\O_{0,f,z}^{l,m}=:\O_{f_2,f_3,z}^{m}$.
Clearly, $\Op_{\Ga}(a)$ (see Lemma \ref{ampliOP}) is defined as an antilinear form on $\S(\R^{2n},L(E_z))$ whenever $a\in \O_{f,z}^{l,m}$ with $m+f_3(0)< -n$.
We note $F$ the set of functions $f_2:\N^{3n}\to \R$ such that there is $\rho<1$ such that for any $(\a,\b)\in \N^{2n}$ $f_{2,\rho,\a,\b}:=\sup_{\ga} f_2(\a,\b,\ga)-\rho|\ga|<\infty$.

\begin{lem}
\label{noyauReste}
Let $K\in \S'(\R^{2n},L(E_z))$, and $\Ga$ a topological isomorphim on $\S(\R^{2n},L(E_z))$ of the form $\Ga=L_{\tau_1}\circ R_{\tau_2}\circ C_\Phi$ with $\tau_1,\tau_2\in \O_M^\times(\R^{2n},GL(E_z))$, $\Phi\in\O^\times_M(\R^{2n},\R^{2n})$. Then 

\noindent (i) Case $\sigma=0$. The following are equivalent:
 
(i-1) There is $f_3:\N^{2n}\to \R$ such that for any $m\leq-f_3(0) -2n$, there exist $f_{2,m}\in F$, $a_m \in \O_{f_{2,m},f_3,z}^{m}$ such that $K=\Op_{\Ga}(a_m)$.

(i-2) $K \in C^\infty(\R^{2n},L(E_z))$ and for any $2n$-multi-index $\nu$, $N\in \N$, there is $C_{\nu,N}>0$ such that for any $(\rx,\zeta)\in \R^{2n}$, $\norm{\del^{\nu} K_{\Ga} (\rx,\zeta)}_{L(E_z)} \leq C_{\nu,N} \langle \zeta \rangle^{-N}$, where $K_{\Ga}:=K\circ \Ga=\wt\tau_1\,K\circ \Phi\,\wt \tau_2\,|J(\Phi)|$. 

(i-3) There is $s\in S^{-\infty}_{0,z}$ such that $K=\Op_{\Ga}(s)$.

\noindent (ii) Case $\sigma>0$. The following are equivalent:

(ii-1) There is $f_1,f_3:\N^{2n}\to \R$ such that for any  $m\leq -f_3(0)-2n$, there exist $f_{2,m}\in F$ and $a_{m} \in \O_{\sg,f_1,f_{2,m},f_3,z}^{m,m}$ such that $K=\Op_\Ga(a_{m})$.

(ii-2) $K\in \S(\R^{2n},L(E_z))$.

(ii-3) There is $s\in S^{-\infty}_z$ such that $K=\Op_{\Ga}(s)$.
\end{lem}
\begin{proof}

$(i)$ The implication \emph{(i-3)} $\Rightarrow$ \emph{(i-1)} is trivial. We will prove \emph{(i-1)} $\Rightarrow$ \emph{(i-2)} $\Rightarrow$ \emph{(i-3)}. Suppose \emph{(i-1)}. Thus, for any $m\leq -2n-f_3(0)$, there is $f_{2,m}\in F$, $a_m \in \O_{f_{2,m},f_3,z}^{m}$ such that for any $u\in \S(\R^{2n},L(E_z))$, 
$$
\langle K\circ \Ga^{-1},u\rangle = \int_{\R^{3n}} e^{2\pi i \langle\vth ,\zeta\rangle} \Tr\big(a_m(\rx,\zeta,\vth)\, u^*(\rx,\zeta) \big) \, d\zeta\,d\vth\,d\rx\, .
$$
Since $m\leq -2n-f_3(0)$, the preceding integral is absolutely convergent and we can permute the order of integration. As a consequence, we get 
$\langle K\circ \Ga^{-1},u\rangle = \int_{\R^{2n}} \Tr\big(U_m(\rx,\zeta)\, u^*(\rx,\zeta) \big) \, d\zeta\,d\rx$ where $U_m(\rx,\zeta):= \int_{\R^{n}} e^{2\pi i \langle \vth,\zeta \rangle}\,a_m(\rx,\zeta,\vth)\,d\vth$, we check easily that $U_m$ is a continous function on $\R^{2n}$, so we deduce that $U_m=:U$ is independant of $m$ and $K\circ \Ga^{-1}$ is a distribution which is continous function equal to $U$. Noting $b_m:= e^{2\pi i \langle \vth,\zeta \rangle}\,a_m(\rx,\zeta,\vth)$ we see that for any $2n$-multi-index $\mu:=(\a,\b)$, $\del^\mu_{\rx,\zeta} b_m = e^{2\pi i\langle \vth,\zeta\rangle}\sum_{\b'\leq \b}\tbinom{\b}{\b'} (2\pi i \vth)^{\b-\b'} \del^{\a,\b',0}a_m$ and we have then the estimates
$$
\norm{\del^{\mu} b_m } \leq C_{\mu,m}  \langle \zeta\rangle^{\sup_{\b'\leq \b} f_{2,m}(\a,\b',0)} \langle \vth\rangle^{m+c_\mu}
$$
where $c_\mu=\sup_{\b'\leq \b} f_{3}(\a,\b')+|\b|$. Defining $m_{\mu}:= -2n-\sup_{|\mu'|\leq |\mu|} c_{\mu'}$, we see that $U$ is smooth and 
$$
\del^\mu U = \int_{\R^{2n}} \del^\mu b_{m_\mu} d\vth = \sum_{\b'\leq \b}\tbinom{\b}{\b'}(2\pi i)^{|\b-\b'|} \int_{\R^{n}} e^{2\pi i\langle \vth,\zeta\rangle} \vth^{\b-\b'} \del^{\a,\b',0} a_{m_{\mu}}(\rx,\zeta,\vth) \, d\vth\, .
$$
All the $\vth$-derivatives of $\vth\mapsto \vth^{\b-\b'} \del^{\a,\b',0} a_{m_{\mu}}(\rx,\zeta,\vth)$ converge to zero when $\norm{\vth}\to \infty$ so we can we integrate by parts in $\vth$ so that for any $p\in \N$:
$$
\del^\mu U= \sum_{\b'\leq \b}\tbinom{\b}{\b'}(2\pi i)^{|\b-\b'|}  \int_{\R^{n}} e^{2\pi i\langle \vth,\zeta\rangle} \langle \zeta\rangle^{-2p} L_{\vth}^p\big( \vth^{\b-\b'} \del^{\a,\b',0} a_{m_{\mu}}\big)(\rx,\zeta,\vth) \, d\vth\, .
$$
Since $a_{m_{\mu}} \in \O_{f_{2,m_{\mu}},f_3,z}^{m_\mu}$ and $f_{2,m_\mu,\rho_\mu,\la}<\infty$ for a $\rho_\mu<1$, we see that the integrand $h_{p}$ of the previous integral satisfies the estimate
$$
\norm{h_{p}(\rx,\zeta,\vth)} \leq C_{p,\mu} \langle \zeta\rangle^{-2p+\sup_{\b'\leq \b} f_{2,m_{\mu},\rho_\mu,\a,\b'}+2p\rho_\mu} \langle \vth \rangle^{-2n}\, .
$$
Given $N>0$ and fixing $p$ such that $(\rho_\mu-1)2p+\sup_{\b'\leq \b} f_{2,m_{\mu},\rho_\mu,\a,\b'}\leq -N$, we finally obtain that $K\circ \Ga^{-1}=U$ is smooth and satisfies for any $\mu \in \N^{2n}$ and $N>0$, $\norm{\del^\mu K\circ \Ga^{-1} (\rx,\zeta)}_{L(E_z)}\leq C_{\mu,N} \langle \zeta\rangle^{-N}$. We also have for any $u\in \S(\R^{2n},L(E_z))$, $\langle K,u\rangle = \langle U,\Ga(u)\rangle = \int_{\R^{2n}} \Tr(U'(\rx,\zeta) u^* \circ \Phi(\rx,\zeta))d\rx\,d\zeta$ where $U'(\rx,\zeta):=\tau_1^*(\rx,\zeta)U(\rx,\zeta) \tau_2^*(\rx,\zeta)$. 
Using the change of variables provided by the diffeomorphism $\Phi$, we get $\langle K,u\rangle= \int_{\R^{2n}} \Tr(K(\rx,\ry)\,u^*(\rx,\ry))\,d\rx\,d\ry$ where $K(\rx,\ry):= (|J(\Phi^{-1})|(\rx,\ry)) U'\circ \Phi^{-1}(\rx,\ry)$. The result follows. 

\noindent Suppose now \emph{(i-2)}. It is not difficult to see that $\F_P$ sends $S^{-\infty}_{0,z}$ (seen as a subspace of $\S'(\R^{2n},L(E_z))$) into $S^{-\infty}_{0,z}$. In particular, we have $s:=\F_P(K_\Ga) \in S^{-\infty}_{0,z}$. A computation shows that 
$\langle K,u\rangle = \langle \Op_\Ga(s),u\rangle$ for any $u\in \S(\R^{2n},L(E_z))$.

\noindent $(ii)$ Suppose \emph{(i-1)}. Following the proof of $(i)$, we see that it is sufficient to prove that $U$ is in $\S(\R^{2n},L(E_z))$, where $U(\rx,\zeta):=\int_{\R^n} e^{2\pi i \langle \vth,\zeta\rangle} a_{m}(\rx,\zeta,\vth)\, d\vth$ (independant of $m$). Let us fix $N>0$.
For any $2n$-multi-index $\mu:=(\a,\b)$, $\del^\mu_{\rx,\zeta} b_m = e^{2\pi i\langle \vth,\zeta\rangle}\sum_{\b'\leq \b}\tbinom{\b}{\b'} (2\pi i \vth)^{\b-\b'} \del^{\a,\b',0}a_m$ and we have the estimates
$$
\norm{\del^{\mu} b_m } \leq C_{\mu,m} \langle\rx\rangle^{\sg m+\sg d_\mu} \langle \zeta\rangle^{\sup_{\b'\leq \b} f_{2,m}(\a,\b',0)} \langle \vth\rangle^{m+c_\mu}
$$
where $c_\mu=\sup_{\b'\leq \b} f_{3}(\a,\b')+|\b|$ and $d_\mu:=\sup_{\b'\leq \b}f_1(\a,\b')$. Defining 
$$
m_{\mu,N}:= \min\{-2n-\sup_{|\mu'|\leq |\mu|} c_{\mu'},-N/\sg-\sup_{|\mu'|\leq |\mu|}d_{\mu'}\}
$$ 
we see that $U$ is smooth and 
$$
\del^\mu U = \int_{\R^{2n}} \del^\mu b_{m_{\mu,N}} d\vth = \sum_{\b'\leq \b}\tbinom{\b}{\b'}(2\pi i)^{|\b-\b'|} \int_{\R^{n}} e^{2\pi i\langle \vth,\zeta\rangle} \vth^{\b-\b'} \del^{\a,\b',0} a_{m_{\mu,N}}(\rx,\zeta,\vth) \, d\vth\, .
$$
All the $\vth$-derivatives of $\vth\mapsto \vth^{\b-\b'} \del^{\a,\b',0} a_{m_{\mu,N}}(\rx,\zeta,\vth)$ converge to zero when $\norm{\vth}\to \infty$ so we can we integrate by parts in $\vth$ so that for any $p\in \N$:
$$
\del^\mu U= \sum_{\b'\leq \b}\tbinom{\b}{\b'}(2\pi i)^{|\b-\b'|}  \int_{\R^{n}} e^{2\pi i\langle \vth,\zeta\rangle} \langle \zeta\rangle^{-2p} L_{\vth}^p\big( \vth^{\b-\b'} \del^{\a,\b',0} a_{m_{\mu,N}}\big)(\rx,\zeta,\vth) \, d\vth\, .
$$
Since $a_{m_{\mu,N}} \in \O_{\sg,f_1,f_{2,m_{\mu,N}},f_3,z}^{m_{\mu,N},m_{\mu,N}}$ and $f_{2,m_{\mu,N},\rho_{\mu,N},\la}<\infty$ for a $\rho_{\mu,N}<1$, we see that the integrand $h_{p}$ of the previous integral satisfies the estimate
$$
\norm{h_{p}(\rx,\zeta,\vth)} \leq C_{p,\mu,N} \langle \rx\rangle^{-N} \langle \zeta\rangle^{-2p+\sup_{\b'\leq \b} f_{2,m_{\mu,N},\rho_{\mu,N},\a,\b'}+2p\rho_{\mu,N}} \langle \vth \rangle^{-2n}\, .
$$
Fixing $p$ such that $(\rho_{\mu,N}-1)2p+\sup_{\b'\leq \b} f_{2,m_{\mu,N},\rho_{\mu,N},\a,\b'}\leq -N$, we finally obtain the following estimate $\norm{\del^\mu U}_{L(E_z)}\leq C_{\mu,N} \langle \rx\rangle^{-N}\langle \zeta\rangle^{-N}$, which yields \emph{(i-2)}. The other implications are straightforward.
\end{proof}

\begin{cly}
\label{correste} Same hypothesis. We have (for $\sg=0$ or $\sg> 0$), $\Op_{\Ga}(S^{-\infty}_{\sigma,z})=\cap_{l,m}\cup_{w,\ka} \Op_{\Ga}(\Pi_{\sg,\ka,z}^{l,w,m})=\Op_{\Ga}(\Pi_{\sg,z}^{-\infty})$.
\end{cly}

\begin{lem}
\label{amplireste}
Let $u \in S(\R^{2n},L(E_z))$ and $\b$ a $n$-multi-index.

\noindent (i) For any triple $f:=(f_1,f_2,f_3)$ such that there exists $\rho<1$ such that for any $2n$-multi-index $(\a,\ga)$, $f_{3,\rho,\a,\ga}<\infty$, 
the following linear forms are continuous on $\O_{f,z}$
\begin{align*}
&R_{\b,u} : a\mapsto \int_{\R^{3n}} \zeta^\b e^{2\pi i \langle \vth,\zeta\rangle} \Tr(a(\rx,\zeta,\vth)\,u(\rx,\zeta))\,d\zeta\,d\vth\,d\rx\, ,\\
&S_{\b,u} : a\mapsto  (i/2\pi)^{|\b|}\int_{\R^{3n}} e^{2\pi i\langle \vth,\zeta\rangle} \Tr(\del_{\vth}^\b a(\rx,\zeta,\vth)\,u(\rx,\zeta))\,d\zeta\,d\vth\,d\rx\, .
\end{align*}
\noindent (ii) $R_{\b,u}=S_{\b,u}$ on any $\Pi^{l,w,m}_{\sigma,\kappa,z}$ space.
\end{lem}

\begin{proof} $(i)$ The continuity of $R_{\b,u}$ is a direct consequence of Proposition \ref{ampliOP} since $R_{\b,u}=L_{u_\b,\Id}$ where $u_\b(\rx,\zeta):=\zeta^\b u(\rx,\zeta)$. Suppose that $\nu_0$ is a $3n$-multi-index, we note $f^{\nu_0}:=\nu\mapsto f(\nu+\nu_0)$. A computation shows for any $\rho$, and $n$-multi-indices $\a,\ga$, $f_{3,\rho,\a,\ga}^{\nu_0} \leq f_{3,\rho,\a+\a_0,\ga+\ga_0} + \rho|\b_0|$. Thus if there is $\rho<1$ such that for any $2n$-multi-index $(\a,\ga)$, $f_{3,\rho,\a,\ga}<\infty$, then for any $2n$-multi-index $(\a,\ga)$, $f^{\nu_0}_{3,\rho,\a,\ga}<\infty$. If $a \in \O_{f,z}$ then $\del^{\nu_0} a \in \O_{f^{\nu_0},z}$ and the linear map $a\mapsto \del^{\nu_0} a$ is continuous.  As a consequence, since $S_{\b,u}=L_{u,\Id}\circ D_\b$, where $D_\b:=(i/2\pi)^{\b}\del^\b_{\vth}$, the continuity of $S_{\b,u}$ on $\O_{f,z}$ follows from Proposition \ref{ampliOP}.

\noindent $(ii)$ The equality is easily obtained on $\Pi_{\sigma,\kappa,z}^{-\infty,w}$ by an integration by parts in $\vth$ and permutations of the order of integration $d\zeta d\vth \to d\vth d\zeta$ in $R_{\b,u}(a)$ (authorized for $a\in \Pi_{\sigma,\kappa,z}^{-\infty,w}$). The result now follows from $(i)$ and the density result of Lemma \ref{topoampli}.
\end{proof}

If $N\geq 1$ and $\b,\ga,$ $n$-multi-indices, we note for any amplitude $a\in \Pi^{l,w,m}_{\sigma,\kappa,z}$, the smooth function $a_{\b,\ga,N}$ as $a_{\b,\ga,N}(\rx,\zeta,\vth):=\int_{0}^1 (1-t)^{N} (\del^{(0,\b,\ga)}a) (\rx,t\zeta,\vth) \,dt$.
It is straightforward to check that the linear map $a\mapsto a_{\b,\ga,N}$ is continuous from $\Pi_{\sg,\ka,z}^{l,w,m}$ into $\Pi_{\sg,\ka,z}^{l-|\b|,|w|+\ka|\b|,m-|\ga|}$.

The following lemma shows that $\la$-quantization of amplitudes and symbols yields the same operators. This result of ``reduction" of amplitudes to symbols will be important for Theorem \ref{lambdainv} and thus, for a $\la$-invariant definition of pseudodifferential operators.  

\begin{lem}
\label{reduction} 
\noindent (i) For any $a\in \Pi_{\sigma,\kappa,z}^{l,w,m}$, $(\del^{0,\b,\b} a)_{\zeta=0} \in S^{l-|\b|,m-|\b|}_{\sigma,z}$ for any $n$-multi-index $\b$.  

\noindent (ii) Let $\Ga$ be as in Lemma \ref{noyauReste} and let $a\in \Pi_{\sigma,\kappa,z}^{l,w,m}$. Then for any symbol $s\in S^{l,m}_{\sigma,z}$ such that $s\sim \sum_{\b} \tfrac{(i/2\pi)^{|\b|}}{\b!} (\del^{0,\b,\b} a)_{\zeta=0}$, there is $r \in S^{-\infty}_{\sigma,z}$ such that $\Op_{\Ga}(a)=\Op_{\Ga}(s+r)$. In particular there exists an unique symbol $s(a)\in S^{l,m}_{\sg,z}$ such that $\Op_{\Ga}(a)=\Op_{\Ga}(s(a))$.  Moreover, we have $s(a)\sim \sum_{\b} \tfrac{(i/2\pi)^{|\b|}}{\b!} (\del^{0,\b,\b} a)_{\zeta=0}$.

\noindent (iii) Suppose that $(M,\exp,E,d\mu )$ has a $S_\sigma$-bounded geometry and $\psi$ is a $\O_M$-linearization. 
Let $a\in \Pi_{\sigma,\kappa,z}^{l,w,m}$, $\la\in [0,1]$ and $(z,\bfr)$ be given a frame. Then there exists an unique symbol $s_\la(a)\in S^{l,m}_{\sigma}$ such that $\Op_{\Ga_{\la,z,\bfr}}(a)=(\Op_{\la}(s_\la(a))_{z,\bfr}$. Moreover, we have $T_{z,\bfr,*}(s_\la(a))\sim \sum_{\b} \tfrac{(i/2\pi)^{|\b|}}{\b!} \mu^{-1} (\del^{0,\b,\b} a)_{\zeta=0}$.
\end{lem}
\begin{proof}
$(i)$ is a direct consequence of Lemma \ref{Amplisymb} $(i)$.

\noindent $(ii)$ Using a Taylor expansion of $a$ at $\zeta=0$, we find that for any $u\in \S(\R^{2n},L(E_z))$, $N\in \N^*$, $\langle \Op_{\Ga}(a),u\rangle=\sum_{0\leq |\b|\leq N} I_{\b} + \sum_{|\b|=N+1}\tfrac{N+1}{\b!} R_{\b,N}$ where 
\begin{align*}
&I_\b:= \int_{\R^{3n}} \zeta^\b e^{2\pi i \langle\vth,\zeta \rangle } \Tr\big(\tfrac{1}{\b!}(\del^{(0,\b,0)}a)_{\zeta=0}(\rx,\vth) \Ga(u)^*(\rx,\zeta)\big) d\zeta\,d\vth\,d\rx \, ,\\
&R_{\b,N}:= \int_{\R^{3n}} \zeta^\b e^{2\pi i \langle\vth,\zeta \rangle } \Tr\big(a_{\b,0,N}(\rx,\zeta,\vth)\, \Ga(u)^*(\rx,\zeta)\big) d\zeta\,d\vth\,d\rx \, .
\end{align*}
We get from Lemma \ref{amplireste} $(ii)$,
$$
I_\b= \int_{\R^{3n}} e^{2\pi i \langle\vth,\zeta \rangle } \Tr\big(\tfrac{(i/2\pi)^{|\b|}}{\b!}(\del^{(0,\b,\b)}a)_{\zeta=0}(\rx,\vth) \Ga(u)^*(\rx,\zeta)\big) d\zeta\,d\vth\,d\rx\, .
$$
Let $s\in S^{l,m}_{\sigma,z}$ be a symbol such that $s\sim \sum_{\b} \tfrac{(i/2\pi)^{|\b|}}{\b!} (\del^{0,\b,\b} a)_{\zeta=0}$. Then noting $s_N:=s-\sum_{|\b |\leq N} \tfrac{(i/2\pi)^{|\b|}}{\b!} (\del^{0,\b,\b} a)_{\zeta=0}\in S^{l-(N+1),m-(N+1)}_{\sigma,z}$, we find with Lemma \ref{amplireste} $(ii)$ that  $\Op_{\Ga}(a-s) = \Op_{\Ga}(r_N)$ where 
$$
r_N:= \sum_{|\b|=N+1} \tfrac{(N+1)(i/2\pi)^{N+1}}{\b!} a_{\b,\b,N} -s_N\, .
$$
We check that $r_N\in \Pi_{\sigma,\kappa,z}^{l-(N+1),w_{N},m-(N+1)}$ where $w_N=|w|+\kappa(N+1)$. Corollary \ref{correste} applied to $\Op_{\Ga}(a-s)$ now implies that there is $r \in S^{-\infty}_{\sigma,z}$ such that $\Op_{\Ga}(a)=\Op_{\Ga}(s+r)$. As a consequence, there exists $s(a)\in S^{l,m}_{\sg,z}$ such that $\Op_{\Ga}(a)=(\Op_{\Ga}(s(a))$. The unicity is a direct consequence of the fact that $\Op_{\Ga}=\Ga^*\circ \F_P^*$ on $\S'(\R^{2n},L(E_z))$. 

\noindent $(iii)$ Direct consequence of $(ii)$ and that fact that $(\Op_{\la}(s))_{z,\bfr}=\Op_{\Ga_{\la,z,\bfr}}(\mu_{z,\bfr} s_{z,\bfr})$.
\end{proof}

\subsection{$S_\sigma$-linearizations}\label{Ssigsec}

In order to have a full symbol-operator isomorphism, a polynomial control at infinity on the linearization is not enough. As we shall see, a stronger, ``amplitude-like" control on the $\psi_{z}^{\bfr}$ maps and a local equivalent of the $P_{x,\xi}$ parallel transport linear isomorphisms (see Remark \ref{remTP}) appears to be crucial for pseudodifferential calculus on $(M,\exp,E)$ and the $\la$-invariance (see Theorem \ref{lambdainv}).

We define $H_{\sigma,\kappa}^w(\mathfrak{E})$ (resp. $E_{\sigma,\kappa}^w(\mathfrak{E})$), where $w\in\R$, $\sigma\in [0,1]$ and $\kappa\geq 0$, as the space of smooth functions $g$ from $\R^{2n}$ into $\mathfrak{E}$ such that for any $2n$-multi-index $\nu$, there exists $C_\nu>0$
such that for any $(\rx,\zeta)\in \R^{2n}$, $\norm{\del^{\nu} g (\rx,\zeta)}\leq C_{\nu} \langle \rx\rangle^{-\sigma (|\nu|-1)}\langle \zeta \rangle^{w+\kappa(|\nu|-1)} $  (if $\nu\neq0$) (resp. $\norm{\del^{\nu} g (\rx,\zeta)}\leq C_{\nu} \langle \rx\rangle^{-\sigma |\nu|}\langle \zeta \rangle^{w+\kappa|\nu|})$. We note $H_{\sigma,\kappa}(\mathfrak{E}):=\cup_{w\in \R} H_{\sigma,\kappa}^w(\mathfrak{E})$, $H_{\sigma}(\mathfrak{E}):=\cup_{\ka\geq 0 } H_{\sigma,\kappa}(\mathfrak{E})$, $E_{\sigma,\kappa}(\mathfrak{E})=\cup_{w\in \R} E_{\sigma,\kappa}^w(\mathfrak{E})$ and $E_{\sigma}(\mathfrak{E})=\cup_{\ka\geq 0} E_{\sigma,\kappa}(\mathfrak{E})$.
Remark that by Leibniz rule, $E_{\sigma,\kappa}(\R)$ and $E_{\sigma,\kappa}(\mathcal{M}_{p}(\R))$ are $\R$-algebras (graduated by the parameter $w$) while $E_{\sigma,\kappa,z}:=E_{\sigma,\kappa}(L(E_z))$ is a $\C$-algebra (under pointwise matricial product). Thus, if $P\in E_{\sigma,\kappa}(\M_p(\R))$, then $\det P \in E_{\sg,\ka}(\R)$. Note also that $f\in H_{\sg,\ka}(\mathfrak{E})$ if and only if for any $i\in \set{1,\cdots,2n}$, $\del_i f \in E_{\sg,\ka}(\mathfrak{E})$. In particular, $f\in H_{\sg,\ka}(\R^p)$ if and only if $df:=(\rx,\zeta)\mapsto (df)_{\rx,\zeta}$ is in $E_{\sg,\ka}(\M_{p,2n}(\R))$. As a consequence, if $f\in H_{\sg,\ka}(\R^{2n})$, its Jacobian determinant $J(g)$ is in $E_{\sg,\ka}(\R)$. Note that any function in $E_{\sg,\ka}^0(\mathfrak{E})$ is bounded and if $f\in H_{\sg,\ka}^0(\mathfrak{E})$ then there is $C>0$ such that $\norm {f(\rx,\zeta)}_{\mathfrak{E}} \leq C \langle \rx,\zeta\rangle$ for any $(\rx,\zeta)\in \R^{2n}$.
The following lemma will give us the behaviour of the  $E_{\sg,\ka}$ and $H_{\sg,\ka}$ spaces under composition.

\begin{lem}
\label{lemGsigma}
(i) Let $f\in H_{\sigma,\kappa}^{w'}(\mathfrak{E})$ (resp. $E^{w'}_{\sigma,\kappa}(\mathfrak{E})$) and $g\in H^w_{\sigma,\kappa}(\R^{2n})$ such that there exists $C,c>0$, $r\geq 0$, such that $\langle g_1(\rx,\zeta) \rangle \geq c \langle \rx\rangle \langle \zeta \rangle^{-r}$ (if $\sg\neq 0$) and $\langle g_2(\rx,\zeta)\rangle \leq C\langle \zeta \rangle$ for any $(\rx,\zeta)\in \R^{2n}$, where $g=:(g_1,g_2)$. Then $f\circ g \in H_{\sg,\ka+|w|+r\sg}^{|w|+|w'|}(\mathfrak{E})$ (resp. $E_{\sg,\ka+|w|+r\sg}^{|w'|}(\mathfrak{E})$).

\noindent (ii) If $P\in E^w_{\sg,\ka}(\M_{n}(\R))$, then $(\rx,\zeta)\mapsto P_{\rx,\zeta}(\zeta) \in H^{w+\ka+1}_{\sg,\ka}(\R^n)$.

\noindent (iii) Let $f\in G_{\sigma}(\R^{n},\mathfrak{E})$  and $g\in H^w_{\sigma,\kappa}(\R^n)$ such that there exists $c>0$, $r\geq 0$, such that, if $\sg\neq 0$, $\langle g(\rx,\zeta) \rangle \geq c \langle \rx\rangle \langle \zeta \rangle^{-r}$ for any $(\rx,\zeta)\in \R^{2n}$. Then $f\circ g\in H_{\sg,\max\set{r\sg,\ka}+|w|}^{|w|}(\mathfrak{E})$. Moreover, if $f\in G_{\sigma}(\R^n,\R^p)$, then $df\circ g \in E_{\sg,\max\set{r\sg,\ka}+|w|}^{0}(\M_{p,n}(\R))$.
\end{lem}
\begin{proof}$(i)$ The Faa di Bruno formula yields for any $2n$-multi-index $\nu\neq 0$,
\begin{equation}
\label{eqFdB}
\del^\nu (f\circ g) = \sum_{1\leq |\la|\leq |\nu|} (\del^\la f)\circ g\  P_{\nu,\la}(g)
\end{equation}
where $P_{\nu,\la}(g)$ is a linear combination (with coefficients independant of $f$ and $g$) of functions of the form $\prod_{j=1}^s (\del^{l^j} g)^{k^j}$ where $s\in \set{1,\cdots ,|\nu|}$.  
The $k^j$ and $l^j$ are $2n$-multi-indices (for $1\leq j\leq s$) such that $|k^j|>0$, $|l^j|>0$, $\sum_{j=1}^s k^j = \la$ and $\sum_{j=1}^s |k^j| l^j= \nu$. As a consequence, since $g\in H_{\sigma,\kappa}^w(\R^{2n})$, we see that for each $\nu,\la$ with $1\leq |\la|\leq |\nu|$ there exists $C_{\nu,\la}>0$ such that for any $(\rx,\zeta)\in \R^n$,
\begin{equation}\label{pnulag}
|P_{\nu,\la}(g) (\rx,\zeta)|\leq C_{\nu,\la} \langle \rx\rangle^{-\sigma(|\nu|-|\la|)} \langle\zeta \rangle^{w|\la|+\kappa(|\nu|-|\la|)} \, . 
\end{equation}
Moreover, since $f\in H_{\sigma,\kappa}^{w'}(\R^{2n})$ (resp. $E_{\sigma,\kappa}^{w'}(\R^{2n})$), there is $C'_\la>0$ such that for any $(\rx,\zeta)\in \R^{2n}$, the estimate $\norm{(\del^\la f) \circ g(\rx,\zeta)}\leq C'_\la \langle \rx\rangle ^{-\sigma(|\la|-1)} \langle \zeta \rangle^{|w'|+(\kappa+r\sg)(|\la|-1)}$ (resp.  
 $\norm{(\del^\la f) \circ g(\rx,\zeta)}\leq C'_\la \langle \rx\rangle ^{-\sigma|\la|} \langle \zeta \rangle^{|w'|+(\kappa+r\sg)|\la|}$) is valid. We deduce then from (\ref{eqFdB}) and (\ref{pnulag}) that $f\circ g$ belongs to $H_{\sg,\ka+|w|+r\sg}^{w+|w'|}(\mathfrak{E})$ (resp. $E_{\sg,\ka+|w|+r\sg}^{|w'|}(\mathfrak{E})$).
 
\noindent $(ii)$ We note $P^{i,j}_{\rx,\zeta}$ the matrix entries of $P_{\rx,\zeta}$. Each component $(f^{i})_{1\leq i \leq n}$ of the map $f:=(\rx,\zeta)\mapsto P_{\rx,\zeta}(\zeta)$ is of the form $f^i= \sum_{j=1}^n P^{i,j}\, \zeta_j$. It is straightforward to check that the applications $(\rx,\zeta)\mapsto \zeta_j$ satify for any $\nu\in \N^{2n}$, $\del^\nu \zeta_j = \O(\langle \zeta\rangle^{1-|\nu|}\langle \rx\rangle^{\sg(1-|\nu|)})$. The result now follows from an application of the Leibniz rule.

 \noindent $(iii)$ Following the proof of $(i)$, (\ref{pnulag}) is still valid, this time with $\la$ as $n$-multi-indices and $\nu$ as $2n$-multi-indices with $1\leq |\la|\leq |\nu|$. Using the fact that $\langle g(\rx,\zeta) \rangle \geq c \langle \rx\rangle \langle \zeta \rangle^{-r}$ for any $(\rx,\zeta)\in \R^{2n}$, we obtain the following estimate
 $$
 \norm{(\del^\la f) \circ g(\rx,\zeta)}\leq C'_\la \langle \rx\rangle ^{-\sigma(|\la|-1)} \langle \zeta \rangle^{r\sigma(|\la|-1)}\leq C'_\la \langle \rx\rangle ^{-\sigma(|\la|-1)} \langle \zeta \rangle^{\max\set{r\sg,\ka}(|\la|-1)}
 $$
which, with (\ref{pnulag}) and (\ref{eqFdB}), yields $f\circ g$ belongs to $H_{\sg,\max\set{r\sg,\ka}+|w|}^{|w|}(\mathfrak{E})$. The fact that $df\circ g$ is in $E_{\sg,\max\set{r\sg,\ka}+|w|}^{0}(\M_{p,n}(\R))$ when $f\in G_{\sg}(\R^n,\R^p)$ is based on the same argument.
\end{proof}

The $H_{\sg,\ka}$ and $E_{\sg,\ka}$ spaces are related to the symbol and amplitude spaces by the following lemma.

\begin{lem}
\label{HEamp}
\noindent (i) If $f\in E^w_{\sg,\ka,z}$, then $(\rx,\zeta,\vth)\mapsto f(\rx,\zeta)$ is in $\Pi^{0,w,0}_{\sg,\ka,z}$.

\noindent (ii) Let $s\in S^{l,m}_{\sigma,z}$, $m\in H^w_{\sg,\ka}(\R^n)$ such that there exist $C,c,r>0$ such that, if $\sg\neq 0$, for any $(\rx,\zeta)\in \R^{2n}$, $ c \langle \rx\rangle \langle \zeta \rangle^{-r}\leq \langle m(\rx,\zeta)\rangle \leq C \langle \rx \rangle \langle \zeta\rangle^{r}$, and $P\in E^{0}_{\sg,\ka}(\M_n(\R))$ such that such that for any $(\rx,\zeta,\vth)\in \R^{3n}$, $\langle P_{\rx,\zeta}(\vth)\rangle \geq c \langle \vth \rangle$. Then $(\rx,\zeta,\vth)\mapsto s(m(\rx,\zeta),P_{\rx,\zeta}(\vth))$ is in $\Pi^{l,\sigma r|l|,m}_{\sigma,\ka+|\sg r-\ka+w|,z}$.

\noindent (iii) If $s\in S_\sigma(\R^n)$, $m\in H^w_{\sg,\ka}(\R^n)$ such that, if $\sg\neq 0$, there exists $c,r>0$ such that for any $(\rx,\zeta)\in \R^{2n}$ $\langle m(\rx,\zeta)\rangle \geq c \langle \rx\rangle \langle \zeta \rangle^{-r}$, then $(\rx,\zeta,\vth)\mapsto s(m(\rx,\zeta)) \Id_{L(E_z)}$ is in $\Pi_{\sg,\ka+|\sg r-\ka+w|,z}^{0,0,0}$.

\noindent (iv) If $a \in \Pi^{l,w,m}_{\sg,\ka,z}$ and $P\in E^{0}_{\sg,\ka}(\M_n(\R))$ is such that such that there is $c> 0$ such that for any $(\rx,\zeta,\vth)\in \R^{3n}$, $\langle P_{\rx,\zeta}(\vth)\rangle \geq c \langle \vth \rangle$, then $a_P:(\rx,\zeta,\vth)\mapsto a(\rx,\zeta,P_{\rx,\zeta}(\vth)) \in \Pi^{l,w,m}_{\sg,\ka,z}$.
\end{lem}
\begin{proof}
\noindent $(i)$ is straightforward.

\noindent $(ii)$ Let us note $g(\rx,\zeta,\vth):=(m(\rx,\zeta),P_{\rx,\zeta}(\vth))$. For any $i,j\in \set{1,\cdots, n}$, we note $P_{\rx,\zeta}^{i,j}$ the $(i,j)$ matrix entry of $P_{\rx,\zeta}$. Since $P\in E^{0}_{\sg,\ka}(\mathcal{M}_n(\R))$, we have $P_{\cdot,\cdot}^{i,j}\in E^{0}_{\sg,\ka}(\R)$. Faa di Bruno formula in Theorem \ref{FaaCS} yields for any $\nu\neq0$ 
\begin{equation}
\label{delnusg}
\del^{\nu} (s\circ g)  =  \sum_{1\leq |\la|\leq |\nu|} (P_{\nu,\la}(g)) \ (\del^{\la}s)\circ g  
\end{equation}
where $P_{\nu,\la}(g)$ is a linear combination of terms of the form $\prod_{j=1}^{s} (\del^{l^j} g)^{k^j} $, where $1\leq s\leq |\nu|$,  
the $k^j$ (resp. $l^j$) are $2n$-multi-indices (resp. $3n$-multi-indices) with $|k^j|>0$, $|l^j|>0$, $\sum_{j=1}^s k^j = \la$ and $\sum_{j=1}^s|k^j|l^j=\nu$. 
Let us note $l^{j}=:(l^{j,1},l^{j,2},l^{j,3})$, $k^{j}=:(k^{j,1},k^{j,2})$ where $l^{j,1},l^{j,2},l^{j,3},k^{j,1},k^{j,2}$ are $n$-multi-indices. We have, noting $Q(\rx,\zeta,\vth):=(\rx,\zeta)$,
$$
(\del^{l^j}g)^{k^j}= \prod_{i=1}^n (\delta_{l^{j,3},0}(\del^{(l^{j,1},l^{j,2})} m)_i\circ Q )^{k^{j,1}_i}\ \prod_{i=1}^n \big(\sum_{k=1}^n \del^{(l^{j,1},l^{j,2})}P^{i,k}_{\cdot,\cdot}\ \del^{l^{j,3}} \vth_{k}\big)^{k^{j,2}_i}
$$
and we get, for a given $s$, $(l^{j})$, $(k^j)$ such that $(\del^{l^j} g)^{k^j} \neq 0$ for all $1\leq j\leq s$,
\begin{align*}
&\text{ if } l^{j,3}=0\, ,  \qquad   (\del^{l^j}g)^{k^j} = \O(\langle \rx\rangle^{-\sigma|l^j||k^{j}|+\sigma|k^{j,1}|}\langle \zeta\rangle^{\ka|l^j||k^{j}|-\ka|k^{j,1}|+w|k^{j,1}|}\langle \vth\rangle^{|k^{j,2}|})\, , \\ 
&\text{ if } |l^{j,3}|=1\, ,\qquad k^{j,1}=0 \text{ and } \    (\del^{l^j}g)^{k^j} = \O(\langle \rx \rangle^{-\sigma |l^j||k^{j}|+\sg|k^j|}\langle \zeta\rangle^{\ka |l^j||k^j|-\ka|k^j|})\, .
\end{align*}

The case is $|l^{j,3}|>1$ is excluded since $k^{j}\neq 0$  and $(\del^{l^j}g)^{k^j} \neq 0$.
By permutation on the $j$ indices, we can suppose as in the proof of Lemma \ref{cchange} that for $1\leq j\leq j_1-1$, we have $l^{j,3}=0$ and for $j_1\leq j\leq s$, we have $|l^{j,3}|=1$, where $1\leq j_1\leq s+1$. Thus, we get 
\begin{align*}
&\prod_{j=1}^s (\del^{l^j} g)^{k^j} = \O(\langle\rx \rangle^{-\sg(\sum_{j=1}^{s}(|l^j|-1)|k^j| + \sum_{j=1}^{j_1-1}|k^{j,2}|)}\\
&\hspace{3cm}\times\langle \zeta\rangle^{w\sum_{j=1}^{s}|k^{j,1}|+\kappa(\sum_{j=1}^{s}(|l^j|-1)|k^j| + \sum_{j=1}^{j_1-1}|k^{j,2}|)}\langle \vth \rangle^{\sum_{j=1}^{j_1-1}|k^{j,2}|})\, .
\end{align*}
We check that $\sum_{j=1}^{j_1-1} |k^{j,2}|=|\la^2|-|\ga|$ and $\sum_{j=1}^{s}(|l^j|-1)|k^j| = |\nu|-|\la|$ where $\la=(\la^1,\la^2)$ and $\nu=(\a,\b,\ga)$. As a consequence, 
\begin{equation}
\label{pnu2}
P_{\nu,\la}(g) = \O(\langle \rx\rangle^{-\sigma(|\a+\b|-|\la^1|)}\langle \zeta\rangle^{w|\la^1|+\ka(|\a+\b|-|\la^1|)} \langle \vth \rangle^{|\la^2|-|\ga|}) \, .
\end{equation}
Since there exist $C,c>0$ such that for any $(\rx,\zeta)\in \R^{2n}$ $\langle m(\rx,\zeta)\rangle \leq C \langle \rx\rangle \langle\zeta\rangle^{r}$ and $\langle m(\rx,\zeta)\rangle \geq c \langle \rx\rangle \langle \zeta\rangle^{-r}$, we see that there is $K_\nu>0$ such that for any $1\leq |\la|\leq |\nu|$ and any $(\rx,\zeta)\in \R^{2n}$,
$\langle m(\rx,\zeta)\rangle^{\sigma(l-|\la^1|)}\leq K_\nu \langle \rx \rangle^{\sg(l-|\la^1|)} \langle \zeta \rangle^{\sg r|l|+\sg r|\la^1|}$. As a consequence, we see that there is $C_\nu>0$ such that for any $1\leq |\la|\leq |\nu|$ and any $(\rx,\zeta,\vth)\in \R^{3n}$,
$$
\norm{(\del^\la s)\circ g (\rx,\zeta,\vth)}_{L(E_z)} \leq C_\nu \langle \rx \rangle^{\sg(l-|\la^1|)} \langle \zeta \rangle^{\sg r|l|+\sg r|\la^1|} \langle \vth \rangle^{m-|\la^2|}.
$$
so, since we can reduce the sum in (\ref{delnusg}) to $2n$-multi-indices $\la$ such that $|\la^2|\geq |\ga|$ (and thus $|\la^1|\leq |\a+\b|$), we obtain the result from (\ref{pnu2}) and a straightforward verification of the case $\nu=0$. 

\noindent $(iii)$ is obtain exactly as $(ii)$ (with $P_{\rx,\zeta}=\Id$), since $(\rx,\zeta)\mapsto \mu_{z,\bfr}(\rx)\Id_{L(E_z)} \in S^{0,0}_{\sigma,z}$. The hypothesis $ m(\rx,\zeta) = \O (\langle \rx\rangle \langle \zeta\rangle^r)$ is not necessary since $l=0$ here.

\noindent $(iv)$ We have, noting $g(\rx,\zeta,\vth):=(\rx,\zeta,P_{\rx,\zeta}(\vth))$, for any $3n$-multi-indices $\nu\neq 0$, $1\leq |\nu'|\leq |\nu|$, 
$P_{\nu,\nu'}(g)$ as a linear combination of terms of the form $\prod_{j=1}^s (\del^{l^j}g)^{k^j}$, with $\sum_{j=1}^{s}|k^j|l^j=\nu$ and $\sum_{j=1}^s k^j=\nu'$, noting $k^j=(k^{j,1},k^{j,2}), l^{j}=(l^{j,1},l^{j,2})$, where $k^{j,1}$ and $l^{j,1}$ are $2n$-multi-indices, we get, following the proof of $(ii)$, 
$$
P_{\nu,\nu'}(g)=\O(\langle \rx\rangle^{-\sigma(|\a+\b|-|\a'+\b'|)}\langle \zeta\rangle^{\ka(|\a+\b|-|\a'+\b'|)} \langle \vth \rangle^{|\ga'|-|\ga|})\, .
$$
Since $P_{\rx,\zeta}=\O(1)$ and $\langle P_{\rx,\zeta}(\vth)\rangle \geq \eps\langle \vth\rangle$ we get the result.
\end{proof}

\begin{defn}
\label{sprime} Let $\sigma \in [0,1]$ and $\psi$ a linearization on $(M,\exp,E,d \mu)$. We say that $\psi$ is a $S_\sg$-linearization if for any frame $(z,\bfr)$, 
there is $\ka_{z,\bfr}\geq 0$ such that

\noindent $(i)$ $\psi_z^\bfr \in H_{\sg,\ka_{z,\bfr}}(\R^n)$ with ${\psi_z^\bfr}(\rx,\zeta)=\O(\langle \rx\rangle \langle \zeta\rangle^{r})$ for a $r\geq 1$ and $\ol{\psi_z^\bfr} \in \O_M(\R^{2n},\R^n)$ ,

\noindent $(ii)$ there is $P^{z,\bfr}\in C^\infty(\R^{2n}, GL_n(\R))$ such that $P^{z,\bfr}$ and  $(P^{z,\bfr})^{-1}$ are in $E^0_{\sg,\ka_{z,\bfr}}(\M_n(\R))$, and for any $(\rx,\zeta)\in \R^{2n}$, $P_{\rx,\zeta}^{z,\bfr}(\zeta)=\Ups_{1,T}^{z,\bfr}(\rx,\zeta)$ and $P_{\rx,0}^{z,\bfr}=\Id_{\R^n}$. 

\noindent $(iii)$ $\tau_{1}^{z,\bfr}$ and $(\tau_1^{z,\bfr})^{-1}$ are in $E^0_{\sg,\ka_{z,\bfr}}(L(E_z))$.

\noindent We shall say that the combo $(M,\exp,E,d\mu,\psi)$ has a $S_{\sigma}$-bounded geometry if this is the case of $(M,\exp,E,d\mu)$ and $\psi$ is a $S_\sg$-linearization. 
\end{defn}

It is clear that a $S_{\sg}$-linearization is also a $\O_M$-linearization. Moreover, we can check, in case of $S_\sg$ bounded geometry, we check the properties $(i)$, $(ii)$ and $(iii)$ in just one frame:

\begin{lem} If $(M,\exp,E,d\mu)$ has a $S_\sg$-bounded geometry and $\psi$ is a linearization such that there exists $(z_0,\bfr_0)$, $\ka_{z_0,\bfr_0}\geq 0$, such that the functions $\psi_{z_0}^{\bfr_0}$, ${\ol \psi}_{z_0}^{\bfr_0}$ satisfy $(i)$, $(ii)$ and $(iii)$, then $\psi$ is a $S_\sg$-linearization.
\end{lem}
\begin{proof} This follows from applications of Lemma \ref{lemGsigma}.
\end{proof}

\begin{rem} The condition $(ii)$ in Definition \ref{sprime} encodes an abstract parallel transport isomorphisms in normal coordinates. Indeed, in the case where the linearization $\psi$ is derived from a connection on $M$, the $GL_n(\R)$-valued smooth functions on $\R^{2n}$: $P^{z,\bfr}:=(\rx,\zeta)\mapsto  M^\bfr_{z,\exp\circ (n_{z,T}^\bfr)^{-1}(\rx,\zeta)} P_{(n_{z,T}^\bfr)^{-1}(\rx,\zeta)} ( M^\bfr_{z,(n_{z}^\bfr)^{-1}(\rx)})^{-1}$ where the applications $P_{x,\xi}$ are the parallel transport isomorphisms on the tangent bundle (see Remark \ref{remTP}), satisfy for any $(\rx,\zeta)\in \R^{2n}$, $P_{\rx,\zeta}^{z,\bfr}(\zeta)=\Ups_{1,T}^{z,\bfr}(\rx,\zeta)$ and $P_{\rx,0}^{z,\bfr}=\Id_{\R^n}$. Thus, in this case, $(ii)$ is satisfied if  $P^{z,\bfr}$ and  $(P^{z,\bfr})^{-1}$ are in $E_{\sg,\ka_{z,\bfr}}^0(\M_n(\R))$ for a $\ka_{z,\bfr}\geq 0$.
\end{rem}

Remark that for any $t\in \R$ and $(\rx,\zeta)\in \R^{2n}$, if $P^{z,\bfr}\in C^\infty(\R^{2n}, GL_n(\R))$ satisfies $(ii)$, then $P^{z,\bfr}_{\rx,t\zeta}(\zeta)= \Ups_{t,T}^{z,\bfr}(\rx,\zeta)$. We shall note $P^{z,\bfr}_t:=(\rx,\zeta)\mapsto P^{z,\bfr}_{\rx,t\zeta}$, so that $P_1^{z,\bfr}=P^{z,\bfr}$ and $P_0^{z,\bfr}=\Id_{\R^n}$. Thus, $\Ups_{t,z,\bfr}(\rx,\zeta)=(\psi_{z}^\bfr(\rx,t\zeta),P^{z,\bfr}_{t,\rx,\zeta}(\zeta))$ and we define the following diffeomorphism on $\R^{3n}$, 
\begin{equation}
\label{xidef} \Xi_{t,z,\bfr}:=(\rx,\zeta,\vth)\mapsto (\Ups_{t,z,\bfr}(\rx,\zeta),\wt P_{t,\rx,\zeta}^{z,\bfr}(\vth))\, .
\end{equation}
We also define the $\R^{2n}$-valued function $\wh\Xi_{t,z,\bfr}:(\rx,\zeta,\vth)\mapsto (\psi_z^\bfr(\rx,t\zeta),\wt P_{t,\rx,\zeta}^{z,\bfr}(\vth))$.
We check that $J(\Xi_{t,z,\bfr}) =J(\Ups_{t,z,\bfr})\, (\det (P_{t}^{z,\bfr})^{-1})$ and $J(\Xi_{t,z,\bfr}^{-1})=J(\Ups_{-t,z,\bfr})\, (\det (P_{t}^{z,\bfr}\circ \Ups_{-t,z,\bfr}))$. Note also that for any $(\rx,\ry)\in \R^{2n}$, $\ol{\psi_{z}^\bfr}(\ry,\rx)= -P_{\rx,\ol{\psi_{z}^\bfr}(\rx,\ry)}^{z,\bfr}(\ol{\psi_{z}^\bfr}(\rx,\ry))$.
\begin{lem}
\label{lem-Phi-la}
Let $(z,\bfr)$ be a given frame, $\la, \la'\in[0,1]$ and $t\in [-1,1]$. Suppose also that $(M,\exp,E, d\mu ,\psi)$ has a $S_\sigma$-bounded geometry. Then

\noindent (i) $P_{t}^{z,\bfr}$, $(P_{t}^{z,\bfr})^{-1}$ are in $E_{\sigma,\kappa_{z,\bfr}}^0(\mathcal{M}_n(\R))$, and $\tau_{t}^{z,\bfr}$, $(\tau_{t}^{z,\bfr})^{-1}$ are in $E_{\sigma,\kappa_{z,\bfr}}^0(L(E_z))$.

\noindent (ii) $m_{t}^{z,\bfr}:=\psi_z^\bfr\circ I_{1,t} \in H_{\sigma,\kappa_{z,\bfr}}(\R^n)$ and there is $c>0$, $r\geq 1$ such that for any $(\rx,\zeta)\in \R^{2n}$, $\langle m_{t}^{z,\bfr} (\rx,\zeta)\rangle \geq c \langle \rx \rangle \langle \zeta\rangle^{-r}$. 

\noindent (iii) There is $c,\eps>0$ such that for any $(\rx,\zeta)\in \R^{2n}$, $\langle \psi_z^\bfr(\rx,\zeta) \rangle \geq c \langle \zeta \rangle^{\eps} \langle \rx \rangle^{-1}$.  

\noindent (iv) $\Phi_{\la,z,\bfr}\in H_{\sigma,\kappa_{z,\bfr}}(\R^{2n})$. In particular $J_{\la,z,\bfr} \in E_{\sigma,\kappa_{z,\bfr}}(\R)$.

\noindent (v) $\Ups_{t,z,\bfr} \in H_{\sigma,\kappa_{z,\bfr}}(\R^{2n})$. In particular $J({\Ups_{t,z,\bfr}}) \in E_{\sigma,\kappa_{z,\bfr}}(\R)$. Moreover, there is $C>0$ such that $\langle (\Ups_{t,T}^{z,\bfr})(\rx,\zeta)\rangle \leq C\langle \zeta \rangle$ for any $(\rx,\zeta)\in \R^{2n}$.

\noindent (vi) $J(\Xi_{t,z,\bfr})$ and $J(\Xi_{t,z,\bfr}^{-1})$ are in $E_{\sigma,\kappa_{z,\bfr}}(\R)$.
\end{lem}
\begin{proof}
$(i)$ The case $t=0$ is obvious. Suppose $t\neq 0$. Since $P_{t}^{z,\bfr}=P^{z,\bfr}\circ I_{1,t}$ and $I_{1,t}\in H_{\sg,\ka_{z,\bfr}}^0$ the result follows from Lemma \ref{lemGsigma} $(i)$. The same argument is applied to $(P_{t}^{z,\bfr})^{-1}$, $\tau_{t}^{z,\bfr}$ and $(\tau_{t}^{z,\bfr})^{-1}$.

\noindent $(ii)$ We shall use the shorthand $m_t:=m_t^{z,\bfr}$. In the case $t=0$, $m_0= \pi_1$, so we obtain the result. Suppose $t\neq 0$. In that case Lemma \ref{lemGsigma} $(i)$ entails that $m_t\in H_{\sg,\ka_{z,\bfr}}(\R^n)$. Since $\Ups_{t,z,\bfr}=(m_{t},\Ups_{t,T}^{z,\bfr})$, we see that $\langle \Ups_{t,z,\bfr}(\rx,\zeta)\rangle=\O(\langle \rx\rangle \langle \zeta\rangle^r)$ for a $r\geq 1$. Thus, there is $C>0$ such that for any $(\rx,\zeta)\in \R^{2n}$, we have $\langle m_t(\rx,\zeta)\rangle \langle P_{t,\rx,\zeta}(\zeta)\rangle^r \geq C \langle \rx,\zeta\rangle$. Since there is $K>0$ such that for any $(\rx,\zeta)\in \R^{2n}$, $\langle P_{t,\rx,\zeta}^{z,\bfr}(\zeta)\rangle \leq K \langle \zeta\rangle$, we obtain the desired estimate.

\noindent $(iii)$ $V:=(\pi_1,\psi_{z}^{\bfr})$ is a diffeomorphism on $\R^{2n}$ with inverse $V^{-1}=(\pi_1,\ol{\psi_z^\bfr})$. Since $\ol{\psi_{z}^\bfr}=\O(\langle \rx,\ry\rangle^{r})$ for a $r\geq 1$ by hypothesis, we see that there is $c>0$ such that 
$\langle \rx,\psi_z^\bfr(\rx,\zeta)\rangle\geq c \langle \rx,\zeta\rangle$ for any $(\rx,\zeta)\in \R^{2n}$. This yields the result.

\noindent $(iv)$ Direct consequence of $(ii)$ and the fact that $\Phi_{\la,z,\bfr}=(m_{\la},m_{\la-1})$.

\noindent $(v)$ follows from a straithforward application of $(ii)$, Lemma \ref{lemGsigma} $(ii)$ and the fact that for any $(\rx,\zeta)\in \R^{2n}$, $\Ups_{t,z,\bfr}(\rx,\zeta)=(m_{t}(\rx,\zeta),P_{t,\rx,\zeta}^{z,\bfr}(\zeta))$.

\noindent $(vi)$ By $(i)$, $(v)$ and and Lemma \ref{lemGsigma} $(i)$, $P_{t}^{z,\bfr}\circ \Ups_{-t,z,\bfr} \in E^0_{\sg,\ka}(\M_n(\R))$. Thus the result follows from $(i)$, $(v)$, and the formulas $J(\Xi_{t,z,\bfr}) =J(\Ups_{t,z,\bfr})\, (\det (P_{t}^{z,\bfr})^{-1})$ and $J(\Xi_{t,z,\bfr}^{-1})=J(\Ups_{-t,z,\bfr})\, (\det (P_{t}^{z,\bfr}\circ \Ups_{-t,z,\bfr}))$.
\end{proof}

\subsection{Pseudodifferential operators}\label{pdosection}

\begin{assum} We suppose in this section and until section \ref{exsec} that $(M,\exp,E,d\mu,\psi)$ has a $S_{\sigma}$-bounded geometry.
\end{assum}

\begin{defn} A pseudodifferential operator of order $l,m$ and type $\sigma$ is an element of $\Psi_\sigma^{l,m}:=\mathfrak{Op}_\la(S^{l,m}_\sigma)$, where $\la\in [0,1]$.
\end{defn}
By Lemma \ref{slmdistr}, $S^{l,m}_{\sigma}$ can be seen as included in $\S'(T^*M, L(E))$, so $\mathfrak{Op}_\la(S^{l,m}_\sigma)$ is well defined. 
The following theorem shows that it does not depend on $\la$, and thus justify the notation $\Psi_\sigma^{l,m}$.  We note  $\tau_{R}^{\la,\la'}:=(\tau_{\la}^{z,\bfr})^{-1}\circ \Ups_{\la'-\la,z,\bfr}\,\tau_{\la'}^{z,\bfr}$ and $\tau_{L}^{\la,\la'}:=(\tau_{\la'-1}^{z,\bfr})^{-1}\tau_{\la-1}^{z,\bfr}\circ \Ups_{\la'-\la}^{z,\bfr}$. If $\psi=\exp$, we have  $\tau_{R}^{\la,\la'}=\tau_{R,\la'-\la}$ and $\tau_{L}^{\la,\la'}=(\tau_{L,\la'-\la})^{-1}$ where $\tau_{L,t}:= \tau_{t}^{z,\bfr}$ if $t\neq 1$ and $\tau_{L,t}:=(\tau_{-1}^{z,\bfr})^{-1}\circ \Ups_{1,z,\bfr}$ if $t=1$, and $\tau_{R,t}:=\tau_{t}^{z,\bfr}$ if $t\neq -1$ and $\tau_{R,t}:=(\tau_{1}^{z,\bfr})^{-1}\circ \Ups_{-1,z,\bfr}$ if $t=-1$.

\begin{thm} 
\label{lambdainv}
Let $\la,\la'\in [0,1]$ and $K= \mathfrak{Op}_\la(a)$, with $a\in S^{l,m}_\sigma$.  Then there exists (an unique) $a'\in S^{l,m}_\sigma$ such that  $K=\mathfrak{Op}_{\la'}(a')$. Moreover, for any frame $(z,\bfr)$, 
$$
a'_{z,\bfr}\sim \sum_{\b} \tfrac{(i/2\pi)^{|\b|}}{\b!} \big(\del^{(0,\b,\b)} \tau_{L}^{\la,\la'} a_{\la'-\la}^{z,\bfr} \tau_{R}^{\la,\la'}\big)_{\zeta=0} 
$$
where $a_{z,\bfr}:=T_{z,\bfr,*}(a)$, $a'_{z,\bfr}:=T_{z,\bfr,*}(a')$, and $a_{t}^{z,\bfr}$ is the amplitude defined for any $t\in [-1,1]$ as
\begin{align*}
a_{t}^{z,\bfr}(\rx,\zeta,\vth):= \tfrac{\mu_{z,\bfr}(m_{t}^{z,\bfr}(\rx,\zeta))}{\mu_{z,\bfr}(\rx)}|J\Xi_{t,z,\bfr}(\rx,\zeta)|\,(a_{z,\bfr}\circ \wh\Xi_{t,z,\bfr}(\rx,\zeta,\vth))\, .
\end{align*}
\end{thm}
\begin{proof} 
Let us fix a frame $(z,\bfr)$ and note $a_{z,\bfr}:=T_{z,\bfr,*}(a)$. We saw in Remark \ref{Oplien} that $\Op_{\la}(a)_{z,\bfr}=\Op_{\Ga_{\la,z,\bfr}}(\mu a_{z,\bfr}))$. Thus, for any $u\in \S(M\times M, L(E))$, we have with $u_{z,\bfr}:=T_{z,\bfr,M^2}(u)\in \S(\R^{2n},L(E_z))$, 
\begin{align*}
\langle K,u\rangle =\int_{\R^{3n}}e^{2\pi i\langle \vth,\zeta\rangle}\Tr\big(\mu a_{z,\bfr}(\rx,\vth)\, (\Ga_{\la,z,\bfr}(u_{z,\bfr})(\rx,\zeta))^*\big)\, \,d\zeta\,d\vth\,d\rx\, .
\end{align*}
Suppose that $m\leq -2n$ so that the integral is absolutely convergent. We now proceed to the global change of variables provided by the diffeomorphism $\Xi_{\la'-\la}^{z,\bfr}$ of $\R^{3n}$ ($\Xi_{t,z,\bfr}$ is defined at (\ref{xidef})). We get $\langle K,u \rangle = \langle \Op_{\la',z,\bfr}(\mu  \tau_{L}^{\la,\la'} a_{\la'-\la}^{z,\bfr} \tau_{R}^{\la,\la'}),u_{z,\bfr}\rangle$. We check with Lemmas \ref{lem-Phi-la} and \ref{HEamp} that $ \tau_{L}^{\la,\la'} a_{\la'-\la}^{z,\bfr} \tau_{R}^{\la,\la'}$ is an amplitude in $\Pi_{\sigma,\ka,z}^{l,w,m}$ for a $\ka \geq 0$ and a $w\in \R$. We also see that the linear map $a_{z,\bfr}\mapsto \mu  \tau_{L}^{\la,\la'} a_{\la'-\la}^{z,\bfr} \tau_{R}^{\la,\la'}$ is continuous on $S_{\sg,z}^{l,m}$, which yields, using Proposition \ref{ampliOP} $(ii)$ and the density result of Lemma \ref{toposymbol}, the equality $\langle K,u \rangle = \langle \Op_{\la',z,\bfr}(\mu  \tau_{L}^{\la,\la'} a_{\la'-\la}^{z,\bfr} \tau_{R}^{\la,\la'}),u_{z,\bfr}\rangle$, for any order $m$ of the symbol $a$. 
The result now follows from Lemma \ref{reduction} $(iii)$.
\end{proof}

\begin{prop}
\label{pdoadjoint}
For each $\la\in [0,1]$ and $l,m\in \R$, $\sigma_{\la}$ is a linear isomophism from $\Psi_\sigma^{l,m}$ onto $S^{l,m}_\sigma$ and $\sigma_\la(A^\dag)= (\sigma_{1-\la}(A))^*$ for any $A\in \Psi^{l,m}_\sigma$. In particular a pseudodifferential $A$ operator is formally selfadjoint (i.e $A=A^\dag$ as operators on $\S$) if and only if its Weyl symbol $\sigma_W(A)$ is selfadjoint (as a $L(E)\to T^*M$ section).
\end{prop}
\begin{proof} The fact that $\sigma_{\la}$ is a linear isomophism from $\Psi_\sigma^{l,m}$ onto $S^{l,m}_\sigma$ is a consequence Theorem \ref{lambdainv} and the fact that $\sigma_{\la}$ is a topological isomorphism from $\S'(M\times M,L(E))$ onto $\S'(T^*M,L(E))$. We check that for any $T\in \S'(T^*M,L(E))$, $\Op_{\la}(T)^\dag= \Op_{1-\la}(T^*)$ which is a direct consequence of the fact that $\Phi_{\la}(x,-\xi)=j\circ\Phi_{1-\la}(x,\xi)$ where $j(x,y)=(y,x)$. 
\end{proof}

\begin{prop}
\label{regularity}
Any operator in $\Psi_\sigma^{l,m}$ is regular. Moreover, for any $A\in \Psi_\sigma^{l,m}$ and $v\in \S$, we have 
$$
A(v)\, :x\mapsto  \int_{T_x^*(M)}d\mu_x^*(\th)\int_{T_x(M)}d\mu_x(\xi)\  e^{2\pi i \langle \th,\xi\rangle}\, \sigma_0(A)(x,\th) \,\tau_{-1}^{-1}(x,\xi)\,v(\psi_x^{-\xi})\,.
$$ 
\end{prop} 
\begin{proof} Let $A\in \Psi_{\sigma}^{l,m}$ and $a:=\sigma_0(A)$. Thus, for any frame $(z,\bfr)$, $A_{z,\bfr}=\Op_{\Ga_{0,z,\bfr}}(\mu a_{z,\bfr})$ so by Lemmas \ref{amplContinu}, \ref{lem-Phi-la} $(ii)$ and $(iii)$, $A_{z,\bfr}$ is continuous from $\S(\R^n,E_z)$ into itself. By Proposition \ref{pdoadjoint}, $A^\dag$ is a pseudodifferential operator in $\Psi_{\sigma}^{l,m}$, so we also obtain $(A^\dag)_{z,\bfr}$ continuous from $\S(\R^n,E_z)$ into itself. The result follows.
\end{proof}

\subsection{Link with standard pseudodifferential calculus on $\R^n$ and $L^2$-continuity}
\label{linkstd}

We suppose in this section that $E$ is the scalar bundle. If $A\in \Psi_{\sg}$, then $A_{z,\bfr}$ belongs to the space, noted $\Psi_{\sg,\psi}$, of regular operators $B$ on $\S(\R^n)$, of the form 
$$
B(v) (\rx) = \int_{\R^{2n}} e^{2\pi i \langle \vth,\zeta\rangle} a(\rx,\vth) v(\psi_{z}^\bfr(\rx,-\zeta)) d\zeta d\vth 
$$
where $a\in S^{\infty}_{\sg}(\R^{2n})$. We study in this section a sufficient condition on $\psi$, such that this space $\Psi_{\sg,\psi}$ is in fact equal to the usual algebra $\Psi_{\sg,std}$ pseudodifferential operators on $\R^n$ with the standard linearization $\psi(x,\zeta)= x+\zeta$. Here $\Psi_{0,std}$ corresponds to the Hormander calculus \cite{Hormander} on $\R^n$ and $\Psi_{1,std}$ is the $SG$-calculus on $\R^n$.

We will note $\psi:=\psi_z^\bfr$, $V_\rx(\zeta):=-\psi(\rx,-\zeta)+\rx$, $M_{\rx,\zeta}:= [\int_0^1 \del_{j}(V_x^{-1})^{i} (t\zeta) dt]_{i,j}$ and $N_{\rx,\zeta}:=[\int_0^1 \del_{j}V_x^{i} (t\zeta) dt]_{i,j}$. We consider the following hypothesis, noted $(H_V)$:

\noindent  $(i)$ there is $\eps,\delta,\eta>0$ such that for any $(\rx,\zeta)\in \R^{2n}$ with $\norm{\zeta}\leq \eps \langle \rx\rangle^{\sg \eta}$, we have $\det M_{\rx,\zeta} \geq \delta$ and $\det N_{\rx,\zeta}\geq \delta$,

\noindent $(ii)$ the functions $(dV_\rx)_{\rx,\zeta}$ and $(dV_\rx^{-1})_{\rx,\zeta}$ are in $E_\sg^0(\M_n(\R))$.

\begin{prop}
\label{propreducRn}
If the hypothesis $(H_V)$ holds, we have $\Psi_{\sg,\psi} = \Psi_{\sg,std}$.
\end{prop}

We set $\chi_{\eps,\eta}(\rx,\zeta):= b(\tfrac{\norm{\zeta}^2}{\eps^{2}\langle \rx\rangle^{2\sg \eta}})$ where $b\in C^{\infty}_c(\R,[0,1])$ is such that $b=0$ on $\R\backslash]-1,1[$ and $b=1$ on $[-1/4,1/4]$. 

\begin{lem}
\label{H1H2cons}
Suppose $(H_V)$.
If $a \in S^{l,m}_\sg(\R^{2n})$, then the application $$
a_{\chi,M}:(\rx,\zeta,\vth)\mapsto \chi_{\eps,\eta}(\rx,\zeta) a(\rx,\wt M_{\rx,\zeta} \vth) |J(V_\rx^{-1}|(\zeta)\,(\det M_{\rx,\zeta} )^{-1}$$
is an amplitude in $\cup_{k,w}$ $\Pi_{\sg,\ka,z}^{l,w,m}(\R^{3n})$. Similarly, 
$$a_{\chi,N}:(\rx,\zeta,\vth)\mapsto \chi_{\eps,\eta}(\rx,\zeta) a(\rx,\wt N_{\rx,\zeta} \vth) |J(V_\rx)|(\zeta)\,(\det N_{\rx,\zeta} )^{-1}$$ is in $\bigcup_{k,w}\Pi_{\sg,\ka,z}^{l,w,m}(\R^{3n})$.
\end{lem}
\begin{proof}
The result follows from Lemma \ref{HEamp} $(ii)$ and applications of Proposition \ref{inverse}.
\end{proof}

\begin{proof}[Proof of Proposition \ref{propreducRn}] Suppose that $a\in S^{l,m}_{\sg}(\R^{2n})$ and define $A$ as the operator in $\Psi_{\sg,\psi}$ with normal symbol $a$. We obtain for any $v\in \S(\R^{2n})$
$$
 A(v)(\rx):=\int_{\R^{2n}} e^{2\pi i \langle \vth,\zeta\rangle} a(\rx,\vth) v(\psi(\rx,-\zeta)) d\zeta d\vth \, .
$$
We suppose first that $a\in S^{-\infty}_{\sg}(\R^{2n})$. We have after a change of variable, and cutting the integral in two parts 
$A(v)(\rx) = A_1(v)(\rx) + A_2(v)(\rx)$ where 
\begin{align*}
&A_1(v)(\rx) = \int_{\R^{2n}}  e^{2\pi i \langle \vth,M_{\rx,\zeta}(\zeta)\rangle} \chi_{\eps,\eta}(\rx,\zeta) a(\rx,\vth) |J(V_\rx^{-1})|(\zeta) v(\rx-\zeta) d\zeta d\vth\, , \\
&A_2(v)(\rx) = \int_{\R^{2n}}  e^{2\pi i \langle \vth,V_\rx^{-1}(\zeta)\rangle} (1-\chi_{\eps,\eta})(\rx,\zeta) a(\rx,\vth) |J(V_\rx^{-1})|(\zeta) v(\rx-\zeta) d\zeta d\vth \, .
\end{align*}
In $A_1$, we permute the integrations $d\zeta$ and $d\vth$ and proceed to a change of the variable $\vth$, while in $A_2$ we integrate by parts in $\vth$ using formula (\ref{Mformula}) so that for any $p\in \N$,
\begin{align*}
&A_1(v)(\rx) = \int_{\R^{2n}}  e^{2\pi i \langle \vth,\zeta \rangle}  a_{\chi,M}(\rx,\zeta,\vth) v(\rx-\zeta) d\zeta d\vth\,  ,\\
&A_2(v)(\rx) = \int_{\R^{2n}}   e^{2\pi i \langle \vth,V_\rx^{-1}(\zeta)\rangle} (1-\chi_{\eps,\eta})(\rx,\zeta) \, ^t M_{\vth}^{p,V_\rx^{-1}(\zeta)}(a) |J(V_\rx^{-1})|(\zeta) v(\rx-\zeta) d\zeta d\vth \, .
\end{align*}
As a consequence with Lemma \ref{H1H2cons}, and with the density of $S^{-\infty}_{\sg}(\R^{2n})$ in $S^{l,m}_{\sg}(\R^{2n})$, we see that $A$ is the sum of two pseudodifferential operators in $\Psi_{\sg,std}$: $A= A_\chi + R$ where $R\in\Psi_{\sg,std}^{-\infty}$ and $A_\chi$ has $a_{\chi,M}$ as (standard) amplitude.
The implication in the other sense is similar.
\end{proof}

\begin{rem} 
In the case of pseudodifferential operator with local compact control over the $x$ variable and with $\psi$ coming from a connection, by cutting-off in the $\zeta$-variable or in other words taking $y:=\psi(x,-\zeta)$ and $x$ sufficiently close to each other, we have in fact $\Psi_{\sg,\psi}$ equal to $\Psi_{\sg,std}$ modulo smoothing elements (see \cite{Shara1}).
\end{rem}

As a consequence, we see that if the hypothesis $(H_V)$ is satisfied for a frame $(z,\bfr)$, then $\Psi_{\sg,\psi}(=\Psi_{\sg,std})$ is stable under composition of operators and the symbol composition formula is then given by a quadruple asympotic summation modulo smoothing symbols. 

We will show in the next section that we can also obtain stability under composition directly, without using a reduction to the standard calculus on $\R^n$. We shall obtain with this method a simpler symbol composition formula on $\Psi_{\sg,\psi}$, analog to the usual one on $\Psi_{\sg,std}$.

As a direct consequence of the previous proposition, we have the following $L^2$-continuity result for pseudodifferential operators on $M$.

\begin{prop} 
\label{L2cont}
If $(H_V)$ is satisfied for the function $V_{\rx}^{-1}$ in a frame $(z,\bfr)$, then any pseudodifferential operators on $M$ of order $(0,0)$ extends as a bounded operator on $L^2(M,d\mu)$.
\end{prop}
\begin{proof} Since $(H_V)$ are satisfied for $V_{\rx}^{-1}$, the proof of the previous proposition entails that $\Psi_{\sg,\psi}^{0,0} \subseteq \Psi_{\sg,std}^{0,0}$, so the result follows from the $L^2$-continuity of standard pseudodifferential operators \cite{Hormander}.
\end{proof}

\subsection{Composition of pseudodifferential operators}
\label{composec}
The goal of this section is to prove that pseudodifferential operators of $\Psi_\sigma^\infty$ are stable under composition without using the hypothesis of the previous section, and to obtain an adapated symbol composition formula. We shall adapt to our situation a technique used for Fourier integral operators in Coriasco \cite{Coriasco}, Ruzhansky and Sugimoto \cite{Ruzhansky3,Ruzhansky}.

Let us note for $(x,\xi)\in TM$ and $\xi'\in T_{\psi_x^{-\xi}}(M)$,
$\psi_{x,\xi,\xi'}:=\psi_{\psi_x^{-\xi}}^{-\xi'}$, $r_x(\xi,\xi'):= \psi_x^{-1}(\psi_{x,\xi,\xi'})$ and $q_x(\xi,\xi'):=\psi^{-1}_{\psi_{x,\xi,\xi'}}(\psi_x^{-\xi})$.
 We define $V_x$ the $2n$ dimensional smooth manifold as $V_x:=\set{(\xi,\xi')\in T_x(M)\times \cup_{y\in M}T_y(M) \ | \ \xi'\in T_{\psi_x^{-\xi}}(M)}$. Each $V_x$ manifold is diffeomorphic to $\R^{2n}$ via the map, defined for any fixed frame $(z,\bfr)$, $n_{z,V_x}^\bfr(\xi,\xi'):=( M_{z,x}^\bfr(\xi), M_{z,\psi_x^{-\xi}}^{\bfr}(\xi'))$, and has a canonical involutive diffeomorphism $R_x$ defined as 
$$
R_x : (\xi,\xi')\mapsto (r_{x}(\xi,\xi'), q_x(\xi,\xi')) \, .
$$ 
In all the following we fix a frame $(z,\bfr)$, and note also $\psi$ the function $m_{-1}^{z,\bfr}$. We note $\rx^{\zeta,\zeta'}:=\psi(\psi(\rx,\zeta),\zeta')$. For each $\rx\in \R^n$, $R_\rx:=n_{z,V_{(n_z^\bfr)^{-1}(\rx)}}^{\bfr}\circ R_{(n_z^\bfr)^{-1}(\rx)} \circ (n_{z,V_{(n_z^\bfr)^{-1}(\rx)}}^{\bfr})^{-1}$ is a diffeomorphism on $\R^{2n}$, and we define $R_{\rx}=:(r_\rx,q_\rx)$, $r=r^{z,\bfr}:=(\rx,\zeta,\zeta')\mapsto r_\rx(\zeta,\zeta')$ and $q=q^{z,\bfr}:=(\rx,\zeta,\zeta')\mapsto q_{\rx}(\zeta,\zeta')$.
Remark that $r_{\rx}(\zeta,\zeta')=-\ol{\psi_z^\bfr}(\rx,\rx^{\zeta,\zeta'})=:\ol\psi_\rx \circ \psi_{\psi_\rx(\zeta)}(\zeta')$ and $q_{\rx}(\zeta,\zeta')=-P^{z,\bfr}_{-1,\psi(\rx,\zeta),\zeta'} (\zeta')$. The map $r_{\rx,\zeta}:\zeta'\mapsto r_{\rx}(\zeta,\zeta')$ is a diffeomorphism on $\R^n$ such that $r_{\rx,\zeta}^{-1}=r_{\psi_\rx(\zeta),\ol\psi_{\psi(\rx,\zeta)}(\rx)}$ so that $(dr_{\rx,\zeta})_{\zeta'}^{-1}=(dr_{{\psi_\rx(\zeta),\ol\psi_{\psi(\rx,\zeta)}(\rx)}})_{r_{\rx,\zeta}(\zeta')}$. We will use the shorthand $\tau:=(\tau_{-1}^{z,\bfr})^{-1}$.

We note $s(\rx,\zeta,\zeta'):=r(\rx,\zeta,\zeta')-\zeta$. We have $s(\rx,\zeta,\zeta')=s_{\rx,\zeta}(\zeta')$ where $s_{\rx,\zeta}=T_{-\zeta}\circ \ol\psi_\rx \circ \psi_{\psi_{\rx}(\zeta)}$ is a diffeomorphism on $\R^n$ such that $s_{\rx,\zeta}(0)=0$.
We also define 
$$
\varphi_{\rx,\zeta}(\zeta'):=r_{\rx,\zeta}(\zeta')-\zeta-(dr_{\rx,\zeta})_0(\zeta')
$$
so that $\varphi_{\rx,\zeta}(0)=0$ and $(d\varphi_{\rx,\zeta})_0=0$, and 
$$
V(\rx,\zeta,\zeta'):= (dr_{\rx,\zeta})_{\zeta'}
$$
as a smooth function from $\R^{3n}$ into $\M_n(\R)$. We shall note $(\rx,\zeta) \mapsto L_{\rx,\zeta}:=-\,^t(dr_{\rx,\zeta})_0$.

We define $ \O_{\sigma,\kappa,\eps_0,\eps_1,c}^{l,w_0,w_1}(\mathfrak{E})$, where $c\in \N$, $l\in \R$, $w:=(w_0,w_1)\in\R^2_+$, $\eps:=(\eps_0,\eps_1)$, $\eps_0\geq 0$, $\eps_1>0$, $\sigma\in [0,1]$ and $\kappa\geq 0$, as the space of smooth functions $g$ from $\R^{3n}$ into $\mathfrak{E}$ such that for any $3n$-multi-index $\nu=(\mu,\ga)\in \N^{2n}\times \N^n$, there exists $C_\nu>0$
such that for any $(\rx,\zeta,\zeta')\in \R^{3n}$, $\norm{\del^{\nu} g (\rx,\zeta,\zeta')}\leq C_{\nu} \langle \rx\rangle^{\sigma(l- |\mu|-\eps_1 |\ga|_c)}\langle \zeta\rangle^{w_0+\ka|\mu|+\eps_0|\ga|} \langle\zeta' \rangle^{w_1+\kappa|\nu|}$. Here, we noted $|\ga|_c:= 0$ if $|\ga|<c$ and $|\ga|_c:=|\ga|-c$ if $|\ga|\geq c$. We note $ \O_{\sg,\ka,\eps}(\mathfrak{E}):=\cup_{c,l,w} \O_{\sigma,\kappa,\eps,c}^{l,w}(\mathfrak{E})$. We check that for any multi-indices $\ga,\ga'$ and $c,c'\in \N$, $|\ga|_c+|\ga'|_c\geq |\ga+\ga'|_{c+c'}$, and $|\ga+\ga'|_c\geq |\ga|_c+|\ga'|_c$.
Thus, $ \O_{\sg,\ka,\eps}(\R)$, $\O_{\sg,\ka,\eps}(\mathcal{M}_{p}(\R))$ and $\O_{\sg,\ka,\eps,z}:=\O_{\sg,\ka,\eps}(L(E_z))$ are algebras (graduated by the parameters $c$, $l$, $w_0$ and $w_1$) and $\del^\nu \O_{\sg,\ka,\eps,c}^{l,w}(\mathfrak{E})\subseteq \O_{\sg,\ka,\eps,c}^{l-|\mu|-\eps_1|\ga|_c,w_0+\ka|\mu|+\eps_0|\ga|,w_1+\ka|\nu|}(\mathfrak{E})$.
If $f\in \O_{\sg,\ka,\eps,c}^{0,w}(\mathfrak{E})$, then $(\rx,\zeta)\mapsto f(\rx,\zeta,0)\in E^{w_0}_{\sg,\ka}(\mathfrak{E})$, and if $f\in \O_{\sg,\ka,\eps,c,z}^{l,w}$, then $(\rx,\zeta,\vth)\mapsto f(\rx,\zeta,0) \in \Pi_{\sg,\ka,z}^{l,w_0,0}$. Remark that any monomial of the form $(\rx,\zeta,\zeta')\mapsto \zeta'^\b$ where $\b \in \N^n$, is in $\O_{\sg,\ka,\eps,|\b|}^{0,0,|\b|}(\R)$ for any $\ka\geq 0$ and $\eps_0\geq 0$, $\eps_1>0$.

In the definition of $S'_\sg$ bounded geometry, we only require a polynomial control over the $\ol\psi_{z}^\bfr$ functions. It appears that for the theorem of composition, a stronger control over these functions is important. We thus introduce the following:

\begin{defn}
\label{Csigma}
We shall say that $(C_\sg)$ is satisfied if there is a frame $(z,\bfr)$, $(\ka_v, w_v)\in \R^2_+$ with $\ka_v\geq 1$, and $\eps_v\in ]0,1[$, such that 
\begin{align}
\label{CHyp1}
V \in \O_{\sg,\ka_v,\eps_v,\eps_v,0}^{0,0,w_v}(\M_n(\R))\,,  \quad\, \text{and}\quad (d\psi_{z,\rx}^\bfr)_\zeta ,\,(d\ol\psi_{z,\rx}^\bfr)_\ry = \O(1)\, . 
\end{align}
\end{defn}
\noindent In particular $(C_\sg)$ entails that $(dr_{\rx,\zeta})_{0}$ and thus $L$ are in $E_{\sg,\ka_v}^0(\M_n(\R))$. 

We note $\RR_{\sg,\ka,\eps_1}^{w_0,w_1}(\mathfrak{E})$ ($\eps_1>0$) as the space of smooth functions $g$ such that for any nonzero $\nu=(\mu,\ga)\in \N^{2n}\times \N^n$, $\del^\nu g  = \O(\langle \rx\rangle^{\sg(1-|\mu|-\eps_1|\ga|)}\langle \zeta\rangle^{w_0 + \ka(|\nu|-1)} \langle \zeta'\rangle^{w_1+\ka(|\nu|-1)})$. It follows from $(C_\sg)$ that $r\in \cup_{w_0,w_1}\RR_{\sg,\ka_v,\eps_v/2}^{w_0,w_1}(\R^n)$.

The following lemma will give us the link between the the $\O$, $\RR$, $H$, $E$ spaces and the behaviour under composition.

\begin{lem}
\label{htilde}

$(i)$ Let $f\in H_{\sg,\ka}^w(\mathfrak{E})$ (resp. $E_{\sg,\ka}^w(\mathfrak{E})$) and $g\in  \RR_{\sg,\ka,\eps_1}^{w_0,w_1}(\R^{2n})$ such that $g_2(\rx,\zeta,\zeta')=\O(\langle\zeta \rangle^{k_2}\langle \zeta'\rangle^{k'_2})$ for a $(k_2,k'_2)\in \R^2_+$ and, if $\sg\neq 0$, $\langle g_1(\rx,\zeta,\zeta')\rangle\geq c\langle \rx\rangle \langle \zeta\rangle^{-k_1}\langle \zeta'\rangle^{-k_1'}$, for a $(k_1,k'_1)\in \R^2_+$ and $c>0$. Then, $f\circ g\in \RR_{\sg,\ka_H,\eps_1}^{w_0+k_2w,w_1+k'_2 w}(\mathfrak{E})$ (resp. $\O_{\sg,\ka_E,\ka_E,\eps_1,0}^{0,k_2 w,k'_2 w}(\mathfrak{E})$) where $\ka_H:=\ka+\max\set{|w_0+k_1\sg +k_2\ka|,|w_1+k'_1\sg +k'_2\ka|}$ and $\ka_E:= \ka+\max\set{|w_0+k_1\sg+(k_2-1)\ka|,|w_1+k'_1\sg+(k'_2-1)\ka|}$.

\noindent $(ii)$ $(\rx,\zeta,\zeta')\mapsto (\psi(\rx,\zeta),\zeta') \in \RR_{\sg,\ka_\psi,1}^{w_\psi,0}(\R^{2n})$ and $(\rx,\zeta,\zeta')\mapsto \rx^{\zeta,\zeta'} \in \RR_{\sg,\ka_\psi,1}$ for a $(\ka_\psi,w_\psi)\in \R^2_+$.

\noindent $(iii)$ The functions $q$, $(\rx,\zeta,\zeta')\mapsto (P^{z,\bfr}_{-1,\psi(\rx,\zeta),\zeta'})^{-1}$ and $(\rx,\zeta,\zeta')\mapsto \det (P^{z,\bfr}_{-1,\psi(\rx,\zeta),\zeta'})^{-1}$ are respectively in $\RR_{\sg,\ka_q,1}(\R^n)$, $\O^{0,0,0}_{\sg,\ka_q,\ka_q,1,0}(\M_n(\R))$, and $\O^{0,0,0}_{\sg,\ka_q,\ka_q,1,0}(\R)$, for a $\ka_q\geq 0$.
Moreover, there exists $C>0$ such that for any $(\rx,\zeta,\zeta')\in \R^{3n}$, $\norm{q_\rx(\zeta,\zeta')}\leq C\langle \zeta' \rangle$.

\noindent $(iv)$ $(\rx,\zeta,\zeta')\mapsto \tau(\rx^{\zeta,\zeta'},q_\rx(\zeta,\zeta'))$ is in $\O^{0,0,0}_{\sg,\ka_\tau,\ka_\tau,1,0,z}$ for a $\ka_\tau\geq 0$.
\end{lem}
\begin{proof}
$(i)$ If $\nu=(\a,\b,\ga)\neq 0$ is a $3n$-multi-index, we have $\del^\nu f\circ g = \sum_{1\leq |\nu'|\leq |\nu|} P_{\nu,\nu'}(g) (\del^{\nu'}f)\circ g$, with $P_{\nu,\nu'}(g)$ a linear combination of terms of the form $\prod_{j=1}^s (\del^{l^j}g)^{k^j}$, with $1\leq s \leq |\nu|$, $\sum_{1}^s l^j |k^j|=\nu$, $\sum_1^s k^j=\nu'$.
As a consequence, we get the following estimate for any $1\leq |\nu|\leq |\nu'|$, $P_{\nu,\nu'}(g)=\O(\langle \rx\rangle^{\sg(|\nu'|-|\mu|-\eps_1|\ga|)} \langle \zeta\rangle^{w_0|\nu'|+\ka(|\nu|-|\nu'|)}\langle \zeta'\rangle^{w_1|\nu'|+\ka(|\nu|-|\nu'|)})$.
Moreover, for any $1\leq |\nu'|\leq |\nu|$, there is $C_\nu>0$ such that for any $(\rx,\zeta,\zeta')\in \R^{3n}$, the following estimate is valid $\norm{(\del^{\nu'}f) \circ g(\rx,\zeta,\zeta')}\leq C_\nu \langle\rx\rangle^{-\sg(|\nu'|-1)} \langle \zeta\rangle^{(k_1\sg+k_2\ka)(|\nu'|-1)+k_2w}\langle \zeta'\rangle^{(k_1'\sg+k'_2\ka)(|\nu'|-1)+k'_2w}$ (resp. $\norm{(\del^{\nu'}f) \circ g(\rx,\zeta,\zeta')}\leq C_\nu \langle\rx\rangle^{-\sg|\nu'|} \langle \zeta\rangle^{(k_1\sg+k_2\ka)|\nu'|+k_2w}\langle \zeta'\rangle^{(k_1'\sg+k'_2\ka)|\nu'|+k'_2w}$). The result follows.

\noindent $(ii)$ By hypothesis, $\psi\in H^{w_\psi}_{\sg,\ka_\psi}$. We deduce that $(\rx,\zeta,\zeta')\mapsto \psi(\rx,\zeta)\in \RR^{w_\psi,0}_{\sg,\ka_\psi,1}$ and the first statement now follows from $(\rx,\zeta,\zeta')\mapsto \zeta'\in \RR^{0,0}_{\sg,\ka_\psi,1}$. The second statement follows from $(i)$.

\noindent $(iii)$ Since $q_{\rx}(\zeta,\zeta')=-P^{z,\bfr}_{-1,\psi(\rx,\zeta),\zeta'}(\zeta')$, the fact that $q_\rx\in  \RR_{\sg,\ka_q,1}(\R^n)$ for a $\ka_q\geq 0$  is a consequence of $(i)$, $(ii)$ and Lemma \ref{lemGsigma} $(iii)$. We also have by $(i)$ and $(ii)$, $(P^{z,\bfr}_{-1,\psi(\rx,\zeta),\zeta'})^{-1} \in  \O^{0,0,0}_{\sg,\ka_q,\ka_q,1,0}(\M_n(\R))$.

\noindent $(iv)$ Since $\tau\in E^0_{\sg,\ka}(L(E_z))$ for a $\ka\geq 0$,
the result follows $(i)$, $(ii)$, $(iii)$ and the estimate $\langle \rx^{\zeta,\zeta'}\rangle \geq c \langle \rx\rangle \langle \zeta\rangle^{-k} \langle \zeta' \rangle^{-k}$ for $c,k>0$.
\end{proof}

\begin{lem}
\label{Csgr}
Suppose $(C_\sg)$. Then 

\noindent (i) $s,\varphi \in \O_{\sg,\ka_v,\eps_v,\eps_v,1}^{0,0,w_s}(\R^n)$ and $\varphi\in \O_{\sg,\ka_v,\eps_v,\eps_v,2}^{-\eps_v,\eps_v,w_\varphi}(\R^n)$ where $w_s:=w_v+1$ and $w_\varphi:=2+w_v+\ka_v$.

\noindent (ii) $V=(dr_{\rx,\zeta})_{\zeta'}$  and  $(dr_{\rx,\zeta})^{-1}_{\zeta'}$ are bounded on $\R^{3n}$.

\noindent (iii) The function $J(R):(\rx,\zeta,\zeta')\mapsto J(R_\rx)(\zeta,\zeta')$ is in $\cup_{\ka,w_0,w_1,\eps_0,\eps_1}\O_{\sg,\ka,\eps_0,\eps_1,0}^{0,w_0,w_1}(\R)$ and $(\rx,\zeta,\zeta')\mapsto \tau (\rx,r_\rx(\zeta,\zeta'))$ is in $\O_{\sg,\ka_\tau,\ka_\tau,\eps_v/2,0,z}^{0,0,0}$ for $\ka_\tau\geq 0$.
\end{lem}

\begin{proof}
\noindent $(i)$ We have $s_{\rx,\zeta}(\zeta')=\sum_{i=1}^n\zeta'_i \int_0^1 \del_{\zeta'_i} r_{\rx,\zeta}(t\zeta')\,dt$. Since $V \in \O_{\sg,\ka_v,\eps_v,0}^{0,0,w_v}(\M_n(\R))$ each function $(\rx,\zeta,\zeta')\mapsto\int_0^1 \del_{\zeta'_i} r_{\rx,\zeta}(t\zeta')\,dt$ is in $\O_{\sg,\ka_v,\eps_v,0}^{0,0,w_v}(\R^n)$ and thus, since $(\rx,\zeta,\zeta')\mapsto \zeta'_i \in \O_{\sg,\ka_v,\eps_v,1}^{0,0,1}(\R)$, we see that $s \in \O_{\sg,\ka_v,\eps_v,1}^{0,0,w_s}(\R^n)$. We have also $\varphi_{\rx,\zeta}(\zeta')=\sum_{|\b|=2} \tfrac{2}{\b!} (\zeta')^\b \int_{0}^1(1-t)\, \del_{\zeta'}^\b r_{\rx,\zeta}(t\zeta')\,dt$ and each function $(\rx,\zeta,\zeta')\mapsto\int_0^1 (1-t)\, \del_{\zeta'}^\b r_{\rx,\zeta}(t\zeta')\,dt$ is in $\O_{\sg,\ka_v,\eps_v,0}^{-\eps_v,\eps_v,w_v+\ka_v}(\R^n)$. With $(\rx,\zeta,\zeta')\mapsto (\zeta')^\b \in \O_{\sg,\ka_v,\eps_v,2}^{0,0,2}(\R)$, we get $\varphi \in \O_{\sg,\ka_v,\eps_v,2}^{-\eps_v,\eps_v,w_\varphi}(\R^n)$.

\noindent $(ii)$ Direct consequence of $(C_\sg)$ and the following equalities for any $(\rx,\zeta,\zeta')\in \R^{3n}$, $(dr_{\rx,\zeta})_{\zeta'} =(d\ol\psi_{\rx})_{\rx^{\zeta,\zeta'}}(d\psi_{\psi_\rx(\zeta)})_{\zeta'}$ and $(dr_{\rx,\zeta})^{-1}_{\zeta'}=(d\ol\psi_{\psi_\rx(\zeta)})_{\rx^{\zeta,\zeta'}}(d\psi_\rx)_{r_{\rx,\zeta}(\zeta')}$. 

\noindent $(iii)$ The first statement follows from Lemma \ref{htilde} $(ii)$. The second statement follows from Lemma \ref{htilde} $(i)$ and the estimate $r_{\rx}(\zeta,\zeta')=\O(\langle \zeta \rangle \langle \zeta'\rangle^{w_v})$.
\end{proof}

We shall use a generalization to four variables of the $\Pi_{\sg,\ka,z}^{l,w,m}$ spaces of amplitude. We define $\wt \Pi_{\sg,\ka,\eps_1,z}^{l,w_0,w_1,m}$ ($0<\eps_1\leq 1$) as the space of smooth functions $a\in C^{\infty}(\R^{4n},L(E_z))$ such that for any $4n$-multi-index $(\nu,\delta)\in \N^{3n}\times \N^n$, (with $\nu=(\mu,\ga)\in \N^{2n}\times \N^n$) there is $C_{\nu,\delta}>0$  such that for any $(\rx,\zeta,\zeta',\vth)\in \R^{4n}$, 
$$
\norm{\del^{(\nu,\delta)} a (\rx,\zeta,\zeta',\vth)}_{L(E_z)}\leq C_{\nu,\delta} \langle \rx\rangle^{\sg(l-|\mu|-\eps_1|\ga|)} \langle \zeta \rangle^{w_0+\ka|\nu|} \langle \zeta'\rangle^{w_1+\ka|\nu|} \langle \vth\rangle^{m-|\delta|}\, .
$$ 
These spaces have natural Fr\'{e}chet topologies and form a graded topological algebra under pointwise composition.

\begin{lem}
\label{amptilde}
$(i)$ If $a \in \wt \Pi_{\sg,\ka,\eps_1,z}^{l,w_0,w_1,m}$, then $a_{\zeta'=0}:(\rx,\zeta,\vth)\mapsto a(\rx,\zeta,0,\vth)$ is in $\Pi_{\sg,\ka,z}^{l,w_0,m}$.
 
\noindent $(ii)$ If $h\in  \O_{\sg,\ka,\eps_0,\eps_1,0,z}^{l,w_0,w_1}$, then $(\rx,\zeta,\zeta',\vth)\mapsto h(\rx,\zeta,\zeta')$ is in $\wt\Pi_{\sg,\max\set{\ka,\eps_0},\eps_1,z}^{l,w_0,w_1,0}$.

\noindent $(iii)$ There is $\ka_\Xi,k_1 \geq 0$ such that for any $b\in S^{l,m}_{\sg,z}$, the application $b\circ \wt \Xi$, where $\wt \Xi(\rx,\zeta,\zeta',\vth):=(\rx^{\zeta,\zeta'},-\wt P_{-1,\psi(\rx,\zeta),\zeta'}^{z,\bfr}(\vth))$, is in $\wt\Pi_{\sg,\ka_{\Xi},1,z}^{l,\sg k_1|l|,\sg k_1|l|,m}$.

\end{lem}
\begin{proof}
$(i)$ and $(ii)$ are direct.

\noindent $(iii)$ If $\mu=(\nu,\delta)\neq 0$ is a $4n$-multi-index, we have $\del^{\mu}(b\circ \wt \Xi)=\sum_{1\leq |\mu'|\leq |\mu| }P_{\mu,\mu'}(\wt \Xi) \, (\del^{\mu'} b )\circ \wt \Xi$ with $P_{\mu,\mu'}(\wt \Xi)$ a linear combination of terms of the form $\prod_{j=1}^s (\del^{l^j}\wt \Xi)^{k^j}$, with $1\leq s \leq |\mu|$, $l^j=(l^{j,1},l^{j,2})\in \N^{3n}\times \N^n$, $k^{j}=(k^{j,1},k^{j,2})\in \N^{n}\times \N^n$, such that $l^{j,2}=0$ for $1\leq j\leq j_1 \leq s$, and $\sum_{1}^s l^j |k^j|=\mu$, $\sum_1^s k^j=\mu'$. We have
$$
(\del^{l^j}\wt\Xi)^{k^j}= \prod_{i=1}^n (\delta_{l^{j,2},0}(\del^{l^{j,1}} \rx^{\zeta,\zeta'})_i )^{k^{j,1}_i}\ \prod_{i=1}^n \big(\sum_{k=1}^n \del^{l^{j,1}}P^{i,k}\ \del^{l^{j,2}} \vth_{k}\big)^{k^{j,2}_i}
$$
where $P^{i,k}$ are the matrix entries of $-\wt P_{-1,\psi(\rx,\zeta),\zeta'}^{z,\bfr}$. By Lemma \ref{htilde} $(ii)$ and $(iii)$, $\rx^{\zeta,\zeta'} \in \RR_{\sg,\ka_\psi,1}^{w_0,w_1}(\R^n)$ and the $P^{i,k}$ are in $\O^{0,0,0}_{\sg,\ka_\psi,\ka_\psi,1,0}(\R)$ for a $(\ka_\psi,w_0,w_1)\in \R^3_+$. We obtain thus the following estimate 
$$
|P_{\mu,\mu'}(\wt \Xi)(\rx,\zeta,\zeta',\vth) |\leq C_\mu \langle \rx\rangle ^{-\sg(|\nu|-|\a'|)} \langle \zeta \rangle^{w_0|\a'|+\ka_\psi(|\nu|-|\a'|)}\langle \zeta'\rangle^{w_1|\a'|+\ka_\psi(|\nu|-|\a'|)} \langle \vth\rangle^{|\b'|-|\delta|}
$$
with $\mu'=:(\a',\b')$. Since $b\in S^{l,m}_{\sg,z}$ we also have the estimate 
$$
\norm{(\del^{\mu'} b) \circ \wt \Xi (\rx,\zeta,\zeta')} \leq C'_{\mu} \langle \rx^{\zeta,\zeta'} \rangle^{\sg(l-|\a'|)} \langle \vth \rangle^{m-|\b'|}  
$$
so the result follows now from the estimate $\langle \rx^{\zeta,\zeta'} \rangle^{\sg(l-|\a'|)} =\O( \langle \rx\rangle^{\sg(l-|\a'|)} (\langle \zeta\rangle \langle \zeta'\rangle)^{\sg k_1 |l|+\sg k_1 |\a'|})$, with $\ka_{\Xi}:=\ka_{\psi}+\max\set{|w_0+\sg k_1-\ka_\psi|,|w_1+\sg k_1 -\ka_\psi|}$.
\end{proof}

\begin{lem}
\label{derivphase}
Let $s\in C^\infty(\R^p,\R^n)$. Then for any $p+n$-multi-index $\nu=(\a,\b)\neq 0$, we have 
$$
\del_{\rx,\vth}^\nu\, e^{i\langle \vth,s(\rx)\rangle } = P_{\nu}(\rx,\vth)\, e^{i\langle \vth,s(\rx)\rangle }
$$
where $P_{\nu}$ is of the form $\sum_{|\ga|\leq |\a|} \vth^\ga \, T_{\nu,\ga}(\rx)$, and $T_{\nu,\ga}$ is a linear combination of terms of the form $\prod_{j=1}^m (\del^{l^j} s)^{\mu^j}$ where $1\leq m\leq |\nu|$, $(l^j)$ are $p$-multi-indices  and $(\mu^j)$ are $n$-multi-indices. Moreover, they satisfy $|\mu^j|>0$, $\sum_{j=1}^m |\mu^j|=|\ga|+|\b|$, $\sum_{j=1}^{m} |\mu^j||l^j|=|\a|$ and if $|\b|=0$, then $|l^j|>0$ and $|\ga|>0$.
\end{lem}
\begin{proof}
We note $g(\rx,\vth):=\langle \vth,s(\rx)\rangle$. By Theorem \ref{FaaCS}, we get the following equality for any $\nu\neq 0$, $\del_{\rx,\vth}^\nu\, e^{i\langle \vth,s(\rx)\rangle } = P_{\nu}(\rx,\vth) e^{i\langle \vth,s(\rx)\rangle}$ where $P_{\nu}(\rx,\vth)=\sum_{1\leq k\leq |\nu|} P_{\nu,k}(g)$ and $P_{\nu,k}$ is a linear combination of terms of the form $\prod_{j=1}^{m} (\del^{l^j} g)^{k^j}$ such that $|l^{j}|>0$, $k^j>0$, $\sum_{1}^m k^{j}= k$ and $\sum_1^m k^j l^j = \nu$. If we suppose that the term $\prod_{j=1}^{m} (\del^{l^j} g)^{k^j}$ is non-zero, then $|l^j|\leq 1$ and if we define $j_1$ such that for any $1\leq j \leq j_1$, $l^{j,2}=0$, we obtain, noting $l^{j}=(l^{j,1},l^{j,2})$,
\begin{align*}
\prod_{j=1}^{m} (\del^{l^j} g)^{k^j}&= \prod_{j=1}^{j_1} \langle \vth, \del^{l^{j,1}} s\rangle^{k^j} \prod_{j=j_1+1}^m (\del^{l^{j,1}} s^{q_j})^{k^j} \\
&= \sum_{|\ga^j|=k^j,\,1\leq j\leq j_1} \ga^1!\cdots \ga^j!\ \vth^{\sum_1^{j_1} \ga^j} \prod_{j=1}^{j_1} (\del^{l^{j,1}}s)^{\ga^j} \prod_{j=j_1+1}^m (\del^{l^{j,1}} s^{q_j})^{k^j} \, .
\end{align*}
Thus, we have $P_{\nu,k}=\sum_{|\ga|=k-|\b|} \vth^\ga \, T_{\nu,\ga,k}(\rx)$ where $T_{\nu,\ga,k}$ is a linear combination of terms of the form $\prod_{j=1}^{j_1} (\del^{l^{j,1}} s)^{\mu^{j}} \prod_{j=j_1+1}^m (\del^{l^{j,1} } s^{q_j})^{k^j}$, where $1\leq q_j\leq n$, $1\leq j\leq m \leq |\nu|$, $1\leq j_1\leq m$, $l^{j,1}\in \N^p$, $k^j\in \N^*$, $\la^j\in \N^n$ are such that $\sum_{1}^m k^j=k$, $\sum_{1}^{j_1} |\la^j||l^{j,1}| + \sum_{j_1
+1}^{m} k^j |l^{j,1}+1|=|\nu|$ and $\sum_{j_1+1}^m k^j = |\b|$. The result follows.
\end{proof}

\begin{lem}
\label{phasecontrol}
Suppose that $(C_\sg)$ is satisfied. Then

\noindent (i) Representing by $\ru$ the letter $s$ or $\varphi$, for any $3n$-multi-index $\nu=(\mu,\ga)\in \N^{2n}\times \N^n$, we have the equality $\del^\nu_{\rx,\zeta,\vth} e^{2\pi  i \langle \vth,\ru_{\rx,\zeta}(\zeta')\rangle} = (\sum_{|\om|\leq |\mu|} \vth^\om T_{\nu,\om,\ru}(\rx,\zeta,\zeta'))\, e^{2\pi  i \langle \vth,\ru_{\rx,\zeta}(\zeta')\rangle}$ where each term $T_{\nu,\om,s}\in \O_{\sg,\ka_v,\eps_v,\eps_v,|\om+\ga|}^{-|\mu|,\ka_v|\mu|,w_s|\om+\ga|+\ka_v|\mu|}(\R)$ and $T_{\nu,\om,\varphi}\in \O_{\sg,\ka_v,\eps_v,\eps_v,2|\om+\ga|}^{-|\mu|-\eps_v|\om+\ga|,\eps_v|\om+\ga|+\ka_v|\mu|,w_\varphi|\om+\ga|+\ka_v|\mu|}(\R)$. In particular, it satisfies the following estimate valid for any $(\rx,\zeta,\zeta')\in \R^{3n}$, and any $n$-multi-index $\rho$,
\begin{align*}
&|\del^\rho_{\zeta'} T_{\nu,\om,s}(\rx,\zeta,\zeta') |\leq C_{\nu,\om,\rho}\langle \rx\rangle^{-\sg(|\mu|+\eps_v |\rho|_{|\om+\ga|})}\langle \zeta \rangle^{\ka_v|\mu|+\eps_v |\rho|} \langle \zeta'\rangle^{w_s|\om+\ga|+\ka_v(|\mu|+|\rho|)}\, , \\
& |\del^\rho_{\zeta'} T_{\nu,\om,\varphi }(\rx,\zeta,\zeta') |\leq C_{\nu,\om,\rho}\langle \rx\rangle^{-\sg(|\mu|+(\eps_v/2)|\rho|)}\langle \zeta \rangle^{\eps_v|\om+\ga|+\ka_v|\mu|+\eps_v|\rho|} \langle \zeta'\rangle^{w_\varphi|\om+\ga|+\ka_v(|\mu|+|\rho|)}\, .
\end{align*}

\noindent (ii) For any $n$-multi-index $\b$, we have $\del^{\b}_{\zeta'} e^{2\pi i \langle \vth, \varphi_{\rx,\zeta}(\zeta')\rangle} = P_{\b,\varphi}(\rx,\zeta,\zeta',\vth) e^{2\pi i \langle \vth, \varphi_{\rx,\zeta}(\zeta')\rangle}$ where $P_{\b,\varphi}(\rx,\zeta,\zeta',\vth)$ is a linear combination of terms of the form $\vth^\om \zeta'^\la t_{\om,\la}(\rx,\zeta,\zeta')$ where $\om$ and $\la$ are $n$-multi-indices satifying $|\om|\leq |\b|$, $(2|\om|-|\b|)_+\leq |\la|\leq |\om|$, and $t_{\om,\la}$ are functions in $\O_{\sg,\ka_v,\eps_v,\eps_v,|\b|}^{-\eps_v|\b|/2,2\eps_v,w'_s|\b|}(\R)$. In particular they are estimated by 
$$
t_{\om,\la}(\rx,\zeta,\zeta') = \O(\langle \rx\rangle^{-\sg\eps_v|\b|/2} \langle \zeta\rangle^{2\eps_v|\b|} \langle \zeta'\rangle^{w'_s|\b|})
$$
where  $w'_s:=w_s+2\ka_v$.
Moreover, $(\rx,\zeta,\vth)\mapsto P_{\b,\varphi}(\rx,\zeta,0,\vth)\, 1_{L(E_z)}\in  \Pi_{\sg,\ka_v,z}^{-\eps_v|\b|/2,\eps_v|\b|,|\b|/2}$.

\noindent (iii) If $\b\in \N^n$ and $f\in \wt \Pi_{\sg,\ka,\eps_1,z}^{l,w_0,w_1,m}$ then the function 
$$
f_{\b,\varphi}: (\rx,\zeta,\vth)\mapsto \del^{\b}_{\zeta'}\big(e^{2\pi i \langle \vth,\varphi_{\rx,\zeta}(\zeta')\rangle} \del^{0,0,0,\b}f(\rx,\zeta,\zeta',L_{\rx,\zeta}(\vth))\big)_{\zeta'=0}$$
belongs to $\Pi_{\sg,\ka_1,z}^{l-\eps'_1 |\b|,w_0+\ka_2|\b|,m-|\b|/2}$, where $\eps'_1:=\min\set {\eps_1/2,\eps_v/2}>0$, $\ka_1:=\max\set{\ka_v,\ka}$, $\ka_2:=\ka+|\eps_v-\ka|$, and the 
application $f\mapsto f_{\b,\varphi}$ is continuous.

\end{lem}

\begin{proof} $(i)$ By Lemma \ref{derivphase}, if $\nu\neq 0$, we have the following equality, valid for any $(\rx,\zeta,\zeta',\vth)\in {\R^{4n}}$, $\del^\nu_{\rx,\zeta,\vth} e^{2\pi  i \langle \vth,\ru_{\rx,\zeta}(\zeta')\rangle} = (\sum_{|\om|\leq |\mu|} \vth^\om T_{\nu,\om,\ru}(\rx,\zeta,\zeta'))\, e^{2\pi  i \langle \vth,\ru_{\rx,\zeta}(\zeta')\rangle}$ where $T_{\nu,\om,\ru}$
is a linear combination of terms of the form $\prod_{j=1}^m( \del_{\rx,\zeta}^{l^j} \ru )^{\mu^j}$ with $1\leq m \leq |\nu|$, $\mu^j\neq 0$, $\sum_{j=1}^m |\mu^j|= |\om+\ga|$ and $\sum_{j=1}^{m} |\mu^j||l^j|=|\mu|$. Since by Lemma \ref{Csgr} $(i)$, $s\in \O_{\sg,\ka_v,\eps_v,\eps_v,1}^{0,0,w_s}(\R^n)$, it is straightforward to check that $T_{\nu,\om,s}\in \O_{\sg,\ka_v,\eps_v,\eps_v,|\om+\ga|}^{-|\mu|,\ka_v|\mu|,w_s|\om+\ga|+\ka_v|\mu|}(\R)$. Moreover, since $\varphi\in \O_{\sg,\ka_v,\eps_v,\eps_v,2}^{-\eps_v,\eps_v,w_\varphi}(\R^n)$, we get $T_{\nu,\om,\varphi}\in \O_{\sg,\ka_v,\eps_v,\eps_v,2|\om+\ga|}^{-|\mu|-\eps_v|\om+\ga|,\eps_v|\om+\ga|+\ka_v|\mu|,w_\varphi|\om+\ga|+\ka_v|\mu|}(\R)$. The first estimate is direct and the second estimate follows from the inequality $|\om+\ga|+|\rho|_{2|\om+\ga|} \geq |\rho|/2$.

\noindent $(ii)$  By Lemma \ref{derivphase}, if $\b\neq 0$, we have for any $(\rx,\zeta,\zeta',\vth)\in {\R^{4n}}$, the following relation $\del^\b_{\zeta'} e^{2\pi  i \langle \vth,\varphi_{\rx,\zeta}(\zeta')\rangle} = (\sum_{1\leq|\om|\leq |\b|} \vth^\om T_{\b,\om,\varphi}(\rx,\zeta,\zeta')) e^{2\pi  i \langle \vth,\varphi_{\rx,\zeta}(\zeta')\rangle}$ where $T_{\b,\om,\varphi}$
is a linear combination of terms of the form $\prod_{j=1}^m( \del^{l^j} \varphi_{\rx,\zeta} )^{\mu^j}$ with $1\leq m \leq |\b|$, $\mu^j\neq 0$, $l^j\neq 0$, $\sum_{j=1}^m |\mu^j|= |\om|$ and $\sum_{j=1}^{m} |\mu^j||l^j|=|\b|$. Let us reorder the $l^j$ indices so that for any $1\leq j\leq j_1$, $|l^j|=1$ and for any  $j\geq j_1+1$, $|l^j|>1$, where $j_1 \in \set{0,\cdots m}$. Thus $\prod_{j=1}^m( \del^{l^j} \varphi_{\rx,\zeta} )^{\mu^j} = \prod_{j=1}^{j_1}( \del^{l^j} \varphi_{\rx,\zeta} )^{\mu^j} \prod_{j\geq j_1+1} ( \del^{l^j} \varphi_{\rx,\zeta} )^{\mu^j}$ and with a Taylor expansion at order 1 of $\del^{l^j}\varphi_{\rx,\zeta}$ in $\zeta'$ around 0 when $1\leq j\leq j_1$, we get $\del^{l^j}\varphi_{\rx,\zeta} = \sum_{1\leq i\leq n} \zeta'_i t_{i,j}^k$ where $t_{i,j}^k = \int_{0}^1 \del_{\zeta'}^{e_i+l^j} \varphi_{\rx,\zeta}(t\zeta')dt$. Thus, using the fact that $\varphi\in \O_{\sg,\ka_v,\eps_v,\eps_v,1}^{0,0,w_s}(\R^n)$, we see that $\prod_{j=1}^{j_1} (\del^{l^j}\varphi_{\rx,\zeta} )^{\mu^j}$ is a linear combination of terms of the form $\zeta'^\la V_\la$ where $|\la|=\sum_{j=1}^{j_1}|\mu^j|$ 
and 
$$
V_\la =\O( \langle \rx \rangle^{-\sg \eps_v \sum_1^{j_1}|l^j||\mu^j|} \langle \zeta\rangle^{\eps_v|\la| + \eps_v\sum_1^{j_1}|\mu^j||l^j|} \langle \zeta'\rangle^{(k_v +w_s)|\la|+ \ka_v\sum_{1}^{j_1}|l^j||\mu^j|}).
$$
As a consequence, we see that $\prod_{j=1}^{m} (\del^{l^j}\varphi_{\rx,\zeta} )^{\mu^j}$ is a linear combination of terms of the form $\zeta'^\la W_\la$ where $|\la|=\sum_{j=1}^{j_1}|\mu^j|$ and
$$
W_\la =\O( \langle \rx \rangle^{-\sg \eps_v (|\b|-v)} \langle \zeta\rangle^{2\eps_v|\b|} \langle \zeta'\rangle^{w'_s|\b|})
$$
where $v:=\sum_{j=j_1+1}^{m} |\mu^j|=|\om|-|\la|$. The first statement now follows from the inequality $2v\leq |\b|-|\la|$.

Since $\varphi_{\rx,\zeta}(0)=0$ and $(d\varphi_{\rx,\zeta})_0=0$, $P_{\b,\varphi}(\rx,\zeta,0,\vth)$ is a linear combination of terms of the form $\vth^\om \prod_{j=1}^m (\del^{0,0,l^j} \varphi(\rx,\zeta,0))^{\mu^j}$ with  
$1\leq |\om|\leq |\b|/2$, $1\leq m \leq |\b|$, $\mu^j\neq 0$, $|l^j|\geq 2$, $\sum_{j=1}^m |\mu^j|= |\om|$ and $\sum_{j=1}^{m} |\mu^j||l^j|=|\b|$.
We check with Lemma \ref{Csgr} $(i)$ that any function of the form $\prod_{j=1}^m (\del^{0,0,l^j} \varphi(\rx,\zeta,\zeta'))^{\mu^j}$ is in $\O_{\sg,\ka_v,\eps_v,|\b|/2}^{-\eps_v|\b|/2,\eps_v|\b|,(w_s/2+\ka_v)|\b|}(\R)$, and thus, $(\rx,\zeta,\vth)\mapsto \prod_{j=1}^m (\del^{0,0,l^j} \varphi(\rx,\zeta,0))^{\mu^j} \,1_{L(E_z)} \in \Pi_{\sg,\ka_v,z}^{-\eps_v |\b|/2,\eps_v|\b|,0}$. Since $(\rx,\zeta,\vth)\mapsto \vth^\om\, 1_{L(E_z)}\in \Pi_{\sg,\ka_v,z}^{0,0,|\b|/2}$ we obtain $(\rx,\zeta,\vth)\mapsto P_{\b,\varphi}(\rx,\zeta,0,\vth)\,1_{L(E_z)}\in \Pi_{\sg,\ka_v,z}^{-\eps|\b|/2,\eps_v|\b|,|\b|/2}$.

\noindent $(iii)$ We have
\begin{align*}
f_{\b,\varphi} (\rx,\zeta,\vth)&= \sum_{\b'\leq \b}\tbinom{\b}{\b'}\del^{\b'}_{\zeta'}(e^{2\pi i \langle \vth,\varphi_{\rx,\zeta}(\zeta')\rangle})_{\zeta'=0}\, \del^{0,0,\b-\b',\b}f(\rx,\zeta,0,L_{\rx,\zeta}(\vth))\\
&=\sum_{\b'\leq \b}\tbinom{\b}{\b'}P_{\b',\varphi}(\rx,\zeta,0,\vth)\, \del^{0,0,\b-\b',\b}f(\rx,\zeta,0,L_{\rx,\zeta}(\vth))\, .
\end{align*}
Since $(\rx,\zeta)\mapsto L_{\rx,\zeta}\in E^0_{\sg,\ka_v}(\M_n(\R))$ and $L_{\rx,\zeta}^{-1}=\O(1)$, we deduce from Lemma \ref{amptilde} $(i)$ and Lemma \ref{HEamp} $(iv)$ that $(\rx,\zeta,\vth)\mapsto \del^{0,0,\b-\b',\b}f(\rx,\zeta,0,L_{\rx,\zeta}(\vth))$ belongs to the amplitude space $\Pi_{\sg,\max\set{\ka,\ka_v},z}^{l-\eps_1|\b-\b'|,w_0+\ka|\b-\b'|,m-|\b|}$. The result now follows from $(ii)$.
\end{proof}

We now introduce two parametrized cut-off functions that will be used later. Let $b \in C^{\infty}_c(\R,[0,1])$ such that $b=1$ on $[-1/4,1/4]$ and $b=0$ on $\R\backslash]-1,1[$. We define for $\eps,\delta,\eta_1,\eta_2 >0$ with $\eps,\delta<1,$
\begin{align*}
&\chi_{\eps}(\vth,\vth'):=b(\tfrac{\norm{\vth'}^2}{\eps^2\langle \vth\rangle^{2}})\, ,\\
&\chi_{\delta,\eta}(\rx,\zeta,\zeta') := b(\tfrac{\norm{\zeta'}^2}{\delta^2\langle \rx\rangle^{2\sg \eta_1} \langle \zeta\rangle^{-2\eta_2}}).
\end{align*}

\begin{lem}
\label{lemchi} The cut-off functions $\chi_{\eps}$ and $\chi_{\delta,\eta}$ are repectively in the spaces $C^{\infty}(\R^{2n},[0,1])$ and $C^{\infty}(\R^{3n},[0,1])$ and satisfy: 

\noindent $(i)$ For any $(\rx,\zeta,\zeta')\in \R^{3n}$, if $\norm{\zeta'}\leq \half \delta \langle \rx\rangle^{\sg\eta_1}\langle \zeta\rangle^{-\eta_2}$, then $\chi_{\delta,\eta}(\rx,\zeta,\zeta')=1$, and if $\norm{\zeta'}\geq \delta\langle  \rx\rangle^{\sg\eta_1}\langle \zeta\rangle^{-\eta_2}$, then $\chi_{\delta,\eta}(\rx,\zeta,\zeta')=0$. In particular, for any $(\rx,\zeta)\in \R^{2n}$, $\chi_{\delta,\eta}(\rx,\zeta,0)=1$ and for any $3n$-multi-index $\nu\neq 0$, $(\del^\nu \chi_{\delta,\eta} )(\rx,\zeta,0)=0$.

\noindent $(ii)$ For any $(\vth,\vth')\in \R^{2n}$, if $\norm{\vth'}\leq \half \eps \langle \vth\rangle$, then $\chi_\eps(\vth,\vth')=1$, and if $\norm{\vth'}\geq \eps\langle \vth\rangle$, then $\chi_\eps(\vth,\vth')=0$. In particular, for any $\vth\in \R^n$, $\chi_\eps(\vth,0)=1$ and for any $2n$-multi-index $\nu\neq 0$, $(\del^\nu \chi_\eps)(\vth,0)=0$.

\noindent $(iii)$ For any $3n$-muti-index $\nu=(\a,\b,\ga)$, we have $\del^{\nu} \chi_{\delta,\eta}(\rx,\zeta,\zeta')=\O(\langle \rx\rangle^{-|\a|} \langle \zeta\rangle^{-\b}\langle \zeta'\rangle^{-|\ga|})$, and $\del^{\nu} \chi_{\delta,\eta}(\rx,\zeta,\zeta')=\O(\langle \rx\rangle^{-\sg|\nu|} \langle \zeta\rangle^{(-1+\eta_2/\eta_1)|\b|+(\eta_2/\eta_1)|\ga|}\langle \zeta'\rangle^{(\eta_1^{-1}-1)|\ga|+\eta_1^{-1}|\b|})$. In particular, the function $\chi_{\delta,\eta}$ is in $\O_{\sg,\ka'_\eta,\ka'_\eta,1,0}^{0,0,0}(\R)$ for a $\ka'_\eta>0$.

\noindent $(iv)$ For any $2n$-muti-index $\nu$, $\del^{\nu} \chi_\eps(\vth,\vth') =\O(\langle \vth\rangle^{-|\nu|} )$ and $\del^{\nu} \chi_\eps(\vth,\vth')=\O(\langle \vth'\rangle^{-|\nu|}).$

\end{lem}
\begin{proof}  $(i)$ and $(ii)$ are straightforward. For any $\nu\neq 0$, 
$\del^\nu \chi_{\delta,\eta} = \sum_{1\leq \nu'\leq |\nu|} P_{\nu,\nu'}(g)\, (\del^{\nu'}b)\,\circ\, g$ where $g(\rx,\zeta,\zeta'):= \tfrac{\norm{\zeta}^2}{\delta^2\langle \rx\rangle^{2\sg\eta_1}\langle \zeta\rangle^{-2\eta_2}}$. We obtain from a direct computation the estimate $P_{\nu,\nu'}(g)=\O(\langle \rx\rangle^{-2\sg\eta_1\nu'-|\a|}\langle \zeta\rangle^{2\eta_2\nu'-|\b|} \langle \zeta'\rangle^{2\nu'-|\ga|})$.
Since for any $\nu\in \N$, we have $\del^{\nu'} b=\O(1)$ we obtain $\del^{\nu} \chi_{\delta,\eta}=\O(\langle \rx\rangle^{-|\a|} \langle \zeta\rangle^{-\b}\langle \zeta'\rangle^{-|\ga|}1_{D_\delta})$ where $D_\delta$ is the set of triples $(\rx,\zeta,\zeta')$ satifying the inequalities $\delta/2\leq \langle  \zeta'\rangle \langle \rx\rangle^{-\sg\eta_1}\langle \zeta\rangle^{\eta_2}\leq \sqrt{2}$. The estimates of $(iii)$ follow. The proof of $(iv)$ is similar.
\end{proof}

We will use in the following lemma the space $\O_{\ka}^{t_0,t_1,j}$ ($\ka\geq 0$, $j\in \N$, $(t_0,t_1)\in \R_+^2$) of functions $f\in C^{\infty}(\R^{4n},\C)$ such that for any $\a\in \N^n$,
there is $C_\a>0$ such that for any $(\rx,\zeta,\zeta',\vth)\in \R^{4n}$,  $|\del^\a_{\zeta'} f (\rx,\zeta,\zeta',\vth)| \leq C_\a \langle \zeta\rangle^{t_0+\ka|\a|} \langle \zeta' \rangle^{t_1+\ka|\a|} \langle \vth\rangle^{-2j}$. Clearly, $\O_{\ka}^{t_0,t_1,j}\O_{\ka}^{t'_0,t'_1,j'}\subseteq \O_{\ka}^{t_0+t'_0,t_1+t'_1,j+j'}$ and $\del_{\zeta'}^\a \O_{\ka}^{t_0,t_1,j}\subseteq \O_\ka^{t_0+\ka|\a|,t_1+\ka|\a|,j}$. 

\begin{lem}
\label{lem-hL} .
Defining $h(\rx,\zeta,\zeta',\vth):= \big(1+\norm{^t(ds_{\rx,\zeta})_{\zeta'}(\vth)}^2-(i/2\pi)\langle \vth,(\Delta s_{\rx,\zeta})(\zeta')\rangle\big)^{-1}$, we have the following relation, valid for any $(\rx,\zeta,\zeta',\vth)\in \R^{4n}$, $p\in \N$,
\begin{align*}
e^{2\pi i \langle \vth,s_{\rx,\zeta}(\zeta')\rangle } = (h(\rx,\zeta,\zeta',\vth)\, L_{\zeta'})^p e^{2\pi i \langle \vth,s_{\rx,\zeta}(\zeta')\rangle}\, 
\end{align*}
where $L_{\zeta'}:=1-(2\pi)^{-2}\Delta_{\zeta'}$. Moreover, if $(C_\sg)$ holds, there is $\ka_L \geq 0$ such that for any $p\in \N$, there is $N_p\in \N^*$, $(h^p_k)_{1\leq k\leq N_p}$ functions in $\O_{\ka_L}^{2p\ka_L,2p\ka_L,p}$, $(\b^{k,p})_{1\leq k\leq N_p}$ $n$-multi-indices satisfying $|\b^{k,p}|\leq 2p$, such that  $(L_{\zeta'}\,h)^p = \sum_{k=1}^{N_p} h_k^{p}\,\del_{\zeta'}^{\b^{k,p}}$. 
\end{lem}
\begin{proof} We obtain $L_{\zeta'} e^{2\pi i \langle \vth,s_{\rx,\zeta}(\zeta')\rangle}= (1/h) e^{2\pi i \langle \vth,s_{\rx,\zeta}(\zeta')\rangle}$ through a direct computation. Let us show the remaining statement by induction on $p$. Note that by Lemma \ref{Csgr} $(ii)$, we have $|1/h|\geq c  \langle\vth\rangle^{2}$ for a $c>0$ and we check that $1/h \in \wt\Pi_{\sg,\ka_v,\eps_v,z}^{0,\eps_v,w'_v,2}$ where $w'_v=\max\set{2w_v,w_v+\ka_v}$. With a reccurence or using Proposition \ref{inverse}, we check that $h\in \O_{\ka_L}^{0,0,1}$ where $\ka_L:=\max\set{2\eps_v,w'_v+\ka_v}$. The property is obviously true for $p=0$.  Suppose now that the property is true for $p\geq 0$, so that $(L_{\zeta'}\,h)^p =\sum_{k=1}^{N_p} h_k^{p}\,\del_{\zeta'}^{\b^{k,p}}$ with $N_p\in \N^*$, $(h^p_k)_{1\leq k\leq N_p}$ functions in $\O_{\ka_L}^{2p\ka_L,2p\ka_L,p}$ and $(\b^{k,p})_{1\leq k\leq N_p}$ $n$-multi-indices satisfying $|\b^{k,p}|\leq 2p$.  We also have 
\begin{align*}
&(L_{\zeta'} h)^{p+1} = (L_{\zeta'} h) \sum_{k=1}^{N_p} h_k^{p}\,\del_{\zeta'}^{\b^{k,p}} = \sum_{k=1}^{N_p} hh^{p}_k \del_{\zeta'}^{\b^{k,p}} -(2\pi)^{-2}\big(\Delta_{\zeta'}(h h_k^p) \del_{\zeta'}^{\b^{k,p}} \\
&\hspace{4cm}+2 \sum_{i=1}^n \del_{\zeta'_i} (hh^p_k)\del^{\b^{k,p}+e_i}_{\zeta'} +hh^{p}_k\Delta_{\zeta'}\del^{\b^{k,p}}_{\zeta'}\big)
\end{align*}
so the property holds for $p+1$.
\end{proof}

We note $\S_{\sg,c}(\R^{3n},L(E_z))$ the space of smooth functions $f$ such that for any $N\in \N^*$ and $\nu=(\mu,\ga)\in \N^{2n}\times \N^n$, $\del^\nu f (\rx,\zeta,\vth)=\O(\langle \rx\rangle^{-\sg N}\langle \zeta\rangle^{c_0+c_1 N +c_2 |\mu|}\langle  \vth \rangle^{-N})$. It follows from Lemma \ref{noyauReste} that if $f\in \S_{\sg,c}(\R^{3n},L(E_z))$, then $\Op_\Ga(f) \in \Op_\Ga(S^{-\infty}_{\sg,z})$. Here and thereafter $\Ga$ satisfies the hypothesis of Lemma \ref{noyauReste}.

\begin{lem}
\label{lemnegamp} Assume that $(C_\sg)$ holds.

\noindent (i) For any $l,w_0,w_1,m,\ka$, $S_{m,w_1}(\wt \Pi_{\sg,\ka,\eps_1,z}^{l,w_0,w_1,,m})\subseteq \S_{\sg,c}(\R^{3n},L(E_z))$ for a triple $c:=(c_0,c_1,c_2)$ and the linear map $S_{m,w_1}:f\mapsto S_{m,w_1}(f)$ is continuous, where
$$
S_{m,w_1}(f):(\rx,\zeta,\vth)\mapsto \int_{\R^{2n}} e^{2\pi i (\langle \vth',\zeta'\rangle+ \langle \vth,s_{\rx,\zeta}(\zeta')\rangle)}\, ^t M_{\vth'}^{p_{m,w_1},\zeta'}(f)(\rx,\zeta,\zeta',\vth') (1-\chi_{\delta,\eta})(\rx,\zeta,\zeta')\, d\vth'\,d\zeta'
$$
and $p_{m,w_1}:=\max \set{m+2n,[|w_1|]+1+2n}$.

\noindent (ii) For any $u\in \S(\R^{2n},L(E_z))$, the linear application $f\mapsto \langle\Op_\Ga S_{m,w_1}(f),u\rangle$ 
is continuous.
\end{lem}
\begin{proof} We fix $N\in \N^*$. First note that $S_{m,w_1}(f)$ is well-defined since for any $(\rx,\zeta)\in \R^{2n}$, there is $C_{\rx,\zeta}>0$ such that $\norm{\, ^t M_{\vth'}^{p_{m,w_1},\zeta'}(f)(\rx,\zeta,\zeta',\vth')(1-\chi_{\delta,\eta})(\rx,\zeta,\zeta')}\leq C_{\rx,\zeta} \langle \vth'\rangle^{-2n}\langle \zeta'\rangle^{-2n}$. Since for any $n$-multi-index $\delta$, $\del^\delta_{\vth'}\, ^t M_{\vth'}^{p_{m,w_1},\zeta'}(f)$ decrease to zero with $\vth'$, we can successively integrate by parts with (\ref{Mformula}), which is valid since $1-\chi_{\delta,\eta}$ assures that $\norm{\zeta'}\geq \half\delta$ on the domain of integration. We obtain thus for any $q\in \N^*$,  
$$
S_{m,w_1}(f):(\rx,\zeta,\vth)\mapsto \int_{\R^{2n}} e^{2\pi i (\langle \vth',\zeta'\rangle+ \langle \vth,s_{\rx,\zeta}(\zeta')\rangle)}\, ^t M_{\vth'}^{p_{m,w_1}+q,\zeta'}(f) (1-\chi_{\delta,\eta})\, d\vth'\,d\zeta'\, .
$$
We note $f_q$ the integrand of the previous integral. If $\nu=(\a,\b,\ga)=(\mu,\ga)$ is a $3n$-multi-index, we see with Lemma \ref{derivphase} that
\begin{align*}
&\del^\nu_{\rx,\zeta,\vth} f_q = e^{2\pi i \langle \vth',\zeta'\rangle}\sum_{\mu'\leq \mu} \tbinom{\mu}{\mu'} e^{2\pi i \langle \vth,s_{\rx,\zeta}(\zeta')\rangle} \sum_{|\om|\leq |\mu'|} \vth^\om T_{\nu',\om,s}(\rx,\zeta,\zeta')\\ 
&\hspace{2cm} \sum_{|\wt\delta|=p_{m,w}+q} \la_\delta (-1)^{|\wt\delta|} \tfrac{\zeta'^{\wt\delta}}{\norm{\zeta'}^{2(p_{m,w_1}+q)}} \del^{\mu-\mu'}_{\rx,\zeta} \del^{\wt\delta}_{\vth'} (f (1-\chi_{\delta,\eta}))\, .
\end{align*}
By Lemma \ref{lemchi} $(iii)$, $(\rx,\zeta,\zeta',\vth')\mapsto \chi_{\delta,\eta}(\rx,\zeta,\zeta') \, 1_{L(E_z)}$ is in $\wt \Pi_{\sg,\ka'_\eta,1,z}^{0,0,0,0}$, so the multiplication operator $f\mapsto f(1-\chi_{\delta,\eta})$ is continuous from $\wt \Pi_{\sg,\ka,\eps_1,z}^{l,w_0,w_1,m}$ into $\wt \Pi_{\sg,\ka_\eta,\eps_1,z}^{l,w_0,w_1,m}$, where $\ka_\eta=\max\set{\ka,\ka'_\eta}$.
Since $\norm{\zeta'}\geq \delta/2$ in the support of $f(1-\chi_{\delta,\eta})$, we get from Lemma \ref{phasecontrol} $(i)$ the following estimates, where $\ka''_\eta:=\ka_v+w_s+\ka_\eta$,
\begin{align*}
&\norm{\del^\nu_{\rx,\zeta,\vth} f_q(\rx,\zeta,\vth,\zeta',\vth')} \leq C_{\nu,q} \langle \vth\rangle^{|\mu|} \langle \vth'\rangle^{m-p_{m,w_1}-q} \sum_{\mu'\leq \mu}\langle \zeta\rangle^{\ka_v|\mu'|+w_0+\ka_\eta|\mu|}\\&\hspace{5cm}\times \langle \zeta'\rangle^{w_1+(\ka_v+w_s)|\mu'|+\ka_\eta|\mu|-(p_{m,w_1}+q) +w_s |\ga|} \langle\rx \rangle^{\sg|l|}\\
&\hspace{4cm} \leq C'_{\nu,q}  \langle \rx\rangle^{\sg|l|} \langle \zeta\rangle^{w_0+\ka''_\eta|\mu|} \langle \vth\rangle^{|\mu|} \langle \vth'\rangle^{m-p_{m,w_1}-q}  \langle \zeta'\rangle^{w_1+\ka''_\eta|\nu|-p_{m,w_1}-q}\, .
\end{align*}
If $k\in \N^*$, and if we set $q:=q_k$ such that $w_1+\ka''_\eta k-p_{m,w_1}-q_k\leq -2n$, we see by applying the theorem of derivation under the integral sign that $S_{m,w}(f)$ is smooth and for any $3n$-multi-index $\nu=(\a,\b,\ga)$ and $q\in \N^*$, after integrations by parts in $\vth'$, with $\nu':=(\mu',\ga)$,
\begin{align*}
&\del^\nu S_{m,w_1}(f)(\rx,\zeta,\vth)=\sum_{\mu'\leq \mu}\sum_{|\om|\leq |\mu'|} \tbinom{\mu}{\mu'} \vth^\om \int_{\R^{2n}} e^{2\pi i (\langle \vth',\zeta'\rangle + \langle \vth,s_{\rx,\zeta}(\zeta')\rangle)} T_{\nu',\om,s}(\rx,\zeta,\zeta')\\
&\hspace{5cm} ^tM_{\vth'}^{p_{m,w_1}+q_{|\nu|}+q,\zeta'} \del^{\mu-\mu'}_{\rx,\zeta} (f (1-\chi_{\delta,\eta}))\, d\vth'\,d\zeta'\, .
\end{align*}
We note $g_q(\rx,\zeta,\zeta',\vth'):=e^{2\pi i \langle \vth',\zeta'\rangle}T_{\nu',\om,s}(\rx,\zeta,\zeta') ^tM_{\vth'}^{p_{m,w_1}+q_{|\nu|}+q,\zeta'} \del^{\mu-\mu'}_{\rx,\zeta} (f (1-\chi_{\delta,\eta}))$. Using now Lemma \ref{lem-hL}, we get the estimates for any $p\in \N$,
\begin{align*}
&\norm{(L_{\zeta'}h)^p g_q(\rx,\zeta,\zeta',\vth')}\leq C_p  \langle \zeta'\rangle^{2p\ka_L}  \langle \zeta\rangle^{2p\ka_L}\langle\vth\rangle^{-2p} \sum_{k=1}^{N_p} \norm{\del_{\zeta'}^{\b^{k,p}} g_q(\rx,\zeta,\zeta',\vth')} \, .
\end{align*}
Thus, with
Lemma \ref{phasecontrol} $(i)$, we obtain with $k_1:=w_s+\ka_v+\ka_\eta+\ka_L$,
\begin{align*}
&\norm{(L_{\zeta'}h)^p g_q(\rx,\zeta,\zeta',\vth')}\leq C'_p \langle \rx\rangle^{\sg |l|} \langle \zeta'\rangle^{w_1+(2p+|\nu|)k_1-p_{m,w_1}-q_{|\nu|}-q} \langle \vth\rangle^{-2p}\\ &\langle \vth'\rangle^{2p +m-p_{m,w_1}-q_{|\nu|}-q} 
\langle \zeta\rangle^{(2p+|\mu|)k_1 +w_0}\sum_{|\wt \b|\leq 2p}\sum_{\mu'\leq \mu} \sum_{|\wt\delta|=p_{m,w_1}+q_{|\nu|}+q} q_{\mu',\wt\b,\wt\delta}(f(1-\chi_{\delta,\eta}))\, 1_{D}(\rx,\zeta,\zeta')
\end{align*}
where $D:=\set{(\rx,\zeta,\zeta')\in \R^{2n}\ | \ \norm{\zeta'}\geq \half \delta \langle \rx\rangle^{\sg\eta_1}\langle\zeta\rangle^{-\eta_2}}$.
If we now fix $p$ such that $-N-2\leq -2p+|\mu|\leq -N$, we see that by taking $q$ such that $A_q\leq -N/\eta_1 -|l|/\eta_1$ where $A_q:=w_1+(2p+|\nu|)k_1-p_{m,w_1}-q_{|\nu|}-q+2n$, and $2p+m-p_{m,w_1}-q_{|\nu|}-q\leq -2n$, we can successively integrate by parts in $\zeta'$ ($p$ times) using the formula of Lemma \ref{lem-hL}. We obtain then the estimate for given constants $c_0, c_1, c_2 >0$, 
\begin{align*}
&\norm{\del^\nu S_{m,w_1}(f)(\rx,\zeta,\vth)}\leq C_{\nu,N} \langle \rx\rangle^{-\sg N}\langle \zeta\rangle^{c_0+ c_1 N +c_2 |\mu|} \langle \vth\rangle^{-N}\\&\hspace{4cm}\sum_{|\wt \b|\leq 2p}\,
\sum_{\mu'\leq \mu}\, \sum_{|\wt \delta|=p_{m,w_1}+q_{|\nu|}+q} q_{\mu',\wt\b,\wt \delta}(f(1-\chi_{\delta,\eta}))
\end{align*}
which yields the result.

\noindent $(ii)$ This statement follows from $(i)$ and Lemma \ref{ampliOP} $(ii)$.
\end{proof}

\begin{lem}
\label{Pi-amp}
Suppose $(C_\sg)$. 

\noindent (i) Defining for any $f\in \wt \Pi_{\sg,\ka,\eps_1,z}^{l,w_0,w_1,m}$,
$$
\Pi(f):(\rx,\zeta,\vth)\mapsto \int_{\R^{2n}} e^{2\pi i (\langle \vth',\zeta'\rangle+\langle \vth,\varphi_{\rx,\zeta}(\zeta')\rangle)} f(\rx,\zeta,\zeta',\vth'+L_{\rx,\zeta}(\vth))\,\chi_{\delta,\eta}(\rx,\zeta,\zeta')\,d\zeta'\,d\vth'\, ,
$$
there is $\delta,\eta,$ such that for any $N\geq |m|$, we have $\Pi(f)=\Pi_N(f) +\Pi_{R,N}(f)$ where 
$
 \Pi_{N}(f) = \sum_{0\leq |\b|\leq N} \tfrac{(i/2\pi)^{|\b|}}{\b!} f_{\b,\varphi}
$
and there is such that $\Pi_{R,N}(f)$ satisfies the estimates for any $3n$-multi-index $\nu=(\mu,\ga)\in \N^{2n}\times \N^n$,
$$
\del^\nu \Pi_{R,N}(f) = \O( \langle \rx\rangle^{\sg(l-\eps'_1(N+1))} \langle \zeta\rangle^{k_0+k_1(N+1+|\mu|) +\eps_v|\ga|} \langle \vth\rangle^{m+|\mu|-(N+1)/2 +n})
$$
where $\eps'_1,k_0,k_1>0$.

\noindent (ii) We have for any $3n$-multi-index $\nu=(\mu,\ga)\in \N^{2n}\times \N^n$, 
$$
\del^\nu \Pi(f) = \O( \langle \rx\rangle^{\sg l} \langle \zeta\rangle^{k'_0+k'_1|\mu| +\eps_v|\ga|} \langle \vth\rangle^{m})
$$
where $k'_0,k'_1>0$. In particular, for any $u\in \S(\R^{2n},L(E_z))$, the linear application $f\mapsto \langle\Op_\Ga\Pi(f),u\rangle$ 
is continuous.
\end{lem}
\begin{proof} 
$(i)$
We proceed to a Taylor expansion of $\wt f(\rx,\zeta,\zeta',\vth',\vth):=f(\rx,\zeta,\zeta',\vth'+L_{\rx,\zeta}(\vth))$ in $\vth'$ around zero at order $N\in \N^*$, so that 
$$
\Pi(f)= \sum_{0\leq |\b|\leq N} \tfrac{1}{\b!} I_{\b}(f) + \sum_{|\b|=N+1} \tfrac{N+1}{\b!} R_{\b,N}(f)=:\Pi_N(f)+\Pi_{R,N}(f)
$$
where 
\begin{align*}
&I_\b(f)= \int_{\R^{2n}} \vth'^\b e^{2\pi i (\langle \vth',\zeta'\rangle+\langle \vth,\varphi_{\rx,\zeta}(\zeta')\rangle)} \del^{0,0,0,\b}f(\rx,\zeta,\zeta',L_{\rx,\zeta}(\vth))\,\chi_{\delta,\eta}(\rx,\zeta,\zeta')\,d\zeta'\,d\vth'\, ,\\
&R_{\b,N}(f)= \int_{\R^{2n}} \vth'^\b e^{2\pi i (\langle \vth',\zeta'\rangle+\langle \vth,\varphi_{\rx,\zeta}(\zeta')\rangle)} r_{\b,N,f}(\rx,\zeta,\zeta',\vth',\vth)\,d\zeta'\,d\vth'\, ,
\end{align*}
and $r_{\b,N,f}:=\int_{0}^1(1-t)^N\del^{0,0,0,\b} f_\chi(\rx,\zeta,\zeta',t\vth'+L_{\rx,\zeta}(\vth))\,dt$, $f_\chi:=f\chi_{\delta,\eta}\in\wt \Pi_{\sg,\ka_\eta,z}^{l,w_0,w_1,m}$.
By integration by parts in $\zeta'$ in the integrals $I_\b(f)$, we get
$$
\Pi_{N}(f)= \sum_{0\leq |\b|\leq N} \tfrac{(i/2\pi)^{|\b|}}{\b!} \del^{\b}_{\zeta'}\big(e^{2\pi i \langle \vth,\varphi_{\rx,\zeta}(\zeta')\rangle} \del^{0,0,0,\b}f(\rx,\zeta,\zeta',L_{\rx,\zeta}(\vth))\big)_{\zeta'=0}=\sum_{0\leq |\b|\leq N} \tfrac{(i/2\pi)^{|\b|}}{\b!} f_{\b,\varphi}\, .
$$
Using integration by parts in $\zeta'$, we obtain $R_{\b,N,f}=(i/2\pi)^{|\b|} I_f$, where for any $p\in \N$, 
\begin{align*}
&I_f(\rx,\zeta,\vth):= \int_{\R^{2n}} e^{2\pi i \langle\vth',\zeta' \rangle} \del_{\zeta'}^\b G(\rx,\zeta,\zeta',\vth',\vth)\,d\zeta'\,d\vth'\,, \\
&G(\rx,\zeta,\zeta',\vth',\vth):= e^{2\pi i \langle \vth,\varphi_{\rx,\zeta}(\zeta')\rangle} r_{\b,N,f}(\rx,\zeta,\zeta',\vth',\vth) \, .
\end{align*}
Using integration by parts in $\zeta'$ and $e^{2\pi i \langle\vth',\zeta' \rangle} =\langle \vth'\rangle^{-2p}L_{\zeta'}^p e^{2\pi i \langle\vth',\zeta' \rangle}$, we check that $I_f$ is smooth on $\R^{3n}$ and if $\nu$ is a $3n$-multi-index, we see that $\del^\nu I_f$ is a linear combination of terms of the form
$$
J_f:= \vth^{\wt\om} \int_{\R^{2n}}  e^{2\pi i (\langle \vth',\zeta'\rangle+\langle \vth,\varphi_{\rx,\zeta}(\zeta')\rangle)} \del^{\b^1}_{\zeta'} T_{\nu',\wt \om,\varphi} P_{\b^2,\varphi} \del^{\nu-\nu'}_{\rx,\zeta,\vth} \del^{\b^3}_{\zeta'} r_{\b,N,f}\, d\zeta'\, d\vth'
$$
where $|\wt \om|\leq |\mu'|$, $\nu'\leq \nu$, $\sum \b^i =\b$, $|\b|=N+1$. We now cut the integral $J_f$ in two parts $J_{\chi}+ J_{1-\chi}$, where the cut-off function $\chi_{\eps}(\vth,\vth')$ appears in $J_{\chi}$.

\noindent \emph{Analysis of $J_{\chi}$}

Using Lemma \ref{phasecontrol} $(ii)$ and integration by parts in $\zeta'$, we see that $J_\chi$ is a linear combination of terms of the form 
$$
J_{\chi,\om} = \vth^{\wt\om}\vth^\om \int_{\R^{2n}}e^{2\pi i (\langle \vth',\zeta'\rangle+\langle \vth,\varphi_{\rx,\zeta}(\zeta')\rangle)}\langle \zeta'\rangle^{-2p}
t_{\om,\la}\, \del_\zeta'^{\b^1} T_{\nu',\wt \om,\varphi} \, \del_{\vth'}^{\la'}\del_{\rx,\zeta,\vth}^{\nu-\nu'} \del^{\b^3} r_{\b,N,f} \, \del^{\la+\rho-\la'}\chi_\eps\, d\zeta'\,d\vth' 
$$ 
where $p\in \N$, $|\rho|\leq 2p$, $|\om|\leq |\b^2|$, $(2|\om|-|\b^2|)_+\leq |\la|\leq |\om|$, $\la'\leq \la+\rho$. We now fix $\eps$ such that $\eps<c/2$ where $c$ is a constant such that $c\langle \vth\rangle\leq \langle L_{\rx,\zeta}(\vth)\rangle$. Thus, in the domain of integration of $J_{\chi,\om}$, we have for any $t\in [0,1]$, $\langle t\vth' + L_{\rx,\zeta}(\vth)\rangle \geq c_1 \langle \vth \rangle$ for a $c_1>0$.
As a consequence, we obtain the following estimate:
\begin{align*}
&\norm{\del_{\vth'}^{\la'}\del_{\rx,\zeta,\vth}^{\nu-\nu'} \del^{\b^3}_{\zeta'} r_{\b,N,f}}\leq C\langle \rx\rangle^{\sg(l-\eps_1|\b^3|)} \langle \zeta\rangle^{(\ka_v+\ka_\eta)|\mu-\mu'|+w_0+\ka_\eta|\b^3|}\\
&\hspace{4cm} \langle \zeta'\rangle^{w_1+\ka_\eta(|\mu-\mu'|+|\b^3|)} \langle \vth\rangle^{|\mu-\mu'|+m-|\b|-|\la'|}\, .
\end{align*}
We also deduce from Lemma \ref{phasecontrol} the estimate
$$
|t_{\om,\la}\, \del_\zeta'^{\b^1} T_{\nu',\wt \om,\varphi} |\leq C'\langle \rx\rangle^{-\sg(|\mu'|+(\eps/2)|\b^1+\b^2|)} \langle \zeta\rangle^{2\eps_v|\b^1+\b^2|+(\ka_v+\eps_v)|\mu|+\eps_v|\ga|} \langle \zeta'\rangle^{c_1(N+1)+c_2|\nu|}\, .
$$
As a consequence, by taking $p$ sufficiently big, the integrand $j(\rx,\zeta,\zeta',\vth,\vth')$ of $J_{\chi,\om}$ satisfies the estimate, for a $\eps'_1>0$ and a $k_1>0$,
$$
\norm{j}\leq C'' \langle \rx\rangle^{\sg(l-\eps'_1(N+1))} \langle \zeta\rangle^{w_0+k_1(N+1+|\mu|) +\eps_v|\ga|} \langle \zeta'\rangle^{-2n} \langle \vth\rangle^{m+|\mu|-(N+1)/2}\,  1_{D_\eps}(\vth,\vth')
$$
where $D_\eps$ is the set of $(\vth,\vth')$ in $\R^{2n}$ such that $\norm{\vth'}\leq \eps \langle \vth\rangle$.
We deduce finally that for any $\nu \in \N^{3n}$, 
$$
J_{\chi} = \O( \langle \rx\rangle^{\sg(l-\eps'_1(N+1))} \langle \zeta\rangle^{w_0+k_1(N+1+|\mu|) +\eps_v|\ga|} \langle \vth\rangle^{m+|\mu|-(N+1)/2 +n})\, .
$$

\noindent \emph{Analysis of $J_{1-\chi}$}

We set $\om:=\langle \zeta',\vth'\rangle+\langle \vth,\varphi_{\rx,\zeta}(\zeta')\rangle$. 
By Lemma \ref{Csgr} $(i)$, we have $\sum_i\norm{\del_{\zeta'_i} \varphi_{\rx,\zeta}(\zeta')} \leq C\langle\rx \rangle^{-\sg\eps_v } \langle \zeta\rangle^{c_1} \langle \zeta'\rangle^{c_2})$ for $C,c_1,c_2>0$. The presence of $\chi_{\delta,\eta}$ in the integrand of $J_{1-\chi}$ allows to use the estimate $\langle \zeta'\rangle\leq \sqrt{2} \delta \langle \rx\rangle^{\sg\eta_1} \langle \zeta\rangle^{-\eta_2}$, so that $\sum_i\norm{\del_{\zeta'_i} \varphi_{\rx,\zeta}(\zeta')} \leq C\, 2^{c_2/2}\, \delta^{c_2}$
by taking $\eta_1\leq \eps_v/c_2$ and $\eta_2\geq c_1/c_2$.
As a consequence, we obtain the following estimate in the domain of integration of $J_{1-\chi}$,
$$
|\nabla_{\zeta'} \om|^2 \geq \norm{\vth'}^2 (1- \tfrac{4}{\eps}\,C\, 2^{c_2/2}\, \delta^{c_2})\, . 
$$
We now fix $\delta$ such that $\tfrac{4}{\eps}\,C\, 2^{c_2/2}\, \delta^{c_2}<1$ so that there is $k>0$ such that $|\nabla_{\zeta'} \om| \geq k \norm{\vth'}$. Noting $U_{\zeta'}:=(2\pi i |\nabla_{\zeta'} \om|^{2})^{-1}\sum_i (\del_{\zeta'_i} \om ) \del_{\zeta'_i}$ we have (see for instance \cite{Ruzhansky}) $U_{\zeta'} e^{2\pi i \om} = e^{2\pi i \om}$ and 
$$
(^t U_{\zeta'})^r = \tfrac{1}{ |\nabla_{\zeta'} \om|^{4r} } \sum_{|\rho|\leq r} P^{\om}_{\rho,r} \del^\rho_{\zeta'}
$$
where $P_{\rho,r}^\om$ is a linear combination of terms of the form $(\nabla_{\zeta'} \om)^\pi \del^{\delta^1}_{\zeta'} \om \cdots\del^{\delta^r}_{\zeta'} \om$, with $|\pi|=2r$, $|\delta^{i}|>0$ and $\sum_{j=1}^r |\delta^j| + |\rho| =2r$. We thus obtain after integration by parts in $\zeta'$, for any $r\in \N^*$, that $J_{1-\chi}$ is a linear combination of integrals of the form
$$
 \vth^{\wt\om+\wh \om} \int_{\R^{2n}}  e^{2\pi i \om} (^t U_{\zeta'})^r\big(\del^{\b^1}_{\zeta'} T_{\nu',\wt \om,\varphi} P_{\wh\om,\b^2,\varphi} \del^{\nu-\nu'}_{\rx,\zeta,\vth} \del^{\b^3}_{\zeta'} r_{\b,N,f}\,\big)(1-\chi_\eps) d\zeta'\, d\vth'
$$
where $|\wh \om|\leq|\b^2|$. We noted $P_{\b^2,\varphi}=:\sum_{\wh \om} P_{\wh \om,\b^2,\varphi} \vth^{\wh \om}$. By Lemma \ref{phasecontrol} $(ii)$, we see that $P_{\wh \om,\b^2,\varphi} \in \O_{\sg,\ka_v,\eps_v,\eps_v,2|\b^2|}^{-\eps_v|\b^2|/2,2\eps_v|\b^2|,(w'_s+1)|\b^2|}$. Let us note $\wt T:=\del^{\b^1}_{\zeta'} T_{\nu',\wt \om,\varphi} P_{\wh\om,\b^2,\varphi}$. Lemma \ref{phasecontrol} $(i)$ yields $\wt T \in \O_{\sg,\ka_v,\eps_v,\eps_v,2(|\nu|+N)}^{-(\eps_v/2)|\b^1+\b^2|,c_0(|\mu|+N)+\eps_v|\ga|,c_0(|\nu|+N)}(\R)$ for a constant $c_0>0$.
With our choice of the parameters $\eta_1$ and $\eta_2$, we also have the following estimate, valid in the domain of integration of $J_{1-\chi}$, 
$$
\del^{\la}_{\vth'} \del^{\ga+e_i}_{\zeta'} \om  = \O\big( \langle \zeta\rangle^{\eps_v|\ga|} \langle \zeta'\rangle^{\ka_v|\ga|} \langle \vth'\rangle^{1-|\la|} \big)\, .
$$
In particular, noting $\O_{\ka_v}^{l,m}$ the space of smooth functions $f$ such that for any $n$-multi-indices $\la,\ga$, $\del^{\la}_{\vth'}\del^{\ga}_{\zeta'} f =\O\big( (\langle \zeta\rangle\langle \zeta'\rangle)^{l+\ka_v|\ga|} \langle \vth'\rangle^{m}\big)$, we see that $|\nabla_{\zeta'} \om|^2 \in \O_{\ka_v}^{0,2}$, and for any $\la\in \N^n$, $\del^\la_{\vth'} |\nabla_{\zeta'}\om|^{-4r}$ = $\O(\langle \vth'\rangle^{-4r})$. Moreover, each term $P_{\rho,r}^\om$ is in $\O_{\ka_v}^{\ka_v r , 3r}$ so that finally, for any $\la\in \N^n$ 
$$
\del^{\la}_{\vth'} \tfrac{P^{\om}_{\rho,r}}{ |\nabla_{\zeta'} \om|^{4r} }   = \O\big( (\langle \zeta\rangle\langle \zeta'\rangle)^{\ka_v r } \langle \vth'\rangle^{-r}\big)\, .
$$

We easily check that if $r\geq 2n$, then $h:=(^t U_{\zeta'})^r\big(\del^{\b^1}_{\zeta'} \wt T \del^{\nu-\nu'}_{\rx,\zeta,\vth} \del^{\b^3}_{\zeta'} r_{\b,N,f}\,\big)(1-\chi_\eps) $ satisfies the estimates for any $q\in \N$, $\norm{L_{\vth'}^q h }\leq C_{\rx,\zeta,\zeta',\vth,q} \langle \vth'\rangle^{-2n}$. As a consequence, we can permute the integration $d\zeta'd\vth'\to d\vth' d\zeta'$ and successively integrate by parts in $\vth'$, so that finally $J_{1-\chi}$ is a linear combination of terms of the form
$$
 \vth^{\wt\om+\wh \om} \int_{\R^{2n}}  e^{2\pi i \om} \langle \zeta'\rangle^{-2q} \del_{\vth'}^{\la^1}\tfrac{P^{\om}_{\rho,r}}{ |\nabla_{\zeta'} \om|^{4r} } \, \del^{\rho^1}_{\zeta'}\wt T\, \del_{\vth'}^{\la^2}\del^{\nu-\nu'}_{\rx,\zeta,\vth} \del^{\b^3+\rho^2}_{\zeta'} r_{\b,N,f}\,\del_{\vth'}^{\la^3}(1-\chi_\eps) d\vth'\,d\zeta' 
$$
where $\sum_{i}\la^i = \la$, $|\la|\leq 2q$, $\sum_i \rho^i = \rho$, $|\rho|\leq r$. We also have the following estimate for $c'_0, c'_1>0$,
$$
\del_{\vth'}^{\la^2}\del^{\nu-\nu'}_{\rx,\zeta,\vth} \del^{\b^3+\rho^2}_{\zeta'} r_{\b,N,f}= \O\big(\langle \rx\rangle^{\sg(l-|\b^3|)}(\langle \zeta\rangle\langle \zeta\rangle)^{c'_0+c'_1(|\mu-\mu'|+|\b^3|+|\rho^2|)}\big)\, .
$$
With Lemma \ref{lemchi} $(iv)$ we now see that the integrand $j'$ of the previous integral is estimated by
$$
\norm{j'} \leq C \langle \vth' \rangle^{-r+|\mu|+N+1}\langle \rx\rangle^{\sg(l-\eps'_1(N+1))} \langle \zeta\rangle^{k_0+k_1 N+k_2 r +k_3|\mu|+\eps_v |\ga|}\langle \zeta'\rangle^{-2q +k_0+k_1N +k_2r+k_3|\nu|} 
$$
for constants $k_0,k_1,k_2,k_3>0$. If we now fix $r\geq 2n$ such that $-r+|\mu|+N+1 +2n = m+|\mu|-(N+1) +n$, and $q$ such that $-2q +k_0+k_1N +k_2r+k_3|\nu|\leq -2n$
we finally obtain the estimate
$\nu \in \N^{3n}$, 
$$
J_{1-\chi} = \O( \langle \rx\rangle^{\sg(l-\eps'_1(N+1))} \langle \zeta\rangle^{k'_0+k'_1(N+1+|\mu|) +\eps_v|\ga|} \langle \vth\rangle^{m+|\mu|-(N+1) +n})\, .
$$
The result follows now from this estimate and the one obtained for $J_{\chi}$.

\noindent $(ii)$ The estimate is obtained by applying $(i)$ and $N+1= \max\set{ 2(n+|\mu|),|m|}$. The second statement is then a consequence of Lemma \ref{ampliOP} $(ii)$.
\end{proof}

\begin{thm} 
\label{compo} If $(C_\sg)$ holds, 
$\Psi_{\sigma}^\infty$ is a $*$-subalgebra of $\Re(\S)$. Moreover, if $A \in \Psi_{\sigma}^{l',m'}$ and $B\in \Psi_{\sigma}^{l,m}$, then $AB\in \Psi_{\sigma}^{l+l',m+m'}$ with the following asymptotic expansion of the normal symbol of $AB$, in a frame $(z,\bfr)$:
$$
\sigma_{0}(AB)_{z,\bfr} \sim \sum_{\b,\ga \in \N^n} c_\b c_\ga \del_{\zeta,\vth}^{\ga,\ga}\big( a(\rx,\vth)\del^{\b}_{\zeta'}\big(e^{2\pi i \langle \vth,\varphi_{\rx,\zeta}(\zeta')\rangle} (\del^{\b}_{\vth'} f_b)(\rx,\zeta,\zeta',L_{\rx,\zeta}(\vth))\big)_{\zeta'=0} \tau^{-1}_{\rx,\zeta} \big)_{\zeta=0}
$$
where  $a:=\sigma_0(A)_{z,\bfr}$, $b:=\sigma_0(B)_{z,\bfr}$, $c_\b:=(i/2\pi)^{|\b|}/\b!$ and 
$$
f_b(\rx,\zeta,\zeta',\vth'):=\tau_{\rx,r_{\rx,\zeta}(\zeta')}\,b\circ \wt \Xi(\rx,\zeta,\zeta',\vth')\,\tau_{\rx^{\zeta,\zeta'},q_{\rx,\zeta}(\zeta')}\,|J(R)|(\rx,\zeta,\zeta')\,|\det (P^{z,\bfr}_{-1,\psi(\rx,\zeta),\zeta'})^{-1}|\, .
$$

\end{thm} 
\begin{proof} We fix a frame $(z,\bfr)$. We note $K_{AB}$ the kernel of the operator $AB$. As a consequence of Proposition \ref{regularity} we have for any $u,v\in \S(\R^n,E_z)$, $\langle(K_{AB})_{z,\bfr},u\ox \ol v \rangle=\big( A_{z,\bfr} (\mu^{-1} B_{z,\bfr}(v))| u\big)$. We shall note $g:=A_{z,\bfr} (\mu^{-1} B_{z,\bfr}(v))$. A computation shows that for any $\rx\in \R^n$, 
$g(\rx)= \int_{\R^n} \mu a(\rx,\vth) \,\wt b(\rx,\vth)\, d\vth$, and
$$
\wt b(\rx,\vth):= \int_{\R^{3n}} e^{2\pi i (\langle \vth,\zeta\rangle +\langle \vth',\zeta'\rangle)} \tau_{\rx,\zeta} b(\psi(\rx,\zeta),\vth')\tau_{\psi(\rx,\zeta),\zeta'} v(\rx^{\zeta,\zeta'})\,d\zeta'\,d\vth'\,d\zeta\, .
$$
We suppose at first that $b\in S^{l,-2n}_{\sg,z}$. Since $\zeta'\mapsto v(\rx^{\zeta,\zeta'})\in \S(\R^n,E_z)$, we can permute the order integration $d\zeta' d\vth' \mapsto d\vth'\,d\zeta'$ in $\wt b(\rx,\vth)$. Thus, after integrations by parts in $\vth'$, we get
for any $p\in \N^*$, 
$$
\wt b(\rx,\vth)=\int_{\R^{2n}} e^{2\pi i \langle \vth,\zeta\rangle} \tau_{\rx,\zeta}\, \big(\int_{\R^n} e^{2\pi i \langle \vth',\zeta'\rangle}\langle \zeta'\rangle^{-2p}(L_{\vth'}^p b)(\psi(\rx,\zeta),\vth')\,d\vth'\big)\, \tau_{\psi(\rx,\zeta),\zeta'}\,v(\rx^{\zeta,\zeta'})\,d\zeta'\,d\zeta\, .
$$
With the estimate $\langle \rx^{\zeta,\zeta'}\rangle \geq c \langle \zeta\rangle\langle \rx\rangle^{-1} \langle \zeta'\rangle^{-1}$ for a $c>0$, we see that for any $N\in \N$, $\norm{v(\rx^{\zeta,\zeta'})}\leq c_N q_{0,N}(v)\langle \rx\rangle^{N} \langle \zeta'\rangle^N\langle \zeta\rangle^{-N}$. 
As a consequence, we get the following estimates for the integrands $b_p$ of $\wt b(\rx,\vth)$: for any $\rx,\zeta,\zeta',\vth,\vth'$, any $p\in \N^*$ and any $N\in \N^*$, $\norm{b_p(\rx,\zeta,\zeta',\vth,\vth')}\leq C_{p,N} \langle \zeta'\rangle^{N-2p} \langle \rx\rangle^{\sg|l|+N} \langle \zeta\rangle^{\sg|l|-N}\langle \vth'\rangle^{-2n}$. Taking $N$ such that $\sg|l|-N\leq -2n$ and then taking $p$ such that $N-2p \leq -2n$, we see that $(\vth',\zeta',\zeta)\mapsto b_p(\rx,\zeta,\zeta',\vth',\vth)$ is absolutely integrable and we can thus apply the following change of variable $(\zeta,\zeta',\vth')\mapsto (R_\rx(\zeta,\zeta'),\vth')$ to $\wt b(\rx,\vth)$. After reversing the integration by parts in $\vth'$ and applying the change of variable $\vth'=-\wt P_{-1,\psi(\rx,\zeta),\zeta'}^{z,\bfr}(\vth'')$, we get  
$$
\wt b(\rx,\vth)= \int_{\R^{3n}} e^{2\pi i( \langle \vth,r_{\rx,\zeta}(\zeta')\rangle+\langle \vth',\zeta'\rangle)} f_b(\rx,\zeta,\zeta',\vth')\,v(\psi(\rx,\zeta))\,d\vth'\,d\zeta'\,d\zeta\, .
$$
By Lemma \ref{amptilde} $(ii)$ and $(iii)$, Lemma \ref{htilde} $(iii)$ and $(iv)$ and Lemma \ref{Csgr} $(iii)$, we see that $f_b\in \wt \Pi^{l,w_l,w_l,m}_{\sg,\ka,\eps_1,z}$ for a $(w_l,\ka)\in \R^2_+$ and $\eps_1>0$, and the linear application $b\mapsto f_b$ is continuous on any symbol space $S^{l,m}_{\sg,z}$ into $\wt \Pi^{l,w_l,w_l,m}_{\sg,\ka,\eps_1,z}$. We have $g(\rx)= \int_{\R^n}e^{2\pi i \langle \zeta,\vth\rangle} \mu a(\rx,\vth) \, c_b(\rx,\zeta,\vth) v(\psi(\rx,\zeta)\, d\zeta\,d\vth$ and $\langle (K_{AB})_{z,\bfr},u\ox\ol v\rangle = \langle \Op_{\Ga_{0,z,\bfr}}(d_b),u\ox \ol v\rangle$ where $d_b(\rx,\zeta,\vth):=\mu a(\rx,\vth)\,c_b(\rx,\zeta,\vth)\,\tau^{-1}(\rx,\zeta)$ and 
$$
c_b(\rx,\zeta,\vth):= \int_{\R^{2n}} e^{2\pi i( \langle \vth,s_{\rx,\zeta}(\zeta')\rangle+\langle \vth',\zeta'\rangle)} f_b(\rx,\zeta,\zeta',\vth')\,d\vth'\,d\zeta'\, .
$$ 
Using now the cut-off function $(\rx,\zeta,\zeta')\mapsto \chi_{\delta,\eta}(\rx,\zeta,\zeta')$ we see that 
\begin{align*}
c_b(\rx,\zeta,\vth)&= \Pi(f_b)(\rx,\zeta,\vth) + S_{m,w_l}(f_b)(\rx,\zeta,\vth)\, .
\end{align*}
For this equality, we used the formula of Lemma \ref{Mformula} and integration by parts and  in $\vth'$ in the integral $\int_{\R^{2n}} e^{2\pi i( \langle \vth,s_{\rx,\zeta}(\zeta')\rangle+\langle \vth',\zeta'\rangle)} f_b(\rx,\zeta,\zeta',\vth')(1-\chi_{\delta,\eta}(\rx,\zeta,\zeta'))\,d\vth'\,d\zeta'\,$, which are authorized since $b\in S^{l,-2n}_{\sg,z}$ by hypothesis. In $\int_{\R^{2n}} e^{2\pi i( \langle \vth,s_{\rx,\zeta}(\zeta')\rangle+\langle \vth',\zeta'\rangle)} f_b(\rx,\zeta,\zeta',\vth')\chi_{\delta,\eta}(\rx,\zeta,\zeta')\,d\vth'\,d\zeta'\,$, we translated the $\vth'$ variable by $-L_{\rx,\zeta}(\vth)$ and permuted the order of integration $d\vth'\,d\zeta'\to d\zeta' \,d\vth'$, which is legal since  $b\in S^{l,-2n}_{\sg,z}$ and $\zeta'\mapsto \chi(\rx,\zeta,\zeta')$ is of compact support. We deduce from Lemma \ref{lemnegamp} $(ii)$ and Lemma \ref{Pi-amp} $(ii)$ that $b\mapsto \langle \Op_{\Ga_{0,z,\bfr}}(d_b),u\ox \ol v\rangle$ is continuous on $S^{l,m}_{\sg,z}$, and thus, by the density result of Lemma \ref{toposymbol}, we have the equality $\langle (K_{AB})_{z,\bfr},u\ox\ol v\rangle = \langle \Op_{\Ga_{0,z,\bfr}}(d_b),u\ox \ol v\rangle$ even when the hypothesis $b\in S^{l,-2n}_{\sg,z}$ does not hold.

Let us recall the linear map $s:a\mapsto s(a)$ given in Lemma \ref{reduction} $(ii)$ (for $\Ga=\Ga_{0,z,\bfr}$) which is such that $\Op_{\Ga_{0,z,\bfr}}(f)=\Op_{\Ga_0,z,\bfr}(s(f))$ for any $f\in \Pi_{\sg,\ka,z}^{l,w,m}$. We define $f_{a,b,\b}:= \mu a (f_b)_{\b,\varphi} \tau^{-1}$, 
$r_N:= \mu a \Pi_{R,N}(f_b)\tau^{-1}$, $s_0:= \mu a S_{m,w_l}(f_b) \tau^{-1}$.
We now consider a symbol $s_{a,b}$ such that 
$$
s_{a,b}\sim \sum_{\b\in \N^n} \tfrac{(i/2\pi)^{|\b|}}{\b!} s\big( f_{a,b,\b}\big)\, .
$$
Such a symbol exists since by Lemma \ref{phasecontrol} $(iii)$, $s(f_{a,b,\b})\in S^{l+l'-\eps'_1|\b|,m+m'-|\b|/2}_{\sg,z}$.
By Lemma \ref{Pi-amp} $(i)$, we have for any $N\geq |m|$, $u_N:=s( \mu a \Pi_N(f_b) \tau^{-1}) -s_{a,b} \in S^{l+l'-\eps'_1(N+1),m+m'-(N+1)/2}_{\sg,z}$. Thus, noting $S_0:=\Op_{\Ga_{0,z,\bfr}}(s_0)$, which is in $\Op_{\Ga_{0,z,\bfr}}(S^{-\infty}_{\sg,z})$ by Lemma \ref{lemnegamp}, $R_N:= \Op_{\Ga_{0,z,\bfr}}(r_N)$ and $U_N:=\Op_{\Ga_{0,z,\bfr}}(u_N)$  we have 
\begin{align*}
(K_{AB})_{z,\bfr}=\Op_{\Ga_{0,z,\bfr}}(d_b) &= \Op_{\Ga_{0,z,\bfr}}(s( \mu a \Pi_N(f_b) \tau^{-1}) ) + R_N + S_0 \\
&= \Op_{\Ga_{0,z,\bfr}}(s_{a,b}) + U_N + R_N +S_0\, .
\end{align*}
Lemma \ref{noyauReste} and Lemma \ref{Pi-amp} $(i)$ now implies that the kernel $U_N$ + $R_N$ (which independant of $N$) is in $\Op_{\Ga_{0,z,\bfr}}(S^{-\infty}_{\sg,z})$. As a consequence, $(K_{AB})_{z,\bfr}=\Op_{\Ga_{0,z,\bfr}}(s_{a,b} + r)$ where $r\in S^{-\infty}_{\sg,z}$ and the symbol product asymptotic formula is entailed by Lemma \ref{reduction} $(ii)$.
\end{proof}

\section{Examples} \label{exsec}

In order to be able to apply the previous results about the pseudodifferential and symbolic calculi on some concrete cases, we shall see in this section examples of exponential manifolds and associated linearizations that satisfy the hypothesis $S_\sigma$-bounded geometry. 
The Euclidean space $\R^n$ seen as exponential manifold, has its own exponential map $\psi:=\exp (x,\xi)\mapsto x+\xi$ as a $S_1$-linearization, leading to the usual pseudodifferential $SG$ calculus (if $\sg=1$) or standard (if $\sg=0$) pseudodifferential calulus on $\R^n$.
However, we can define other kinds of linearization, leading to new kind of pseudodifferential and symbol calculi, with a non-bilinear linearization map. We will see in particular that we can construct on the flat $\R^n$, a family of $S_\sg$-linearizations that generalize the case of the flat euclidian geometry, and we obtain a extension of the normal ($\la=0$) and antinormal ($\la=1$) quantization on $\R^n$. 

We will also prove that the 2-dimensional hyperbolic space, which is a Cartan--Hadamard manifold (and thus an exponential Riemannian manifold) has $S_1$-bounded geometry. This allows to define a global Fourier transform, Schwartz spaces $\S(\HH)$, $\S(T^*\HH)$, $\S(T\HH)$, $\B(\HH)$ and the space of symbols $S_1^{l,m}(T^*\HH)$. As a consequence, we can define in an intrinsic way a global complete pseudodifferential calculus on $\HH$, if one chose a fixed $S_\sg$-linearization $\psi$ on $T\HH$. There are many possible linearizations, for instance one can take $\psi$ such that in a frame $(z,\bfr)$ $\psi_{z}^\bfr$ is the standard linearization $x+\xi$ of $\R^{n}$.

\subsection{A family of $S_\sg$-linearizations on the euclidean space}

Recall that $G_\sigma^\times(\R^n)$ ($0\leq \sigma \leq 1$) is defined as the subgroup of diffeomorphisms $s$ on $\R^n$ such that for any $n$-multi-index $\a\neq 0$, there are $C_\a$, $C'_\a >0$, such that for any $\rx\in \R^n$, $\norm{\del^\a s (\rx)} \leq C_\a \langle \rx \rangle^{\sigma(1-|\a|)}$ and $\norm{\del^\a s^{-1} (\rx)} \leq C'_\a \langle \rx \rangle^{\sigma(1-|\a|)}$. $G_\sigma^\times (\R^n)$ contains $GL_n(\R)$ and the translations $T_v:=w\mapsto v+w$.

We fix $\eta\in ]0,1[$ such that for any matrix $A\in \M_n(\R)$ such that $\norm{A}_1\leq  \eta$, we have $\det (I_n+A) \geq \half$, where $\norm{A}_1:= \max_{i,j}  |A_{i,j}|$. Taking now $h\in G_0(\R^n,\R^n)$ such that for any $1\leq i,j\leq n$, $|\del_j h^i|\leq \eta/16$, and 
$g(x):= h(x) - h(0) - dh_0 (x)$
we see that $s:=\Id +g$ is a diffeomorphism on $\R^n$ which belongs to $G_0^\times (\R^n)$, satisfying $s(0)=0$ and $ds_0 =\Id$.
 
\noindent We set, for $\sg\in [0,1]$,
$$
\psi(x,\xi):=x+\xi + \langle x\rangle^\sg g( \tfrac{\xi}{ \langle x\rangle^\sg} )= x+ \langle x\rangle^\sg s (  \tfrac{\xi}{ \langle x\rangle^\sg} ).
$$
We obtain the following

\begin{prop} $(\R^n,+,d\la,\psi)$ has a $S_\sg$-bounded geometry and satisfies $(C_\sg)$ (see Definition \ref{Csigma}).
\end{prop}
\begin{proof} A computation shows that $\psi \in H_\sg(\R^n)$ and $\psi(\rx,\zeta)=\O(\langle \rx \rangle\langle \xi\rangle)$. We have $\ol \psi(x,y)= \langle x\rangle^\sg s^{-1}(\tfrac{y-x}{\langle x\rangle^\sg})$, and thus $\ol \psi \in \O_M(\R^{2n},\R^n)$. Noting $\wh g:=g\circ (g+\Id)^{-1}\circ -\Id \in G_0(\R^n)$, we also have 
\begin{align*}
\Ups_{1,T}(x,\xi)&= \xi+\langle x\rangle^\sg g(\tfrac{\xi}{\langle x\rangle^\sg}) + \langle \psi(x,\xi)\rangle^\sg \wh g \big(\langle \psi(x,\xi)\rangle^{-\sg}\langle x\rangle^\sg s(\tfrac{\xi}{\langle x\rangle^\sg})\big)\\
&= (\Id + V_{x,\xi} + W_{x,\xi} ) (\xi) 
\end{align*}
where $V_{x,\xi}:= [\int_0^1 \del_j v_x^i (t\xi)dt]_{i,j}$, $W_{x,\xi}:= [\int_0^1 \del_j w_{x,\xi}^i (t\xi)dt]_{i,j}$, and 
$v_x:= M_x \circ g\circ M_x^{-1}$, $w_{x,\xi} := M_{\psi(x,\xi)}\circ  \wh g \circ M_{\psi(x,\xi)}^{-1}\circ M_x \circ s \circ M_x^{-1}$, $M_x$ being the multiplication by $\langle x\rangle^\sg$. We get $dv_x = dg\circ M_x^{-1}$ and $dw_{x,\xi}=d\wh g\circ (M_{\psi(x,\xi)}^{-1}\circ M_x\circ s\circ M_x^{-1}) \, ds\circ M_x^{-1}$.   
and thus, after computations we check that $V_{x,\xi}$ and $W_{x,\xi}$ are in $E_\sg^0$. Moreover, we have $\norm{V_{x,\xi}}_1 \leq \eta/2$ and $\norm{W_{x,\xi}}_1\leq \eta/2$, which proves that $P_{x,\xi}:=\Id + V_{x,\xi} + W_{x,\xi}$ is invertible with $\det P_{x,\xi}\geq \half$. As a consequence its inverse $P_{x,\xi}^{-1}=(\det P_{x,\xi})^{-1} \, ^t \cof(P_{x,\xi})$ is also in $E_\sg^0$. We deduce then that $(\R^n,+,d\la,\psi)$ has a $S_\sg$-bounded geometry. With $r(x,\xi,\xi')=-\ol\psi(x,\psi(\psi(x,-\xi),-\xi'))$, we get
$$
 r(x,\xi,\xi') =-\langle x\rangle^\sg s^{-1}\big( s(\tfrac{-\xi}{\langle x\rangle^\sg})+ \tfrac{\langle \psi(x,-\xi)\rangle^\sg}{\langle x\rangle^\sg} s(\tfrac{-\xi'}{\langle \psi(x,-\xi)\rangle^\sg}) \big)\, .
$$
so that $(dr_{x,\xi})_{\xi'} = (d s^{-1}\circ w )\,( ds \circ u)$ where $w(x,\xi,\xi'):= s(\tfrac{-\xi}{\langle x\rangle^\sg})+ v(x,\xi,\xi')$, $v(x,\xi,\xi):= \tfrac{\langle \psi(x,-\xi)\rangle^\sg}{\langle x\rangle^\sg} s(\tfrac{-\xi'}{\langle \psi(x,-\xi)\rangle^\sg})$, $u(x,\xi,\xi'):= -\tfrac{\xi'}{\langle \psi(x,-\xi)\rangle^\sg}$. 
We check that $v$ satisfies 
$$
\del^{(\mu,\ga)} v =\O\big(\langle \psi(x,-\xi)\rangle^{-\sg |\ga|}\langle x\rangle^{-\sg( |\mu|+1)}\langle \zeta\rangle^{\ka_1|\mu|}\langle \zeta'\rangle^{|\mu|+1}\big).
$$  
It follows from Peetre's inequality that for any $\eps \in [0,1]$ and $x,y\in \R^n$, $\langle x + y\rangle \geq 2^{-\eps/2} \tfrac{\langle x\rangle^\eps}{\langle y\rangle^\eps}$, which implies that $\langle \psi(x,-\xi)\rangle^\sg =\O\big( \langle x\rangle ^{-\sg \eps} \langle \xi\rangle^{\sg\eps}\big)$. As a consequence we get the estimates
\begin{align*}
&\del^{(\mu,\ga)} w = \O\big(\langle x\rangle^{-\sg (1+|\mu|+\eps|\ga|)}\langle \zeta\rangle^{\ka_1|\mu|+ \eps|\ga| + \delta_{\ga,0}}\langle \zeta'\rangle^{|\mu|+1}\big)\, , \\ 
&\del^{(\mu,\ga)} u = \O\big(\langle x\rangle^{-\sg (|\mu|+\eps|\ga|)}\langle \zeta\rangle^{\ka_1|\mu|+ \eps|\ga|}\langle \zeta'\rangle^{1-|\ga|}\big)\, .
\end{align*}
We deduce from this that $(C_\sg)$ is satisfied.
\end{proof}

We also check that the hypothesis $(H_V)$ of section \ref{linkstd} is satisfied so that 
the previous pseudodifferential calculus (for $\la\in \set{0,1}$) is then valid on  $(\R^n,+,d\la,\psi)$, and proves in particular the space of operators of the form
$$
A (v) (x) =\int_{\R^{2n}} e^{2\pi i \langle \th,\xi\rangle } a(x,\th) v(\psi(x,-\xi)) \, d\xi \,d\th = \int_{\R^{2n}} e^{-2\pi i \langle \th,\ol\psi_{x}(y)\rangle } a(x,\th) v(y) |J(\ol\psi_x)|(y)\, dy \,d\th
$$  
where $a\in S^{\infty}_{\sg}(\R^{2n})$, is equal to the standard algebra of algebra of pseudodifferential operators $\R^n$. However, since $(C_\sg)$ is satisfied, we have now at our disposal a new symbol composition formula given by Theorem \ref{compo}, adapted to the new linearization $\psi$.

\subsection{$S_1$-geometry of the Hyperbolic plane}

The (hyperboloid model of the) 2-dimensional hyperbolic space is defined as the submanifold $\HH:=\set{x=(x_1,x_2,x_3)\in \R^3 \ : \ x_1^2+x_2^2-x_3^2=-1 \ \text{and}\ x_3>0}$ of the $(2,1)$-Minkowski space $\R^{2,1}$ with the bilinear symmetric form $\langle v,w\rangle_{2,1}=v_1 w_1+v_2 w_2-v_3 w_3$. The induced metric on $\HH$: $ds^2= (dx_1)^2+(dx_2)^2-(dx_3)^2$ is Riemannian and it is known that $\HH$ is a symmetric Cartan--Hadamard manifold with constant negative sectional curvature (equal to $-1$). The map $\varphi : \R^2\to \HH$ given by 
$$
\varphi(x,y):= (\sinh x,\,\cosh x \sinh y,\,\cosh x \cosh y )
$$
is a diffeomorphism with inverse $\varphi^{-1}(x_1,x_2,x_3)=(\argsh x_1,\,\argsh(\tfrac{x_2}{\cosh(\argsh x_1 )}))$. 
 As a consequence we can construct another model of the hyperbolic space, noted $R^2$ with domain $\R^2$ and metric obtained by pulling back the metric on $\HH$ onto $\R^2$. A computation shows that this metric is $ds^2:= (dx)^2+ \cosh^2 x \,(dy)^2$. We will note $\norm{\cdot}_{p}$ the norm on $T_p R^2\simeq \R^2$ given by this metric, where $p$ is a point in $\R^2$, and $\norm{\cdot}$ is the Euclidian norm. The geodesic equation on $R^2$ leads to the following system of ordinary differential equations:
\begin{align}
&x'' - \cosh x\,\sinh x\,(y')^2=0 \, \nonumber,\\
&y'' + 2 \tanh x \,x'\,y' = 0 \, \label{geodeq}. 
\end{align}
For each $p=(x,y)\in \R^2$ and $v\in \R^2$ such that $\norm{v}_p=1$ there exists an unique solution on $\R$ $\ga_{p,v}=(x(t),y(t))$ of (\ref{geodeq}) such that $\ga_{p,v}(0)=p$ and $\ga'_{p,v}(0)=v$. 

At each point $p=(x,y)\in \R^2$, we can define the ellipse of unit vectors centered at $0$ in $T_p R^2\simeq \R^2$ with equation $X^2+(\cosh^2 x) \, Y^2=1$. The polar equation of this ellipse is $e_p(\th)$ where
$$
e_{p}(\th):= \tfrac{1}{\sqrt{1+\sinh^2 x \, \sin^2 \th}\, }\, .
$$
Thus, any tangent vector $v\in T_p R^2$ with decompostion $v=\norm{v}(\cos \th,\sin \th)$ also admits the following polar decomposition $v=\norm{v}_p(\cos_p \th,\sin_p \th)$ where $\cos_{p} \th :=e_{p}(\th)\,\cos \th$ and $\sin_{p} \th:=e_{p}(\th) \,\sin \th$. Remark that $e_p$, $\cos_p$, $\sin_p$ and $\norm{\cdot}_p$ are in fact independant of the second coordinate $y$ of $p$. We shall therefore also use the notations $e_x:=e_{(x,y)}$ and similarly for $\cos_x$, $\sin_x$ and $\norm{\cdot}_x$. Note that for any vector $v:=\norm{v}(\cos \th,\sin \th)$, we have $\norm{v}_x = \norm{v} / e_x(\th)$.

If $p\in \HH$ and $v\in \R^{2,1}$ are such that $\langle p,v\rangle_{2,1} =0$ and $\langle v,v\rangle_{2,1}=1$, then the unique geodesic $\a_{p,v}$ on $\HH$ such that $\a_{p,v}(0)=p$ and $\a'_{p,v}(0)=v$ is $\a_{p,v}(t)=\cosh t \, p + \sinh t\, v$ (see for instance \cite[p.195]{Jost}). As a consequence, the geodesics $\ga_{p,v}$ on the $R^2$ hyperbolic space can be obtained by pushing forward the $\a_{p,v}$ geodesics with the diffeomorphic isometry $\varphi$. We check after tedious calculations that for any given $p=(x,y)\in \R^2$ and $\th\in \R$, the following curve
\begin{align}
&\ga_{p,\th}^1(t)= \argsh\big( \cosh t \, \sinh x + \sinh t\, \cosh x \, \cos_x \th \big) \, \nonumber ,\\
&\ga_{p,\th}^2(t)=  \argsh \big ( \tfrac{\cosh t \, \cosh x\, \sinh y + \sinh t\,( \sinh x \, \sinh y\, \cos_x \th +\cosh x \, \cosh y\, \sin_x \th  )}{\cosh\big( \argsh ( \cosh t \, \sinh x + \sinh t\, \cosh x \, \cos_x \th ) \big) } \big)\, \label{exponential} ,
\end{align}
where $t\in \R$, is the unique maximal solution of the geodesic system (\ref{geodeq}) satisfying the initial conditions: $\ga_{p,\th}(0)= p$ and $\ga_{p,\th}'(0) = (\cos_x(\th),\sin_x(\th))$. An explicit formula for the exponential map at any point can therefore be obtained, since we have $\exp_p(v)=\ga_{p,\th}(\norm{v}_x)$ where $v\in T_p R^2-\set{0}$ and $\th \in \R$ such that $v=\norm{v}(\cos \th,\sin \th)$. The main interest of this hyperbolic model with domain equal to $\R^2$ is that it is possible to find explicitely the logarithmic map (the inverse of the exponential map) at any point. We find, after an elementary but long computation, the following inverse, for any $p=(x,y)$ and $p'=(x',y')\in \R^2$,
\begin{align}
&\exp_p^{-1}(p') = \tfrac{\argch f_p(p')}{\sqrt{(f_p(p'))^2-1}}\,\twobyone{-g_p(p')}{\cosh x' \,\sech x\,\sinh(y'-y) } \, \label{logarithm} , \\
&f_p(p'):=\cosh(x')\cosh(y'-y)\cosh(x)-\sinh(x')\sinh(x)\, \nonumber ,\\
&g_p(p'):= \cosh(x')\cosh(y'-y)\sinh(x)-\sinh(x')\cosh(x)\, \nonumber .
\end{align} 
We have $\norm{\exp_p^{-1}(p')}_p=\argch f_p(p')$ which is the geodesic distance between two arbitrary points $p, p'$ in the $R^2$ hyperbolic model. The goal of this section is to prove the following result.

\begin{thm}
\label{R2mainthm}
$\HH$ has a $S_1$-bounded geometry. 
\end{thm}

We note $\R^2_C:=\R^2\backslash]-\infty,0]\times \set{0}$ and $\R^2_P:= ]0,+\infty[\times ]-\pi,\pi[$. For any $x\in \R$, the map $\chi_x:\R^{2}_C\to \R^2_P$ given by $\chi_x(v_1,v_2):=(\norm{v}_x, \arctan (v_1,v_2))$ where $\arctan(v_1,v_2)$ is the unique element $\th$ of $]-\pi,\pi[$ such that $v_1+i v_2 = \norm{v} \exp(i \th)$, is a diffeomorphism with inverse $\chi_x^{-1}(r,\th)= (r \cos_x \th,r\sin_x \th).$

\begin{lem}
\label{passage}
Let $x\in \R$ and $f\in C^\infty(\R^2,\R)$ such that $ f\circ \chi_x^{-1} \in C^\infty(\R^2_P,\R)$ satisfies for any $(\a,\b)\in \N^2\backslash\set{(0,0)}$, and $(r,\th)\in \R_P^2$, $|\del^{\a,\b} f\circ \chi_x^{-1}(r,\th) | \leq C_{\a,\b} \langle r\rangle^{1-\a}$
where $C_{\a,\b}>0$. Then $f\in G_{1}(\R^2,\R)$.
\end{lem}
\begin{proof} By Theorem \ref{FaaCS}, for any $(\a,\b)\in \N^2\backslash\set{(0,0)}$, $
\del^{\a,\b}f=\sum_{1\leq |(\a',\b')|\leq |(\a,\b)|} (\del^{\a',\b'}f\circ \chi_x^{-1})\circ \chi_x \  P_{\a,\b,\a',\b'}(\chi_x)
$ on $\R^2_C$, where $P_{\a,\b,\a',\b'}(\chi_x)$ is a linear combination of functions of the form $\prod_{j=1}^s (\del^{l^j} \chi_x)^{k^j}$ where $s\in \set{1,\cdots ,\a+\b}$.  
The $k^j$ and $l^j$ are $2$-multi-indices (for $1\leq j\leq s$) such that $|k^j|>0$, $\sum_{j=1}^s k^j = (\a',\b')$ and $\sum_{j=1}^s |k^j| l^j= (\a,\b)$. By definition, $\chi_x(v)=(\chi_x^1(v),\chi_x^2(v))=(\norm{v}_x,\arctan(v_1,v_2))$. It is straightforward to check that for any 2-multi-index $\nu$, $|\del^\nu \chi_x^1 (v) |\leq C_\nu \langle v \rangle^{1-|\nu|}$ and $|\del^\nu \chi_x^2 (v)| \leq C'_\nu \langle v\rangle^{-|\nu|}$ on $\R^2_C$.  
As a consequence, for each $\a,\b,\a',\b'$ with $1\leq \a'+\b'\leq \a+\b$ there exists $C_{\a,\b,\a',\b'}>0$ such that for any $v\in \R^2_C$,
\begin{equation*}
|P_{\a,\b,\a',\b'}(\chi_x) (v)|\leq C_{\a,\b,\a',\b'}  \langle v \rangle^{\a' -(\a+\b)}\, . 
\end{equation*}
Moreover, by hypothesis, there is $C_{\a',\b'}>0$ such that for any $v\in \R^2_C$, $|(\del^{\a',\b'}f\circ \chi_x^{-1})\circ \chi_x(v)|\leq C_{\a',\b'} \langle v \rangle ^{1-\a'}$. This gives $f\in G_{1}(\R^2_C,\R)$. The extension to $G_1(\R^2,\R)$ is a direct consequence of the smoothness of $f$ on $\R^2$ and the fact that $\R^2_C$ is dense in $\R^2$. 
\end{proof}

We shall use the following proposition, which gives a formal expression of the successive derivatives of the inverse (and its real powers) of a smooth function.

\begin{prop}
\label{inverse} Let $s>0$ be given. For any nonzero $n$-multi-index ($n\in \N^*$) $\a$, 
there exist a finite nonempty set $J_\a$, nonzero real numbers $(\la_{s,\a,p})_{p\in J_{\a}}$ 
and $n$-multi-indices $\b^{\a,p,j}$ (with $p\in J_\a$, $1\leq j\leq |\a|$) such that 

\noindent - for any $p\in J_\a$, $\sum_{1\leq j\leq |\a|} \b^{\a,p,j}=\a$,

\noindent - for any smooth function $f\in C^\infty(\R^n,\R_+^*)$,
$$
\del^\a \tfrac{1}{f^s}  = \tfrac{1}{f^{|\a|+s}}\, \sum_{p\in J_\a} \la_{s,\a,p}\, \prod_{j=1}^{|\a|} \del^{\b^{\a,p,j}} f\, .
$$ 
\end{prop}
\begin{proof} The result is true for the case $|\a|=1$. Suppose then that the result holds for any $n$-multi-index $\a$ such that $|\a|=k$, where $k\in \N^*$ and let $\a'$ be a $n$-multi-index such that $|\a'|=k+1$. Let $i$ be the smallest element of $\set{1,\cdots,n}$ such that $\a'_i\geq 1$, and set 
$\a:=(\a'_1,\cdots,\a'_{i-1},\a'_i-1,\a'_{i+1},\cdots,\a'_n)$. Thus for any $f\in C^\infty(\R^n,\R^*_+)$, $\del^{\a'}\tfrac{1}{f^s}=\del_i \del^\a \tfrac{1}{f^s}$. Since $|\a|=k$, there is exist a finite nonempty set $J_\a$, nonzero real numbers $(\la_{s,\a,p})_{p\in J_{\a}}$ 
and $n$-multi-indices $\b^{\a,p,j}$ (with $p\in J_\a$, $1\leq j\leq |\a|$) such that 
for any $p\in J_\a$, $\sum_{1\leq j\leq |\a|} \b^{\a,p,j}=\a$,
and such that for any $f\in C^\infty(\R^n,\R^*_+)$,
$\del^\a \tfrac{1}{f^s}  = \tfrac{1}{f^{|\a|+s}}\, \sum_{p\in J_\a} \la_{s,\a,p}\, \prod_{j=1}^{|\a|} \del^{\b^{\a,p,j}} f$. 
As a consequence, with the formula $\del_i \prod_{j=1}^{|\a|} g_j = \sum_{q=1}^{|\a|} \prod_{j=1}^{|\a|} \del^{\delta_{q,j}e_i} g_{j}$, we obtain for any $f\in C^\infty(\R^n,\R^*_+)$, 
$$
\del^{\a'} \tfrac{1}{f^s} = \tfrac{1}{f^{|\a'|+s}}\big( \sum_{p\in J_{\a}}-(|\a|+s)\la_{s,\a,p} (\prod_{j=1}^{|\a|} \del^{\b^{\a,p,j}} f)\del_i f + \sum_{(p,q)\in J_\a\times \N_{|\a|}}\la_{s,\a,p} (\prod_{j=1}^{|\a|} \del^{\delta_{q,j}e_i + \b^{\a,p,j}} f)f \big)\, .
$$
Thus, if we take $J_{\a'}=J_{\a}\coprod (J_{\a}\times \N_{|\a|})$, $\la_{s,\a',\wt p}:= -(s+|\a|)\la_{s,\a,p}$ if $\wt p =p \in J_{\a}$, $\la_{s,\a',\wt p}:= \la_{s,\a,p}$ if $\wt p = (p,q)\in J_{\a}\times \N_{|\a|}$, $\b^{\a',\wt p,j}:= \b^{\a,p,j}$ if $\wt p =p\in J_{\a}$ and $1\leq j \leq |\a|$, $\b^{\a',\wt p,j}:= e_i$ if $\wt p =p\in J_{\a}$ and $j= |\a|+1=|\a'|$, $\b^{\a',\wt p,j}:= \delta_{q,j}e_i+\b^{\a,p,j}$ if $\wt p =(p,q)\in J_{\a}\times \N_{|\a|}$ and $1\leq j \leq |\a|$ and $\b^{\a',\wt p,j}:= 0$ if $\wt p =(p,q)\in J_{\a}\times \N_{|\a|}$ and $ j = |\a|+1=|\a'|$, the result now holds for $\a'$.
\end{proof}

In the following we set the convention $J_{0}:=\set{1}$, $\la_{s,0,1}:=1$ and $\prod_{j=1}^0:=1$, so that the formula giving $\del^\a \tfrac{1}{f^s}$ in the previous lemma is still valid when $\a=0$. When $s\in \N^*$, the result is also valid for complex valued nowhere zero smooth functions.

We note $H_P$ the space of $C^\infty(\R^2_P,\R)$ functions of the form $(r,\th)\mapsto a(\th) \cosh r + b(\th) \sinh r$ where $a,b \in \mathcal{B}(\R)$, and $A_{P,k}$ the space of functions $f\in C^\infty(\R^2_P,\R)$ such that for any 2-multi-index $(\a,\b)$ with $\a\leq k \in \N$, there is $C_{\a,\b}>0$ such that for any $(r,\th)\in \R^2_P$, $|\del^{\a,\b} f (r,\th) |\leq C_{\a,\b} \langle r\rangle^{k-\a}$, and also such that for any 2-multi-index $(\a,\b)$ with $\a\geq k+1$, there is $C'_{\a,\b}>0$ such that for any $(r,\th)\in \R^2_P$,  $|\del^{\a,\b} f (r,\th)|\leq C'_{\a,\b} e^{-2r}$. Clearly, $A_{P,k}\subset S_{P,k}$ where $S_{P,k}$ is the space of functions $f\in C^\infty(\R^2_P,\R)$ such that for any 2-multi-index $(\a,\b)$, there is $C_{\a,\b}>0$ such that for any $(r,\th)\in \R^2_P$, $|\del^{\a,\b} f (r,\th) |\leq C_{\a,\b} \langle r\rangle^{k-\a}$. By Leibniz rule, $S_{P,k}S_{P,k'}\subseteq S_{P,k+k'}$. We note $N_P$ the space of functions $f\in C^\infty(\R^2_P,\R)$ such that
for any 2-multi-index $(\a,\b)$ there is $C_{\a,\b}>0$ such that for any $(r,\th)\in \R^2_P$,  $|\del^{\a,\b} f (r,\th)|\leq C_{\a,\b} e^{-2r}$. If $r_0>0$ we define the spaces $H_{P,r_0}$, $A_{P,k,r_0}$, $S_{P,k,r_0}$ and $N_{P,r_0}$ exactly as before, except that we now replace the domain $\R^2_P$ by $\R^2_{P,r_0}:=]r_0,+\infty[\times ]-\pi,\pi[$.

\begin{lem}
\label{techno}
Let $f,g,h,w\in H_{P,r_0}$ where $r_0>0$, such that there is $\eps>0$, $C>1$ such that for any $(r,\th)\in \R^2_{P,r_0}$, $f\geq C$, $f \geq \eps\, e^{r}$ and $h^2+g^2 \geq \eps\, e^{2r}$.

\noindent (i) The functions $\tfrac{w}{(h^2+g^2)^{3/2}}$, $\tfrac{w}{(f^2-1)^{3/2}}$ and any function of the form $(r,\th)\mapsto \tfrac{\sum_{k=-4}^{4} b_k(\th)e^{kr}}{((h^2+g^2)(1+h^2+g^2))^{3/2}}$, where $b_k\in \B(\R)$, are in $N_{P,r_0}$.

\noindent (ii) The functions $\argch \sqrt{1+h^2+g^2}$ and $\argch f$ are in $A_{P,1,r_0}$.

\noindent (iii) The functions $\tfrac{w}{\sqrt{h^2+g^2}}$ and $\tfrac{w}{\sqrt{f^2-1}}$ are in  $A_{P,0,r_0}$.

\end{lem}

\begin{proof}
$(i)$ We give a proof for $\tfrac{w}{(h^2+g^2)^{3/2}}$. The other cases are similar. By Proposition \ref{inverse} and Leibniz rule, we have for any 2-multi-index $\nu$, 
$$
\del^{\nu}\tfrac{w}{(h^2+g^2)^{3/2}} = \sum_{\nu'\leq \nu } \tbinom{\nu}{\nu'} \tfrac{\del^{\nu-\nu'} w }{(h^2+g^2)^{3/2+|\nu'|}} \sum_{p\in J_{\nu'}} \la_{3/2,\nu',p} \prod_{j=1}^{|\nu'|} \del^{\b^{\nu',p,j}} (h^2+g^2)\, .
$$
Note that we have for any 2-multi-index $\nu$, $\del^\nu (h^2+g^2) = \O(e^{2r})$ and $\del^\nu w = \O(e^r)$. The result follows.

\noindent $(ii)$ By $(i)$, since $\del_r^2 \argch \sqrt{1+h^2+g^2}$ is of the form $(r,\th)\mapsto \tfrac{\sum_{k=-4}^{4} b_k(\th)e^{kr}}{((h^2+g^2)(1+h^2+g^2))^{3/2}}$ where $b_k\in \B(\R)$, and $\del_r^2 \argch f$ is of the form $\tfrac{w}{(f^2-1)^{3/2}}$ where $w \in H_{P,r_0}$, we only need to check that for $0\leq \a \leq 1$, and $\b\in \N, \del^{\a,\b} \argch \sqrt{1+h^2+g^2} = \O(\langle r \rangle^{1-\a})$ and  $\del^{\a,\b} \argch f= \O(\langle r \rangle^{1-\a})$. Since $\del_{r} \argch \sqrt{1+h^2+g^2} = \tfrac{(\del_r h)h+(\del_r g)g}{\sqrt{(h^2+g^2)(1+h^2+g^2)}}$, $\del_{r} \argch f = \tfrac{\del_r f}{\sqrt{f^2-1}}$, $\del_{\th} \argch \sqrt{1+h^2+g^2} = \tfrac{(\del_\th h)h+(\del_\th g)g}{\sqrt{(h^2+g^2)(1+h^2+g^2)}}$ and $\del_{\th} \argch f = \tfrac{\del_\th f}{\sqrt{f^2-1}}$, the result follows from an application of Proposition \ref{inverse}.

\noindent $(iii)$ By $(i)$, since $\del_r \tfrac{w}{\sqrt{h^2+g^2}}$ is of the form $\tfrac{w_1}{(h^2+g^2)^{3/2}}$ where $w_1 \in H_{P,r_0}$, and $\del_r \tfrac{w}{\sqrt{f^2-1}}$ is of the form $\tfrac{w_2}{(f^2-1)^{3/2}}$ where $w_2 \in H_{P,r_0}$, we only need to check that for $\b\in \N, \del^{0,\b}\tfrac{w}{\sqrt{h^2+g^2}} = \O(1)$ and  $\del^{0,\b} \tfrac{w}{\sqrt{f^2-1}}= \O(1)$. This is a direct consequence of Proposition \ref{inverse}.
\end{proof}

\begin{proof}[Proof of Theorem \ref{R2mainthm}] By Lemma \ref{Ssigma} $(iii)$ and Proposition \ref{ssigmDiff}, it is sufficient to prove that for any $p:=(x,y)\in \R^2\backslash\set{0}$, $\exp_{p}^{-1}\circ \exp_0$ and $\exp_{0}^{-1}\circ \exp_p$ are in $G_1(\R^2)$. A computation based on (\ref{exponential}) and (\ref{logarithm}) shows that on $\R^2_P$,
\begin{align*}
&\exp_{p}^{-1}\circ \exp_0 \circ \chi_{0}^{-1}= (\argch f )\big (\tfrac{w_1}{\sqrt{f^2-1}}  ,  \tfrac{w_2}{\sqrt{f^2-1}}  \big )\, ,\\
&\exp_{0}^{-1}\circ \exp_{p}\circ \chi_x^{-1} = (\argch \sqrt{1+h^2+g^2} )\big(\tfrac{h}{\sqrt{h^2+g^2}} , \tfrac{g}{\sqrt{h^2+g^2}} \big)\, ,
\end{align*}
\noindent where 
\begin{align*}
&f(r,\th):= \cosh r \cosh y \cosh x - \sinh r (\sinh x \cos \th + \sinh y \cosh x \sin \th)\, ,\\
&w_1(r,\th):= -\cosh r \cosh y \sinh x + \sinh r (\cosh x \cos \th + \sinh y \sinh x \sin \th)\, ,\\
&w_2(r,\th):= - \cosh r \sinh y \sech x +\sinh r \sin \th \cosh y \sech x \, ,\\
&h(r,\th):= \cosh r \sinh x + \sinh r \cosh x \cos_x \th \, , \\
&g(r,\th):= \cosh r \cosh x \sinh y + \sinh r (\sinh x \sinh y \cos_x \th + \cosh x \cosh y \sin_x \th)\, .
\end{align*}
All these functions belong to $H_P$ and $f\geq 1$. Note that $f(r,\th)=1$ if any only if $\exp_0\chi_0^{-1}(r,\th)=p$, in which case $\exp_{p}^{-1}\circ \exp_0 \circ \chi_{0}^{-1}(r,\th)=0$,  so that $\exp_{p}^{-1}\circ \exp_0 \circ \chi_{0}^{-1}$ is well defined as a smooth function on the whole $\R^2_{P}$. The same argument holds for $\exp_{0}^{-1}\circ \exp_{p}\circ \chi_x^{-1}$. We check that 
$$
\half(\cosh x \cosh y - \sqrt{\cosh^2 x \cosh^2 y -1}) e^{r}\leq  f(r,\th)\leq \cosh r \ e^{\argch (\cosh x \cosh y)} 
$$
so that by defining $r_0:= \log 2/\eps$ where $0<\eps< \min\set {1, \half(\cosh x \cosh y - \sqrt{\cosh^2 x \cosh^2 y -1}) }$ we have for any $(r,\th)\in \R^2_{P,r_0}$, $f(r,\th)\geq \eps e^r\geq 2$. Note also that for any $v\in \R^2_C$, we have $\argch f(\chi_0(v)) = \norm{\exp_p^{-1}\circ \exp_0 (v)}_p$ and 
$$
\argch \sqrt{1+h^2(\chi_x(v))+g^2(\chi_x(v))}= \norm{\exp_0^{-1}\circ \exp_p(v)}_0\, .
$$
The first equality entails (since $\exp_p^{-1}\circ \exp_0 (\R^2_C)$ is a dense open subset of $\R^2$) that for any $v$ in $\R^2$, $\cosh \norm{v}_p \leq \cosh \norm{\exp_0^{-1}\circ \exp_p (v)}_0 e^{\argch (\cosh x\cosh y)}$. We then obtain for any $(r,\th)\in \R^2_P$, $\sqrt{1+h^2+g^2}\geq \cosh r \ e^{-\argch (\cosh x \cosh y)}$. In particular, defining 
$$r'_0:=\argch (\sqrt{2}\exp (\argch (\cosh x \cosh y))),$$ we get for any $r\geq r'_0$, the following estimate $h^2+g^2\geq \tfrac{1}{8} e^{-2\argch(\cosh x\cosh y)} e^{2r} $. If we now apply Lemma \ref{techno} for the space $H_{P,r''_0}$ where $r''_0:= \max\set{r_0,r'_0}$, we see that $\exp_{p}^{-1}\circ \exp_0 \circ \chi_{0}^{-1}$ and $\exp_{0}^{-1}\circ \exp_{p}\circ \chi_x^{-1}$ are in $S_{P,1}$. The result then follows from Lemma \ref{passage}.
\end{proof}

\section{Conclusion}

We have seen in this paper certain hypothesis on the geometry ($S_\sg$ or $\O_M$-bounded geometry) of a manifold with linearization that allows a coordinate free definition of most of the topological vector spaces that are needed for Fourier analysis and global complete symbol calculus with uniform and decaying control over the $x$ variable. Given a linearization on the manifold with some properties of control at infinity, we constructed symbol maps and $\la$-quantization, explicit Moyal star-products on the cotangent bundle, and classes of pseudodifferential operators. We proved a stability under composition result, and an associated symbol product asymptotic formula under a hypothesis $(C_\sg)$ of control at infinity of the linearization. The calculus presented here is a generalization of the standard and $SG$ symbol calculi over the Euclidean space $\R^n$ and can be applied to the hyperbolic 2-space since, as proven in section 5.2, it has a $S_1$-bounded geometry.  $L^2$-continuity of pseudodifferential operators of order $(0,0)$ has been established in section \ref{linkstd} under the hypothesis $(H_V)$. We do not know however if this result still holds without this hypothesis.

The full analysis of the obtained Moyal product on $\S(T^*M)$ and spectral properties of pseudodifferential operators in $\Psi_{\sg}^{l,m}$ remain to be studied. Extension and connection of the symbol calculus presented here could be made with, for instance, noncommutative geometry (Gayral, Gracia-Bond\'{i}a, Iochum, Sch\"{u}cker and V\'{a}rilly \cite{Gayral}), the magnetic Moyal calculus (Iftimie, Mantoiu and Purice \cite{Iftimie}), spectral asymptotics (Shubin \cite{Shubin4}), essential self-adjointness (Braverman, Milatovich and Shubin \cite{Braverman}), Fourier integral operators (Coriasco \cite{Coriasco}, Ruzhansky and Sugimoto \cite{Ruzhansky3,Ruzhansky2}), Wiener type calculus (Sj\"{o}strand \cite{Sjostrand,Sjostrand2}), generalized operators (Garetto \cite{Garetto1}), Gelfand--Shilov spaces (Cappiello, Gramchev and Rodino \cite{Cappiello}), regularized traces (Paycha \cite{Paycha}), and white noise analysis for an infinite dimensional Moyal product (L\'{e}andre \cite{Leandre} and Dito and L\'{e}andre \cite{Dito}).   ¨

\section*{Acknowledgements} Thanks are due to my Ph.D. advisor Bruno Iochum, Thomas Krajewski, Thomas Sch\"{u}cker, Sylvie Paycha, Gerd Grubb, Ryszard Nest, Moulay-Tahar Benameur, Ludwik D\c{a}browski, Michael Puschnigg, Victor Gayral, Mathieu Beau and Baptiste Savoie for helpful discussions and observations.

\end{document}